\documentclass[12pt,a4paper]{article}
\usepackage{bbm}
\usepackage{mathrsfs}
\usepackage{graphicx}
\usepackage{amsmath}
\usepackage{amssymb}
\usepackage{amsfonts}
\usepackage{enumerate}
\usepackage{theorem}

\makeatletter
\def\tank#1{\protected@xdef\@thanks{\@thanks
        \protect\footnotetext[0]{#1}}}
\def\bigfoot{

    \@footnotetext}
\makeatother

\newcommand{\ea}{\end{array}}
\newtheorem{theorem}{Theorem}[section]
\newtheorem{proposition}{Proposition}[section]

\newtheorem{lemma}{Lemma}[section]
\newtheorem{definition}{Definition}[section]
\newtheorem{Rem}{Remark}[section]

{\theorembodyfont{\rmfamily}
}

\newenvironment{proof}{Proof.}

\title
{\bf Random attractor for the 3D viscous primitive equations driven by  fractional noises \thanks{This work was partially
supported by NNSF of China(Grant No. 11401057),   Natural Science Foundation Project of CQ  (Grant No. cstc2016jcyjA0326),
Fundamental Research Funds for the Central Universities(Grant No. 106112015CDJXY100005) and China Scholarship Council (Grant No.:201506055003).} }
\author{
 Guoli Zhou
\thanks{ Chong Qing University, P.R. China }
\tank{E-mail:zhouguoli736@126.com.}
}
\begin{document}
\maketitle

\begin{abstract}
We develop a new and general method to prove the the existence of the random attractor (strong attractor) for the  primitive equations (PEs) of large-scale ocean and atmosphere dynamics under $non$-$periodic$ boundary conditions and driven by
 infinite-dimensional additive fractional Wiener processes. In contrast to our new method, the common method, compact Sobolev embedding theorem, is to obtain the uniform $a$ $priori$ estimates in some Sobolev space whose regularity is high enough. But this is very complicated for the 3D stochastic
PEs with the $non$-$periodic$ boundary conditions. Therefore, the existence of universal attractor ( weak attractor) was established in previous work.
 The main idea of our method is that we first derive that $\mathbb{P}$-almost surely the solution operator of stochastic PEs is compact. Then we construct a compact absorbing set by virtue of the compact property of the the solution operator and the existence of a absorbing set. We should point out that our method has some advantages over the common method of using compact Sobolev embedding theorem, i.e., if the random attractor in some Sobolev space do exist in view of the common method, our method would then further implies the existence of random attractor in this space. The present work provides a general way for proving the existence of random attractor for common classes of  dissipative stochastic partial differential equations and
 improves the existing results concerning random attractor of stochastic PEs. In a forth coming paper, we use this new method to prove the existence of strong attractor for the stochastic moist primitive equations, improving the results, the existence of weak (universal) attractor of the deterministic model.
\end{abstract}

\noindent{\it Keywords:} \small Primitive equations, Fractional noise, Random attractor

\noindent{\it {Mathematics Subject Classification (2000):}} \small
{60H15, 35Q35.}

\section{Introduction}
\par
The paper is concerned with the PEs in a bounded domain with fractional noises. To outline its content in detail, we introduce a smooth bounded domain $M\subset \mathbb{R}^{2}$ and the cylindrical domain $\mho=M\times (-h,0)\subset \mathbb{R}^{3},$ and consider the following 3D stochastic PEs of Geophysical Fluid Dynamics.
\begin{eqnarray*}
&&\partial_{t} \upsilon+(\upsilon\cdot \nabla)\upsilon+ w \partial_{z} \upsilon +f \upsilon^{\perp} +\nabla p+L_{1}\upsilon=\dot{W}_{1}^{H},\\
&&\partial_{z}p +T=0,\\
&&\nabla\cdot \mathbf{\upsilon}+\partial_{z} w=0,\\
&&\partial_{t} T+\upsilon\cdot \nabla T+w \partial_{z} T+L_{2}T=Q+ \dot{W}_{2}^{H}.
\end{eqnarray*}
The unknowns for the 3D stochastic viscous PEs are the fluid velocity field $(\upsilon,w )=(\upsilon_{1},\upsilon_{2},w )\in \mathbb{R}^{3}$ with  $\upsilon=(\upsilon_{1},\upsilon_{2})$ and $ \upsilon^{\perp}=(-\upsilon_{2}, \upsilon_{1} ) $ being horizontal, the temperature $T$ and the pressure $p.$
$f=f_{0}(\beta+y)$ is the given Coriolis parameter, $Q$ is a given heat source. The viscosity and the heat diffusion operators $L_{1}$ and $L_{2}$ are given by
\begin{eqnarray*}
L_{i}=-\nu_{i}\Delta-\mu_{i}\partial_{zz} , \ \ \ i=1,2.
\end{eqnarray*}
Here the positive constants $\nu_{1}, \mu_{1}$ are the horizontal and vertical Reynolds numbers, respectively, and
$\nu_{2}, \mu_{2}$ are positive constants which stand for the horizontal and vertical heat diffusivity, respectively. To simplify the nations, we assume
$ \nu_{i}= \mu_{i}=1,\ \ i=1,2.$ The results in this paper are still valid when we consider the general cases. We set
$\nabla = (\partial_{x}, \partial_{y})$ to be the horizontal gradient operator and $\Delta=\partial_{x}^{2}+ \partial_{y}^{2}$
to be the horizontal Laplacian. Here, we take $\dot{W}^{H}_{i}(t,x,y,z),i=1,2$, the informal derivative for the fractional Wiener process $W^{H}_{i}$ given below.
\par
The boundary of $\mho$ is partitioned into three parts: $\Gamma_{u}\cup\Gamma_{b}\cup\Gamma_{s}, $ where
\begin{eqnarray*}
&&\Gamma_{u}=\{(x,y,z)\in \overline{\mho}:z=0 \},\\
&&\Gamma_{b}=\{(x,y,z)\in \overline{\mho}:z=-h \},\\
&&\Gamma_{s}=\{(x,y,z)\in \overline{\mho}:(x,y)\in \partial M, -h \leq z\leq 0 \}.
\end{eqnarray*}
Here $h$ is a sufficiently smooth function. Without loss generality, we assume $h=1.$
We consider the following boundary conditions of the stochastic 3D viscous PEs.
\begin{eqnarray*}
&&\partial_{z} \upsilon =\eta,\ \ w=0,\ \ \partial_{z} T=-\alpha (T-\tau)\ \ \mathrm{on}\  \Gamma_{u},\\
&&\partial_{z} \upsilon=0,\ \ w=0,\ \ \partial_{z} T=0\ \ \mathrm{on}\  \Gamma_{b},\\
&&\upsilon\cdot \vec{n}=0,\ \ \partial_{\vec{n}}\upsilon\times \vec{n} =0,\ \  \partial_{\vec{n}} T=0\ \ \mathrm{on}\  \Gamma_{s},
\end{eqnarray*}
where $\eta(x,y)$ is the wind stress on the surface of the ocean, $\alpha$ is a positive constant, $\tau$ is the typical temperature distribution on the top surface of the ocean and $\vec{n}$ is the norm vector to $\Gamma_{s}.$ We assume for the sake of simplicity that $Q$ is independent of time and $\eta=\tau=0$. It is worth pointing that results presented in the paper can still be obtained provided some simple modifications are made.  By elementary calculus, we have that
\[
w(x,y,z,t)=-\int_{-1}^{z}\nabla\cdot \upsilon(x,y,\lambda,t)d\lambda,
\]
\[
p(x,y,z,t)=p_{s}(x,y,t)-\int_{-1}^{z}T(x,y,\lambda,t)d\lambda.
\]
Using the fact, we obtain the following equivalent 3D stochastic PEs:
\begin{eqnarray}
\partial_{t} \upsilon\!\!\!\!&+&\!\!\!\!L_{1}\upsilon+(\upsilon\cdot \nabla)\upsilon-\Big{(}\int_{-1}^{z}\nabla\cdot \upsilon(x,y,\lambda,t)d\lambda\Big{)} \partial_{z}\upsilon \nonumber \\
&&+\nabla p_{s}(x,y,t)-\int_{-1}^{z}\nabla T(x,y,\lambda,t)d\lambda +f \upsilon^{\perp} =\dot{W}^{H}_{1};\\
\partial_{t} T\!\!\!\!&+&\!\!\!\!\upsilon\cdot \nabla T-\Big{(}\int_{-1}^{z}\nabla\cdot \upsilon(x,y,\lambda,t)d\lambda\Big{)} \partial_{z} T+L_{2}T=Q+\dot{W}^{H}_{2};\\
\partial_{z} \upsilon\!\!\!\!\!\!&|&\!\!\!\!\!\!_{\Gamma_{u}}=\partial_{z} \upsilon|_{\Gamma_{b}}=0,
\upsilon\cdot \vec{n}|_{\Gamma_{s}}=0, \partial_{\vec{n}}\upsilon\times \vec{n}|_{\Gamma_{s}}=0;\\
\Big{(}\partial_{z}T\!\!\!\!&+&\!\!\!\!\alpha T\Big{)}|_{\Gamma_{u}}=\partial_{z}T|_{\Gamma_{b}}=0, \ \ \partial_{\vec{n}}T|_{\Gamma_{s}}=0;\\
\upsilon(x,y,\!\!\!\!\!\!&z&\!\!\!\!\!\!,0)=\upsilon_{0}(x,y,z), \ \ T(x,y,z,0)=T_{0}(x,y,z).
\end{eqnarray}

\par
The Primitive Equations are the basic model used in the study of climate and
weather prediction, which describe the motion of the atmosphere when the hydrostatic
assumption is enforced $\cite{G,H1,H2}$. As far as we know, their mathematical study was initiated by J.L.Lions, R.Teman and S.Wang$(\cite{LTW1}-\cite{LTW4})$. And this research field has developed and  has received
considerable attention from the mathematical community over the last two decades.
Lions, Temam
and Wang $\cite{LTW2}$ obtained the existence of global weak solutions for the primitive equations. Guill$\acute{e}$n-Gonz$\acute{a}$lez, et al. $\cite{GMR}$ obtained the global existence of strong solutions to the primitive equations with small initial data. Moreover, they proved the local existence of strong solutions to the equation. The local existence of strong solutions to the primitive equations under the small depth hypothesis was studied by Hu et al. $\cite{HTZ}$. Taking advantage of the fact that the pressure is essentially two-dimensional in the PEs, Cao and Titi $\cite{CT1}$ proved the global results for the existence of strong solutions of the full
three-dimensional PEs. Subsequently, I. Kukavica and M. Ziane $\cite{KZ}$ developed a different proof which allows one to treat non-
rectangular domains as well as different, physically realistic, boundary conditions. The existence of
the global attractor is given by Ju $\cite{J}$.  For the PEs with partial dissipation, we refer the reader to the
papers $\cite{CIN, CLT1, CLT2, CLT3, CT2}$.

\par
Despite the developments in the deterministic case, the theory for the
stochastic PEs remains underdeveloped. B. Ewald, M. Petcu, R. Teman $\cite{EPT}$ and N. Glatt-Holtz, M. Ziane $\cite{GHZ}$ considered a two-dimensional stochastic PEs. Then N. Glatt-Holtz and R. Temam $\cite{GHT1, GHT2}$ extended the case to the greater generality of physically relevant boundary conditions and nonlinear multiplicative noise. Following the methods closer to $\cite{CT1}$, Boling Guo and Daiwen Huang $\cite{GH}$ studied the global well-posedness of the three-dimensional system with a additive noise in the horizontal momentum equations and obtained some kind of weak type compactness properties of the solutions to the stochastic system. Using methods different from $\cite{GH}$, A. Debussche, N. Glatt-Holtz, R. Temam and M. Ziane considered three-dimensional system with multiplicative noise.

\par
Although the PEs express very fundamental laws of physics, the deterministic models are numerically intractable.
Studies have shown that resolved states are associated with many possible unresolved states.
This calls for stochastic methods for numerical weather and climate prediction which potentially allow a proper
representation of the uncertainties, a reduction of systematic biases and improved representation of long-term
climate variability.
Furthermore, while current deterministic parameterization schemes are inconsistent with
the observed power-law scaling of the energy spectrum,  new statistical dynamical approaches that are underpinned by exact stochastic model representations have emerged that overcome this limitation. The observed
power spectrum structure is caused by cascade processes which can be best represented by a stochastic $\mathbf{non}$-$\mathbf{Markovian}$ Ansatz. $\mathbf{Non}$-$\mathbf{Markovian}$ terms are necessary
to model $\mathbf{memory}$ $\mathbf{effects}$ due to model reduction. It means that in order to make skillful predictions the
model has to take into account also past states and not only the current state (as for a Markov process). For more details, please refer to $\cite{CJJ}$ and other references.

\par
   Based on the fact, we consider stochastic PEs, where both the horizontal velocity field and the temperature
 are perturbed by $\mathbf{fractional}$  $\mathbf{Brownian}$  $ \mathbf{motion}$ (fBm). We define stochastic integrals through pathwise generalized Stieltjes integrals as is the case in $\cite{MN,NR,Z}$. So far, various forms of stochastic integrals with respect to fBm have been developed by several authors (see Chapter 5 in Nualart $\cite{N3}$  and the references
therein). The theory of stochastic partial differential equations of parabolic type driven by an fractional noise have received much attention
(see Maslowski and Nualart $\cite{MN}$, Maslowski and Schmalfuss $\cite{MS}$,
Nualart $\cite{N1,N2}$,  Nualart and Vuillermot $\cite{NV}$, Tindel, Tudor and Viens $\cite{TTV}$, and the references therein).

\par
In this article, we mainly study the existence of random attractor for the stochastic PEs perturbed by fractional noises. We should point out that the definitions of the attractors between our paper and $\cite{GH}$ have essential differences. The random attractor obtained in this work is $\mathbb{P}$-a.e. $\omega$ compact in $(H^{1}(\mho))^{3}$ and attracts any orbit starting from $-\infty$ in the strong topology of $(H^{1}(\mho))^{3}$. While the attractor studied in $\cite{GH}$ is not necessary  a compact subset in  $(H^{1}(\mho))^{3}$. In addition, the attractor attracts any orbit in the weak topology of $(H^{1}(\mho))^{3}$.

\par
To consider the dynamical behavior of the strong solution of $3$D stochastic PEs, we will encounter the following difficulties.

\par
The first difficulty involved here is to study moment estimates and growth property of the Ornstein-Uhlenbeck (O-U) processes driven by fractional noises. Since fBm is not a semimartingale and has not independent increment, we can not follow the common method to use the It$\hat{o}$ isometry and the law of large numbers to study its properties. In addition, the integrals with H$\ddot{o}$lder continuous integrators are defined pathwise, which is a qualitative difference to the definition of the classical stochastic integral where the integrand is a white noise. We overcome the difficulties by using analysis techniques and taking advantage of stationary increments and polynomial growth property as well as regularity of fBm. The fundamental results can apply to  the study of long time behavior of stochastic partial differential equations driven by fractional noises.

\par
Secondly, to prove the existence of the random attractor, by virtue of the common method Sobolev compact theorem we should obtain the uniform $a$ $priori$ estimates with respect to initial data in $(H^{2}(\mho))^{3}$(see $\cite{CDF,CF}$). But it is very difficult to achieve under the $non$-$periodic$ boundary conditions. To overcome the difficulty, we need to establish a new method to obtain  a compact absorbing ball in $(H^{1}(\mho))^{3}$ which guarantee the existence of random attractor in $(H^{1}(\mho))^{3}$. The main idea of our new method is that we first try to derive that $\mathbb{P}$-almost surely the solution operator of stochastic PEs is compact in $(H^{1}(\mho))^{3} .$ Then we construct a compact absorbing ball in the strong solution space $(H^{1}(\mho))^{3}$ by using the the solution operator to act on a absorbing ball.

\par
Thirdly, showing the compact property of solution operator is still difficult. To overcome the difficulty, we take advantage of the regularity of solution operator and Aubin-Lions Lemma to achieve our goal. Specifically, we establish the continuity of the strong solutions to the 3D stochastic PEs in the space $(H^{1}(\mho))^{3}$ with respect to time $t$ and with respect to the initial condition $(\upsilon_{0}, T_{0}).$ Notice that $\cite{GH}$ only proved the strong solution is Lipschitz continuous in the space $( L^{2}(\mho))^{3}$ with respect to the initial data but this is not enough to study the asymptotical behavior in $(H^{1}(\mho))^{3}$ considered  here. The new difficulty arose here in obtaining the regularities of the strong solution about time $t$ and initial condition is that we have no valid boundedness  for the derivatives of the vertical velocity. To overcome the difficulty, the special geometry involved with the vertical velocity is used to obtain delicate $a\ priori$ estimates.

\par
It is important to point out that our method develops a general way for proving the existence of random attractor for common classes of dissipative stochastic partial differential equations and has some advantages over the common method of using compact Sobolev embedding theorem, i.e., if an absorbing ball for the solutions in space $(H^{2}(\mho))^{3}$ does exist, our method would then further imply the existence of global random attractor in $(H^{2}(\mho))^{3}$. In our forth coming paper $\cite{WZ}$ using the new method we prove the existence of strong attractor for the stochastic moist primitive equations, which improves the results, the existence of weak (universal) attractor for the deterministic model in $\cite{GH1}$.
\par
The remaining of this paper is organized as follows. In section $2,$ we state some preliminaries and recall some results.  The main result is presented in section $3,$ and section $4$ is for its
proof. Finally, in section $5$, the appendix of our work, we obtain the $a$ $priori$ estimates for the global existence of the strong solutions to the stochastic PEs. As usual, constants
$C$ may change from one line to the next, unless, we give a special declaration
; we denote by $C(a)$ a constant which depends on some parameter $a.$
\section{Preliminaries}
For $1\leq p\leq \infty,$ let $L^{p}(\mho), L^{p}(M)$ be the usual Lebesgue spaces with the norm $|\cdot|_{p}$ and $|\cdot|_{L^{p}(M)}$
respectively. If there is no confusion, we will write $|\cdot|_{p}$ instead of $|\cdot|_{L^{p}(M)}.$ For a positive integer
$m,$ we denote by $(H^{m,p}(\mho), \|\cdot\|_{m,p})$ and
$(H^{m,p}(M),\|\cdot\|_{H^{m,p}(M)})$ the usual Sobolev spaces, see($\cite{ARA}$). When $p=2,$ we denote by
$(H^{m}(\mho), \|\cdot\|_{m})$ and $ (H^{m}(M),\|\cdot\|_{H^{m}(M)})$ for short.
Without confusion, we shall sometime abuse notation and denote by $\|\cdot\|_{m}$ the norm in $ H^{m}(M)$.  Let
\begin{eqnarray*}
V_{1}=\{\mathbf{\upsilon}\in (C^{\infty}(\mho))^{2}&:& \partial_{z} {\upsilon} |_{z=0}=0, \partial_{z}\upsilon|_{z=-h}=0,\\
 &&\mathbf{\upsilon}\cdot \vec{n}|_{\Gamma_{s}}=0, \partial_{\vec{n}} \upsilon \times \vec{n}|_{\Gamma_{s}}=0,  \int_{-1}^{0}\nabla \cdot \mathbf{\upsilon} dz=0 \},
\end{eqnarray*}
\begin{eqnarray*}
V_{2}=\{T\in C^{\infty}(\mho)&:& \partial_{z} T |_{z=-h}=0, (\partial_{z} T+\alpha T)|_{z=0}=0, \partial_{\vec{n}} T |_{\Gamma_{s}}=0 \}.
\end{eqnarray*}
We denote by $\mathcal{V}_{1}$ and $\mathcal{V}_{2}$ be the closure spaces of $V_{1}$ in $(H^{1}(\mho))^{2},$ and $V_{2}$ in $H^{1}(\mho)$ under $H^{1}-$topology, respectively. Let $H_{1}$ be the closure space of $V_{1}$ with respect to the norm $|\cdot|_{2}.$ Define $H_{2}=L^{2}(\mho).$ Set
$$\mathcal{V}= \mathcal{V}_{1}\times \mathcal{V}_{2}, \mathcal{H}=H_{1}\times H_{2} .$$
Let $U_{i}:=(\mathbf{\upsilon}_{i},T_{i})$ be the horizontal velocity and temperature with $i=1,2$. We equip $\mathcal{V}$ with the inner product
\begin{eqnarray*}
\langle U_{1},  U_{2}\rangle_{\mathcal{V}}&:=&\langle \mathbf{\upsilon}_{1},  \mathbf{\upsilon}_{2}\rangle_{\mathcal{V}_{1}}+\langle T_{1},  T_{2}\rangle_{\mathcal{V}_{2}},\\
\langle \mathbf{\upsilon}_{1},  \mathbf{\upsilon}_{2}\rangle_{\mathcal{V}_{1}}&:=&\int_{\mho}(\nabla \mathbf{\upsilon}_{1}\cdot \nabla \mathbf{\upsilon}_{2}+\partial_{z} \mathbf{\upsilon}_{1}\cdot \partial_{z} \mathbf{\upsilon}_{2})d\mho,\\
\langle T_{1},  T_{2}\rangle_{\mathcal{V}_{2}}&:=&\int_{\mho}(\nabla T_{1}\cdot \nabla T_{2}+\partial_{z} T_{1}\cdot \partial_{z} T_{2})d\mho+\alpha\int_{\Gamma_{u}} T_{1}T_{2}d\Gamma_{u}.
\end{eqnarray*}
Subsequently, the norm of $\mathcal{V}$ is defined by $\|U_{i}\|_{1}=\langle U_{i},  U_{i}\rangle^{\frac{1}{2}}.$  We define the inner product of $\mathcal{H}$ by
\begin{eqnarray*}
\langle U_{1},  U_{2}\rangle_{\mathcal{H}}&:=&\langle \mathbf{\upsilon}_{1},  \mathbf{\upsilon}_{2}\rangle+\langle T_{1},  T_{2}\rangle,\\
\langle \mathbf{\upsilon}_{1},  \mathbf{\upsilon}_{2}\rangle&:=&\int_{\mho} \mathbf{\upsilon}_{1}\cdot  \mathbf{\upsilon}_{2}d\mho,\\
\langle T_{1},  T_{2}\rangle&:=&\int_{\mho} T_{1}\cdot  T_{2}d\mho.
\end{eqnarray*}
Denote $\mathcal{V}_{i}'$ the dual space of $\mathcal{V}_{i}$ with $i=1,2.$ And define the linear operator $A_{i}:\mathcal{V}_{i}\mapsto \mathcal{V}_{i}',i=1,2$ as follows:
\begin{eqnarray*}
 \langle A_{1}v_{1},    v_{2} \rangle=\langle v_{1},   v_{2} \rangle_{\mathcal{V}_{1} }, \forall\ v_{1},v_{2}\in  \mathcal{V}_{1};\ \ \ \  \langle A_{2}T_{1},    T_{2}\rangle
 =\langle T_{1},   T_{2} \rangle_{\mathcal{V}_{2} }, \forall\ T_{1},T_{2}\in  \mathcal{V}_{2}.
\end{eqnarray*}
Let $D(A_{i}):=\{\eta\in \mathcal{V}_{i}, A_{i}\eta\in H_{i}    \}.$ Because $A_{i}^{-1}$ is a self-adjoint compact operators in $H_{i},$ thanks to the classic spectral theory, we can define the power $A_{i}^{s}$ for any $s\in \mathbb{R}.$ Then $D(A_{i})'= D(A_{i}^{-1})$ is the dual space of $D(A_{i})$ and
$V_{i}=D(A_{i}^{\frac{1}{2}}), V_{i}'=D(A_{i}^{-\frac{1}{2}}) .$ Furthermore, we have the compact embedding relationship
\begin{eqnarray*}
D(A_{i})\subset V_{i}\subset H_{i}\subset V_{i}'\subset D(A_{i})',
\end{eqnarray*}
and
\begin{eqnarray*}
\langle \cdot,  \cdot \rangle_{\mathcal{V}_{i}}=\langle A_{i}\cdot, \cdot    \rangle= \langle A_{i}^{\frac{1}{2}}\cdot, A_{i}^{\frac{1}{2}}\cdot\rangle, \ \ i=1,2.
\end{eqnarray*}
And set $\alpha\in (0,1).$ For a function $f: [0,T]\rightarrow \mathbb{R}$, we define the Weyl derivative $D^{\alpha}_{0+}f$ by
\begin{eqnarray*}
D^{\alpha}_{0+}f(t)=\frac{1}{\Gamma(1-\alpha)}\Big{(}\frac{f(t)}{t^{\alpha}}+\alpha\int_{0}^{t}\frac{f(t)-f(u)}{(t-u)^{\alpha+1}}du\Big{)},
\end{eqnarray*}
provided the singular integral in the right hand side exists for almost all $t\in(0,T),$ where $\Gamma$ denotes the Euler function. Similarly, we can define $D^{\alpha}_{T-}f(t)$ as
\begin{eqnarray*}
D^{\alpha}_{T-}f(t)=\frac{(-1)^{\alpha}}{\Gamma(1-\alpha)}\Big{(}\frac{f(t)}{(T-t)^{\alpha}}+\alpha\int_{t}^{T}\frac{f(t)-f(u)}{(u-t)^{\alpha+1}}du\Big{)}.
\end{eqnarray*}
Let $\phi\in L^{1}([0,T]),$ a Lebesgue space with values in $\mathbb{R}.$ Then the left and right-sided fractional Riemann-Liouville integrals of $\phi$ of order $\alpha$ are defined for almost
all $t\in (0,T)$ by
\begin{eqnarray*}
I_{0+}^{\alpha}\phi(t)=\frac{1}{\Gamma (\alpha)}\int_{0}^{t}(t-u)^{\alpha-1}\phi(u)du
\end{eqnarray*}
and
\begin{eqnarray*}
I_{T-}^{\alpha}\phi(t)=\frac{(-1)^{-\alpha}}{\Gamma (\alpha)}\int_{t}^{T}(u-t)^{\alpha-1}\phi(u)du,
\end{eqnarray*}
respectively. Let $f= I_{0+}^{\alpha}\phi$. Then the  Weyl derivative of $f$ exists and $D_{0+}^{\alpha}f=\phi.$ A similar result holds for the right-sided fractional integral. For the theory of fractional integrals and derivatives we refer to $\cite{SKM}.$
\par
Let $W^{\alpha,1}([0,T]; \mathbb{R})$ be the space of measurable functions $f:[0,T]\rightarrow \mathbb{R}$ such that
\begin{eqnarray*}
|f|_{\alpha,1}:=\int_{0}^{T}\Big{(}\frac{|f(s)|}{s^{\alpha}}+\int_{0}^{s}\frac{|f(s)-f(u)|}{(s-u)^{\alpha+1}}du  \Big{)}ds<\infty,
\end{eqnarray*}
where $0<\alpha<\frac{1}{2}.$  Similarly, we can define $W^{\alpha,1}([a,b]; \mathbb{R})$ with $a,b\in \mathbb{R}$ and $a<b$. Let
\begin{eqnarray*}
C_{\alpha}(g):= \frac{1}{\Gamma(1-\alpha)\Gamma(\alpha)}\sup\limits_{0<s<t<T}\Big{(}\frac{|g(t)-g(s)|}{(t-s)^{1-\alpha}}+\int_{s}^{t}\frac{|g(u)-g(s)|}{(u-s)^{2-\alpha}}du   \Big{)}.
\end{eqnarray*}
Following $\cite{Z}$ we define the generalized Stieltjes integral $\int_{0}^{T}fdg$ by
\begin{eqnarray}
\int_{0}^{T}fdg=(-1)^{\alpha}\int_{0}^{T}D^{\alpha}_{0+}f(s)D^{1-\alpha}_{T-}g_{T-}(s)ds,
\end{eqnarray}
where $g_{T-}(s)=g(s)-g(T).$ Under the above hypotheses the integral $\int_{0}^{t}fdg$ exists for all $t\in [0,T],$ and $(\mathrm{cf}.\cite{NR})$
\begin{eqnarray*}
\int_{0}^{t}fdg=\int_{0}^{T}f1_{(0,t)}dg.
\end{eqnarray*}
 Furthermore, we have
\begin{eqnarray}
|\int_{0}^{t}fdg| \leq C_{\alpha}(g)|f|_{\alpha,1}.
\end{eqnarray}
\par
Let $((B^{H}_{i}(t))_{t\in \mathbb{R}^{+}})_{i\in \mathbb{N}^{+}}$ be a sequence of one-dimensional, independent, identically distributed fractional Brownian motions with Hurst parameter $H\in (0,1),$ defined on the complete probability space $(\Omega, \mathcal{F}, \mathbb{P} )$ and starting at the origin. For $j\in \{1,2\},$  let $G_{j}$ be a linear, self-adjoint, positive trace-class operator with generating kernel $\eta_{j}$, and  we write $(e_{i,j})_{i\in \mathbb{N}^{+}},$ for an orthonormal basis of $H_{j},$ consisting of eigenfunctions of the operator $G_{j}$ and $(\lambda_{i,j})_{i\in \mathbb{N}^{+},}$ for the sequence of the corresponding eigenvalues.
We introduce the $H_{j}$-valued fractional Wiener process $(W^{H}_{j}(.,t))_{t\in \mathbb{R}^{+}}$ with $j=1,2,$ by setting
\begin{eqnarray}
W^{H}_{j}(.,t):=\sum^{\infty}_{i=1}\lambda_{i,j}^{\frac{1}{2}}e_{i,j}(.)B^{H}_{i}(t),
\end{eqnarray}
where the series converges a.s. in the strong topology of $H_{j}$ by virtue of the basic properties of the $B^{H}_{i}(t)$'s and the fact that $G_{j}$ is trace-class. From these basic properties, we can also conclude that $(W^{H}_{j}(.,t))_{t\in \mathbb{R}^{+},j=1,2}$ is a centered Gaussian process whose covariance is given by
\begin{eqnarray*}
&&\mathbb{E}(\langle W^{H}_{j}(.,s),\mathbf{\upsilon}\rangle \langle W^{H}_{j}(.,t),\hat{\mathbf{\upsilon}}\rangle )\\
&=&\frac{1}{2}\Big{(}s^{2H}+t^{2H}-|s-t|^{2H}\Big{)}\langle G_{j}\mathbf{\upsilon}, \hat{\mathbf{\upsilon}}\rangle\nonumber\\
&=&\frac{1}{2}\Big{(}s^{2H}+t^{2H}-|s-t|^{2H}\Big{)}\int_{\mho\times \mho} \eta_{j} \mathbf{\upsilon}\cdot\hat{\mathbf{\upsilon}}d\mho d\mho\nonumber
\end{eqnarray*}
for all $s,t\in \mathbb{R}^{+}$ and all $\mathbf{\upsilon}, \hat{\mathbf{\upsilon}}\in H_{j}$.
Take a parameter $\alpha\in (1-H, \frac{1}{2})$ which will be fixed throughout the paper.  Let $f\in W^{\alpha,1}([0,T];\mathbb{R}),$
we define
$$\int_{0}^{T}f(s)dB^{H}_{i}(s) $$
in sense of $(2.6)$ pathwise, since by $\cite{NR}$ we have $C_{\alpha}(B^{H}_{i})<\infty, \mathbb{P}-a.s.$ for $t\in \mathbb{R}^{+}$ and $i\in \mathbb{N}^{+}.$ Denote $\mathcal{L}(\mathcal{V}_{j})$ the space of linear bounded operators on $\mathcal{V}_{j}$ with $j=1,2.$ Suppose $l: \Omega \times [0,T]\rightarrow \mathcal{L}(\mathcal{V}_{j})$ be an operator-valued function such that
$l e_{i,j}\in W^{\alpha,1}([0,T];\mathcal{V}_{j})$ for $i\in \mathbb{N}^{+},\ \omega\in \Omega.$  We define
\begin{eqnarray}
\int_{0}^{T}\varphi(s)dW^{H}_{j}(s)&:=&\sum_{i=1}^{\infty}\int_{0}^{T}l(s)G^{\frac{1}{2}}_{j}e_{i,j}dB^{H}_{i}(s)\nonumber\\
&=&\sum_{i=1}^{\infty}\sqrt{\lambda_{i,j}}\int_{0}^{T}l(s)e_{i,j}dB^{H}_{i}(s),
\end{eqnarray}
the convergence of the sums in $(2.9)$ being understood as $\mathbb{P}-a.s.$ convergence in $\mathcal{V}_{j}.$
We consider the equation:
\begin{eqnarray}
dZ_{j}(t)=-A_{j} Z_{j}dt+dW^{H}_{j}(t),
\end{eqnarray}
with initial condition $Z_{j}(0)=0,\ j=1,2.$ The solution of $(2.10)$ can be understood pathwise in the mild sense, that is, the solution $(Z_{j}(t))_{t\in [0,T]}$ is a $\mathcal{V}_{j}$-valued process whose paths are with probability one elements of the space $W^{\alpha,1}([0,T];\mathcal{V}_{j})$ for $\alpha \in (1-H, \frac{1}{2}),$ such that
\begin{eqnarray}
Z_{j}(t)=\int_{0}^{t}S_{j}(t-s)dW^{H}_{j}(s),
\end{eqnarray}
where $S_{j}(t-s)=e^{-A_{j} (t-s)}, t\in [0,T].$
Denote by $0<\gamma_{1,j}\leq \gamma_{2,j}\leq \cdots$ the eigenvalues of $A_{j}$
with corresponding eigenvectors $e_{1,j},e_{2,j},\cdots $.
\begin{proposition}Assume $\tau>0$ and
\begin{eqnarray}
\sum_{i=1}^{\infty}\lambda_{i,1}^{\frac{1}{2}}\gamma_{i,1}^{\frac{5}{2}}<\infty,
\end{eqnarray}
then $(Z_{1}(t))_{t\in[0,\tau]}$ exists as a generalized Stieltjes integral in the sense of $\cite{Z}$ and we have
\begin{eqnarray*}
(Z_{1}(t))_{t\in[0, \tau]}\in C([0,\tau];(H^{3}(\mho))^{2})\ \ a.s.,
\end{eqnarray*}
where $ C([0,\tau];(H^{3}(\mho))^{2})$ is the set of continuous functions defined on $[0,\tau]$ with values in $(H^{3}(\mho))^{2}. $
\end{proposition}
\begin{proof}
According to Lemma $2.2$ in $\cite{MN}$ and $(2.7),$ there exists a finite, positive random variable $C_{\alpha}(B^{H}_{i}),$ depending only on $\alpha, B^{H}_{i}$ such that
\begin{eqnarray*}
\|Z_{1}(t)\|_{3}&\leq&C\sum_{i=1}^{\infty}\lambda_{i,1}^{\frac{1}{2}}|\int_{0}^{t}A_{1}^{\frac{3}{2}}S_{1}(t-s)e_{i,1}dB^{H}_{i}(s)|_{2}\nonumber\\
&\leq&C\sum_{i=1}^{\infty}\lambda_{i,1}^{\frac{1}{2}}C_{\alpha}(B^{H}_{i})|A_{1}^{\frac{3}{2}}S_{1}(t-s)e_{i,1}|_{\alpha,1}\nonumber\\
&=&C\sum_{i=1}^{\infty}\lambda_{i,1}^{\frac{1}{2}}C_{\alpha}(B^{H}_{i})\int_{0}^{t}\Big{(}\frac{|A_{1}^{\frac{3}{2}}S_{1}(t-s)e_{i,1}|_{2}}{s^{\alpha}}  \nonumber\\
&&+\int_{0}^{s}\frac{|A_{1}^{\frac{3}{2}}S_{1}(t-s)e_{i,1}-A_{1}^{\frac{3}{2}}S_{1}(t-u)e_{i,1}|_{2} }{(s-u)^{1+\alpha}}du\Big{)}ds.\nonumber\\
\end{eqnarray*}
Since the orthonormal basis $(e_{i,1})_{i\in \mathbb{N}^{+}}$ of $H_{1}$ consist of eigenfunctions of operator $A_{1}$ with corresponding eigenvalues $(\gamma_{i,1})_{i\in \mathbb{N}^{+}},$ we get
\begin{eqnarray*}
\|Z_{1}(t)\|_{3}&\leq&C\sum_{i=1}^{\infty}\lambda_{i,1}^{\frac{1}{2}}\gamma_{i,1}^{\frac{3}{2}}C_{\alpha}(B^{H}_{i})\int_{0}^{t}\frac{e^{-\gamma_{i,1}(t-s)}}{s^{\alpha}}ds\\
&&+\sum_{i=1}^{\infty}\lambda_{i,1}^{\frac{1}{2}}\gamma_{i,1}^{\frac{3}{2}}C_{\alpha}(B^{H}_{i})\int_{0}^{t}\Big{(}\int_{0}^{s}\frac{|e^{-\gamma_{i,1}(t-s)}-e^{-\gamma_{i,1}(t-u)}|}{(s-u)^{1+\alpha}}du\Big{)}ds\\
&\leq&C\sum_{i=1}^{\infty}\lambda_{i,1}^{\frac{1}{2}}\gamma_{i,1}^{\frac{3}{2}}C_{\alpha}(B^{H}_{i})\Big{(}\frac{t^{1-\alpha}}{1-\alpha}
+\int_{0}^{t}\Big{(}\int_{0}^{s} \frac{\gamma_{i,1}^{\frac{1}{2}}(s-u)^{\frac{1}{2}}}{(s-u)^{1+\alpha}}du\Big{)}ds\Big{)}\\
&\leq&C\sum_{i=1}^{\infty}\lambda_{i,1}^{\frac{1}{2}}\gamma_{i,1}^{2}C_{\alpha}(B^{H}_{i})({t^{1-\alpha}}+t^{\frac{3}{2}-\alpha}),
\end{eqnarray*}
where $\alpha\in (0, \frac{1}{2}).$ From $(2.12),$ we deduce
$$\sum_{i=1}^{\infty}\lambda_{i,1}^{\frac{1}{2}}\gamma_{i,1}^{2}E(C_{\alpha}(B^{H}_{i}))\leq
C\sum_{i=1}^{\infty}\lambda_{i,1}^{\frac{1}{2}}\gamma_{i,1}^{\frac{5}{2}}<\infty$$ by virtue of the fact that the $B^{H}_{i}$'s are identically distributed.
Therefore, for each $t\in [0,\tau]$
\begin{eqnarray}
\|Z_{1}(t)\|_{3}\leq C\sum_{i=1}^{\infty}\lambda_{i,1}^{\frac{1}{2}}\gamma_{i,1}^{2}C_{\alpha}(B^{H}_{i})<\infty\ a.s..
\end{eqnarray}
Furthermore, we obtain
\begin{eqnarray}
\|Z_{1}(t)-Z_{1}(s)\|_{3}&\leq& \|\int_{s}^{t}S_{1}(t-u)dW^{H}_{1}(u)\|_{3}\nonumber\\
&&+\|\int_{0}^{s}(S_{1}(t-u)-S_{1}(s-u))dW^{H}_{1}(u)\|_{3}\nonumber\\
&=&I_{1}+I_{2}.
\end{eqnarray}
Analogously to the derivation of $(2.13),$ we get for $I_{1}$
\begin{eqnarray*}
I_{1}\leq C\sum_{i=1}^{\infty}\lambda_{i,1}^{\frac{1}{2}}
\gamma_{i,1}^{2}C_{\alpha}(B^{H}_{i})({(t-s)^{1-\alpha}}+(t-s)^{\frac{3}{2}-\alpha}).
\end{eqnarray*}
As for $I_{2},$ from $(2.7)$ and $(2.14),$ we deduce
\begin{eqnarray*}
I_{2}&\leq& \sum_{i=1}^{\infty}\lambda_{i,1}^{\frac{1}{2}} \|\int_{0}^{s}(S_{1}(t-u)-S_{1}(s-u))e_{i,1}dB^{H}_{i}(u) \|_{3}\nonumber\\
&\leq&C\sum_{i=1}^{\infty}\lambda_{i,1}^{\frac{1}{2}}\gamma_{i,1}^{\frac{3}{2}}C_{\alpha}(B^{H}_{i})|(S_{1}(t-u)-S_{1}(s-u))e_{i,1}|_{\alpha,1}\nonumber\\
&=& C\sum_{i=1}^{\infty}\lambda_{i,1}^{\frac{1}{2}}\gamma_{i,1}^{\frac{3}{2}}C_{\alpha}(B^{H}_{i})\Big{(}\int_{0}^{s}\frac{|(S_{1}(t-u)-S_{1}(s-u))e_{i,1}|_{2}}{u^{\alpha}}du\nonumber\\
&&+\int_{0}^{s}\int_{0}^{u}\frac{|(S_{1}(t-u)-S_{1}(s-u))e_{i,1}-(S_{1}(t-s_{1})-S_{1}(s-s_{1}))e_{i,1}|_{2}}{(u-s_{1})^{1+\alpha}}ds_{1}du\Big{)}\nonumber\\
&=:& C\sum_{i=1}^{\infty}\lambda_{i,1}^{\frac{1}{2}}\gamma_{i,1}^{\frac{3}{2}}C_{\alpha}(B^{H}_{i}) (I_{3}+I_{4}).
\end{eqnarray*}
Similarly to the above proof, we can derive that
\begin{eqnarray*}
I_{3}&\leq& C\gamma_{i,1}^{2}s^{1-\alpha}(t-s)^{\frac{1}{2}}.
\end{eqnarray*}
Proceeding as in $(2.13)$ we get
\begin{eqnarray*}
I_{4}&\leq&\int_{0}^{s}\int_{0}^{u}\frac{|(e^{-\gamma_{i,1}(t-u)}- e^{-\gamma_{i,1}(s-u)})- (e^{-\gamma_{i,1}(t-s_{1})}- e^{-\gamma_{i,1}(s-s_{1})})        |}{(u-s_{1})^{1+\alpha}}\nonumber\\
&=&\int_{0}^{s}\Big{(}(e^{-\gamma_{i,1}(s-u)}- e^{-\gamma_{i,1}(t-u)})\nonumber\\
&&\int_{0}^{u}\frac{1- (e^{-\gamma_{i,1}(s-s_{1})}- e^{-\gamma_{i,1}(t-s_{1})})(e^{-\gamma_{i,1}(s-u)}- e^{-\gamma_{i,1}(t-u)})^{-1} }{ (u-s_{1})^{1+\alpha}}\Big{)}ds_{1}du.
\end{eqnarray*}
Since
\begin{eqnarray*}
1- (\!\!\!\!\!\!\!\!&&\!\!\!\!\!\!\!\!e^{-\gamma_{i,1}(s-s_{1})}- e^{-\gamma_{i}(t-s_{1})})(e^{-\gamma_{i,1}(s-u)}- e^{-\gamma_{i,1}(t-u)})^{-1} \\
&=&1- (e^{-\gamma_{i,1}(s)}- e^{-\gamma_{i,1}(t)})e^{\gamma_{i,1}(s_{1}-u)}(e^{-\gamma_{i,1}s}- e^{-\gamma_{i,1}t})^{-1}\\
&=&1-e^{-\gamma_{i,1}(u-s_{1})},
\end{eqnarray*}
we get the estimates of $I_{4}$ that
\begin{eqnarray*}
I_{4} &\leq & \gamma_{i,1}^{\frac{1}{2}}(t-s)^{\frac{1}{2}}\int_{0}^{s}\int_{0}^{u}
\frac{1-e^{-\gamma_{i,1}(u-s_{1})}}{ (u-s_{1})^{1+\alpha}}ds_{1}du\nonumber\\
&\leq&C \gamma_{i,1}^{\frac{1}{2}}(t-s)^{\frac{1}{2}}\int_{0}^{s}\int_{0}^{u}
 \frac{\gamma_{i,1}^{\frac{1}{2}}(u-s_{1})^{\frac{1}{2}}}{(u-s_{1})^{1+\alpha}}du\nonumber\\
&\leq&C\gamma_{i,1}s^{\frac{3}{2}-\alpha}(t-s)^{\frac{1}{2}}.
\end{eqnarray*}
By $(2.14)$ and estimates of $I_{1}-I_{4}$ we obtain that
\begin{eqnarray*}
\|Z_{1}(t)-Z_{1}(s)\|_{3}\leq C \sum_{i=1}^{\infty}\lambda_{i,1}^{\frac{1}{2}}\gamma_{i,1}^{\frac{5}{2}}
C_{\alpha}(B^{H}_{i})
(t-s)^{\frac{1}{2}},
\end{eqnarray*}
which complete the proof.
\hspace{\fill}$\square$
\end{proof}
Following the same steps, we can also have the result below.
\begin{proposition}Assume $\tau>0$ and
\begin{eqnarray}
\sum_{i=1}^{\infty}\lambda_{i,2}^{\frac{1}{2}}\gamma_{i,2}^{\frac{5}{2}}<\infty,
\end{eqnarray}
then $(Z_{2}(t))_{t\in[0,\tau]}$ exists as a generalized Stieltjes integral in the sense of $\cite{Z}$ and we have
\begin{eqnarray*}
(Z_{2}(t))_{t\in[0, \tau]}\in C([0,\tau];H^{3}(\mho))\ \ a.s.,
\end{eqnarray*}
where $ C([0,\tau]; H^{3}(\mho))$ is the set of continuous functions defined on $[0,\tau]$ with values in $H^{3}(\mho). $
\end{proposition}
Before proving that the mild solution $(2.11)$ is equivalent to a distribution solution of $(2.10),$ we introduce some notations and definitions from $\cite{MYS}.$
\begin{definition}
 The nonrandom function $ f:\mathbb{R}\rightarrow \mathbb{R}$ is called piecewise
H$\ddot{o}$der of order $\alpha$ on the interval $[T_{1},T_{2}]\subset \mathbb{R}\ (f \in C^{\alpha}_{pw}
[T_{ 1} ,T_{ 2} ])$, if there exists
a finite set of disjoint subintervals $\{[a_{i} ,b_{i} ),1 \leq i \leq N |
\bigcup_{i=1}^{N}[a_{ i} ,b_{ i} ) \cup T_{ 2}=[T_{ 1} ,T_{ 2} ]\}$ and the function $f \in C^{\alpha}[a_{ i} ,b_{ i} )$ for $ 1 \leq i \leq N.$
As before, we denote
\begin{eqnarray*}
\|f\|_{C^{\alpha}[a_{ i} ,b_{ i} )}:=\sup\limits_{a_{i}\leq t< b_{i}}|f(t)|+\sup\limits_{a_{i}\leq s<t< b_{i}}\frac{|f(t)-f(s)|}{|t-s|^{\alpha}}.
\end{eqnarray*}
\end{definition}
For $f\in C^{\alpha}[T_{1} ,T_{2}],$ let
\begin{eqnarray*}
\|f\|_{C^{\alpha}_{pw}[T_{1} ,T_{2} ]}:=\max\limits_{1\leq i\leq N}\|f\|_{C^{\alpha}[a_{ i} ,b_{ i} )}.
\end{eqnarray*}
To prove the equivalence of mild solution and distribution solution of linear stochastic system,  we need stochastic Fubini theorem with respect to fractional Brownian motion.  For convenience, we cite the theorem here, of which proof can be found in $\cite{MYS}.$
\begin{theorem}
Denote by $(B^{H}(u))_{u\in [T_{1},T_{2}]} $ the scalar fBm with Hurst index $H\in (0,1)$ and $T_{1}<T_{2}, T_{1}, T_{2}\in \mathbb{R}.$  Let $\Omega' \subset \Omega$ such that $ P(\Omega') = 1$ and assume
for any $\omega \in \Omega'$ the function $\Phi(t, u, \omega) $ satisfy the conditions:
\par
1) $\forall t \in (T_{1}, T_{2}), \Phi(t, u, \omega)$ is piecewise H$\ddot{o}$lder of order $\beta > 1 - H$ in $u \in
[T_1, T_2]$, and there exists $C = C(\omega) > 0$ such that $\|\Phi(t, ¡¤, \omega)\|_{C^{\beta}
_{pw}[T_1,T_2]}\leq C;$
\par
2) the function $ \int_{T_{1}}^{T_{2}} \Phi(t, u, \omega)dB^{ H}(u) $ is Riemann integrable in the interval
$[T_1, T_2]$.
\par
Then there exist the repeated integrals
\begin{eqnarray*}
I_1 := \int_{T_{1}}^{ T_{2}} (\int_{T_{1}}^{ T_{2}}\Phi(t, u, \omega)dB^{H}(u)) dt\ \mathrm{and}\
I_2 := \int_{T_{1}}^{ T_{2}} (\int_{T_{1}}^{ T_{2}}\Phi(t, u, \omega)dt)dB^{H}(u),
\end{eqnarray*}
and $I_1 = I_2, \mathbb{P}$ -a.s..
\end{theorem}
\begin{proposition}
The mild solution $(2.11)$ to $(2.10)$ is also a distribution solution to $(2.10)$ and vice versa.
\end{proposition}
\begin{proof}
Let $\xi \in D(A_{j}^{*}).$ For $t\in [0,T],$ using $(2.13)$ and Proposition $2.2.$ of $\cite{Z},$ we have
\begin{eqnarray}
\int_{0}^{t}\langle Z_{j}(s), A_{j}^{*}\xi\rangle ds&=& \int_{0}^{t}\langle\int_{0}^{s} S_{j}(s-u)dW^{H}_{j}(u),  A_{j}^{*}\xi\rangle ds\nonumber\\
&=&\sum^{\infty}_{i=1}\lambda_{i,j}^{\frac{1}{2}}\int_{0}^{t} \int_{0}^{s}\langle S_{j}(s-u)e_{i,j},  A_{j}^{*}\xi\rangle dB^{H}_{i}(u)ds\nonumber\\
&=&\sum^{\infty}_{i=1}\lambda_{i,j}^{\frac{1}{2}}\int_{0}^{t}\int_{0}^{t}1_{[0, s]}(u) \langle S_{j}(s-u)e_{i,j},   A_{j}^{*}\xi\rangle dB^{H}_{i}(u)ds,
\end{eqnarray}
where $j=1,2.$ Now we define the the function $ f_{j}: [0,t]\times [0,t]\rightarrow \mathbb{R}$ by
\begin{eqnarray*}
f_{j}(s,u):=1_{[0, s]}(u) \langle S_{j}(s-u)e_{i,j},   A_{j}^{*}\xi\rangle .
\end{eqnarray*}
It is easy to check that for $\forall s\in[0,t],f_{j}(s,\cdot) $ is piecewise H$\ddot{o}$lder of order $1$ in $u\in [0,t]$ (see Definition 2.1 ), and there exists $C=C(\omega)>0$ such that $\|f_{j}(s,\cdot)\|_{C^{1}_{pw}[0,t]}\leq C.$ Indeed, for $s\in [0,t],$ in order to check the regularity of $f_{j}(s,\cdot)$  we only need to consider $u\leq s.$ For $u_{1},u_{2}\in [0,s],$ we have
\begin{eqnarray*}
|f_{j}(s,u_{1})-f_{j}(s,u_{2})|&\leq &|\langle (S_{j}(s-u_{1})-S_{j}(s-u_{2}))e_{i,j},  A_{j}^{*}\xi\rangle|\\
&\leq&|A_{j}^{*}\xi |_{2}|e^{-(s-u_{1})\gamma_{i,j}}-e^{-(s-u_{2})\gamma_{i,j}}|\\
&\leq& C\gamma_{i,j}|u_{1}-u_{2}|.
\end{eqnarray*}
Since  the integrand in $Z_{j}$ is infinitely differentiable with respect to time, by Theorem 4.2.1 in $\cite{Z}$ we get that the stochastic calculus agrees with the
Riemann--Stieltjes integral. Using $(2.13)$ again, we know $Z_{j}$ is Riemann integrable.
Therefore, applying Theorem $2.1$ to $(2.16)$, we obtain
\begin{eqnarray*}
\int_{0}^{t}\langle Z_{j}(s), A_{j}^{*}\xi\rangle ds&=&\sum^{\infty}_{i=1}\lambda_{i,j}^{\frac{1}{2}}\int_{0}^{t}\int_{0}^{t}1_{[0, s]}(u) \langle S_{j}(s-u)e_{i,j},   A_{j}^{*}\xi\rangle ds dB^{H}_{i}(u)\\
&=&\sum^{\infty}_{i=1}\lambda_{i,j}^{\frac{1}{2}}\int_{0}^{t}\int_{u}^{t} \langle A_{j}S_{j}(s-u)e_{i,j}, \xi\rangle ds dB^{H}_{i}(u)\\
&=&\sum^{\infty}_{i=1}\lambda_{i,j}^{\frac{1}{2}}\int_{0}^{t} \langle (S_{j}(t-u)-I)e_{i,j}, \xi\rangle dB^{H}_{i}(u)\\
&=&\langle Z_{j}(t), \xi\rangle- \langle W^{H}_{j}(t), \xi\rangle
\end{eqnarray*}
which implies that mild solution $(2.11)$ to $(2.10)$ is the distribution solution to $(2.10)$. Conversely,
by Theorem 4.2.1 in $\cite{Z}$ we get that the stochastic calculus $Z_{j}$ agrees with the
Riemann--Stieltjes integral.  Then following the steps in $\cite{PZ}$ we can show that a distribution solution to $(2.10)$ is  also a mild solution to $(2.10)$.
\hspace{\fill}$\square$
\end{proof}
\par
To simplify the notations, we denote
\begin{eqnarray*}
w(x,y,z,t)=-\int^{z}_{-1}\nabla \cdot v(x,y,\lambda,t)d\lambda:=\varphi(v)(x,y,z,t)
\end{eqnarray*}
and set
\begin{eqnarray*}
\int_{\mho}\cdot d\mho:=\int_{\mho}\cdot\ ,\ \ \ \ \int_{M}\cdot dM:=\int_{M}\cdot\ .
\end{eqnarray*}
\begin{definition}
Given $\mathcal{T}>0,$ we say a continuous $\mathcal{V}$-valued $(\mathcal{F}_{t})=(\sigma(W^{H}_{j}(s),s\in [0,t]), j=1,2)$ adapted random field $(U(.,t))_{t\in [0,T]}=(\upsilon(.,t),T(.,t))_{t\in [0,\mathcal{T}]}$ defined on $(\Omega, \mathcal{F}, \mathbb{P})$ is a strong solution (weak solution) to problem $(1.1)-(1.5)$ if the following two conditions hold:\\
(1) We have $U\in L^{2}([0,\mathcal{T}];(H^{2}(\mho))^{3})\cap C([0,\mathcal{T}];\mathcal{V})\ ( U\in L^{2}([0,\mathcal{T}];(H^{1}(\mho))^{3})\cap C([0,\mathcal{T}];\mathcal{H}))$ a.s..\\
(2)The integral relation
\begin{eqnarray*}
&&\int_{\mho}\!\upsilon(t)\!\cdot  \phi_{1} -\int^{t}_{0}dt\int_{\mho}\{[(\upsilon\cdot \nabla )\phi_{1} +\varphi(\upsilon)\partial_{z}\phi_{1} ]\upsilon-[(fk\times \upsilon)\cdot \phi_{1}+(\int_{-1}^{z}Tdz')\nabla \cdot \phi_{1}] \}\nonumber\\
&&+\int_{0}^{t}\int_{\mho}\upsilon\cdot L_{1}\phi_{1} =\int_{\mho}\upsilon_{0}\cdot \phi_{1}+\int_{\mho}W^{H}_{1}(t,w)\cdot \phi_{1},\nonumber\\
&&\int_{\mho}T(t)\phi_{2}-\int_{0}^{t}\int_{\mho}\{[(\upsilon\cdot \nabla)\phi_{2}+ \varphi(\upsilon)\partial_{z}\phi_{2}]T-TL_{2}\phi_{2} \}=\int_{\mho}T_{0}\phi_{2}\nonumber\\
&&+\int_{0}^{t}\int_{\mho}Q\phi_{2}+\int_{\mho}W^{H}_{2}(t,w)\cdot \phi_{2},
\end{eqnarray*}
hold a.s. for all $t\in [0,\mathcal{T}]$ and $\phi=(\phi_{1},\phi_{2})\in D(A_{1})\times D(A_{2}).$
\end{definition}
Let $u(t)=\upsilon(t)-Z_{1}(t)$ and $\theta(t)= T(t)-Z_{2}(t), t\in \mathbb{R}^{+}.$ A stochastic process $U(t,w)=(\upsilon,T)$ is a strong solution to $(1.1)-(1.5)$ on $[0,\mathcal{T}]$, if and only if $(u,\theta)$ is a strong solution to the following problem on $[0,\mathcal{T}]:$
\begin{eqnarray}
&&\partial_{t} u-\Delta u-\partial_{zz} u+[(u+{Z}_{1})\cdot\nabla ](u+{Z}_{1})+\varphi(u+{Z}_{1}) \partial_{z}(u+{Z}_{1})\nonumber\\
&& +f( u+{Z}_{1})^{\bot}+\nabla p_{s}-\int_{-1}^{z}\nabla T dz'=0;\\
&&\partial_{t}\theta -\Delta \theta-\partial_{zz}\theta +[(u+{Z}_{1})\cdot\nabla ](\theta+Z_{2})+\varphi(u+{Z}_{1}) \partial_{z} (\theta+Z_{2})=Q;\\
&&\int^{0}_{-1}\nabla\cdot udz=0;\\
&&\partial_{z} u|_{\Gamma_{u}}=\partial_{z}u|_{\Gamma_{b}}=0;
u\cdot \vec{n}|_{\Gamma_{s}}=0, \partial_{\vec{n}}u\times \vec{n}|_{\Gamma_{s}}=0;\\
&&\Big{(}\partial_{z}\theta+\alpha \theta\Big{)}|_{\Gamma_{u}}=\partial_{z}\theta|_{\Gamma_{b}}=0, \ \ \partial_{\vec{n}}\theta|_{\Gamma_{s}}=0;\\
&&(u\Big{|}_{t=0}, \theta\Big{|}_{t=0})=(0, T_{0}).
\end{eqnarray}

\begin{definition}
Let $Z_{j},j=1,2,$ are defined above, $v_{0}\in \mathcal{V}_{1}, T_{0}\in \mathcal{V}_{2}$ and $\mathcal{T}$ be a fixed positive time. For $P-a.e. \omega\in \Omega, (u,\theta)$ is called a strong solution of the system $(2.17)-(2.22)$ on the time interval $[0,\mathcal{T}]$ if it satisfies $(2.17)-(2.18)$ in the weak sense such that
\begin{eqnarray*}
&&u\in C([0,\mathcal{T}];\mathcal{V}_{1})\cap L^{2}([0,\mathcal{T}]; (H^{2}(\mho))^{2}),\\
&&\theta\in C([0,\mathcal{T}];\mathcal{V}_{2})\cap L^{2}([0,\mathcal{T}]; H^{2}(\mho)).
\end{eqnarray*}
\end{definition}
Let $(X, d)$ be a polish space and $(\tilde{\Omega}, \tilde{\mathcal{F}}, \tilde{\mathbb{P}}  )$ be a probability space, where $ \tilde{\Omega}$ is the two -sided Wiener space $C_{0}(\mathbb{R}; X )$ of continuous functions with values in $X$, equal to $0$ at $t=0$. We consider a family of mappings
$S(t,s;\omega):X\rightarrow X,\ \ -\infty<s \leq t< \infty,$ parametrized by $\omega\in \tilde{\Omega},$ satisfying for $ \tilde{\mathbb{P}}$-$a.e.\ \omega$ the following properties (i)-(iv):
\par
(i)$\ \ S(t,r;\omega)S(r,s;\omega)x= S(t,s;\omega)x$ for all $s\leq r\leq t$ and $x\in X;$
\par
(ii)$\ \ S(t,s;\omega)$ is continuous in $X,$ for all $s\leq t;$
\par
(iii)\ \  for all $s<t$ and $x\in X$, the mapping
\[
\omega\mapsto S(t,s;\omega)x
\]
\ \ \ \ \ \ \ \ \ \ \ \ is measurable from $(\tilde{\Omega},\tilde{\mathcal{F}})$ to $(X, \mathcal{B}(X )  )$ where $\mathcal{B}(X ) $ is the Borel-$\sigma$- algebra of $ X$;
\par
(iv)\ \ for all $t, x\in X,$ the mapping $s\mapsto S(t,s;\omega)$ is right continuous at any point.
\par
A set valued map $K: \tilde{\Omega}\rightarrow 2^{X}$ taking values in the closed subsets of $X$ is said to be measurable if for each $x\in X$ the map
$\omega\mapsto d(x, K(\omega))$ is measurable, where $d(A,B)=\sup\{\inf\{d(x,y):y\in B \}:x\in A \}$ for $A,B \in 2^{X}, A,B\neq \emptyset;$
and $d(x,B)=d(\{x\},B).$ Since $ d(A,B)=0$ if and only if $A\subset B, d$ is not a metric. A closed set valued measurable map $K:\tilde{\Omega}\rightarrow 2^{X}$ is named a random closed set.
\par
Given $t\in \mathbb{R}$ and $\omega\in \tilde{\Omega}, K(t,\omega)\subset X$ is called an attracting set at time $t$ if , for all bounded sets $B\subset X,$
\[
d(S(t,s;\omega)B, K(t,\omega) )\rightarrow 0,\ \ provided\ s\rightarrow -\infty.
\]
Moreover, if for all bounded sets $B\subset X,$ there exists $t_{B}(\omega)$ such that for all $s\leq t_{B}(\omega)$
\[
S(t,s;\omega)B\subset K(t,\omega),
\]
we say $K(t,\omega) $ is an absorbing set at time $t.$

Let $\{\vartheta_{t}:\tilde{\Omega}\rightarrow \tilde{\Omega}   \}, t\in T, T=\mathbb{R},$ be a family of measure preserving transformations of the probability space $(\tilde{\Omega}, \tilde{\mathcal{F}},\tilde{ \mathbb{P}} )$ such that for all $s< t$ and $\omega\in \tilde{\Omega}$
\par
(a) $(t,\omega)\rightarrow \vartheta_{t}\omega$ is measurable;
\par
(b) $\vartheta_{t}(\omega)(s)=\omega(t+s)-\omega(t)$;
\par
(c) $S(t,s;\omega)x=S(t-s,0;\vartheta_{s}\omega)x.$\\
Thus $(\vartheta_{t} )_{t\in T}$ is a flow, and
$((\tilde{\Omega}, \tilde{\mathcal{F}},\tilde{ \mathbb{P}} ), (\vartheta_{t} )_{t\in T} )$ is a measurable dynamical system.

\begin{definition}
Given a bounded s$\!$et $B\subset X$, the s$\!$et
\begin{eqnarray*}
\mathcal{A}(B,t,\omega)=\bigcap\limits_{T\leq t}\overline{\bigcup\limits_{s\leq T}S(t, s,\omega)B}
\end{eqnarray*}
is said to be the $\Omega$-limit set of $B$ at time $t$. Obviously, if we denote $\mathcal{A}(B,0,\omega)=\mathcal{A}(B,\omega),$ we have
$\mathcal{A}(B,t,\omega)=\mathcal{A}(B,\vartheta_{t}\omega).$
\end{definition}
We may identify
\begin{eqnarray*}
\mathcal{A}(B,t,\omega)=\{x\in X &:  &\mathrm{there}\ \mathrm{exists}\ s_{n}\rightarrow -\infty\ \mathrm{and}\ x_{n}\in B\nonumber\\
 &&\mathrm{such}\ \mathrm{that}\ \lim\limits_{n\rightarrow\infty}S(t,s_{n},\omega)x_{n}=x\}.
\end{eqnarray*}
Furthermore, if there exists a compact attracting set $K(t,\omega)$ at time $t,$ it is not difficult to check that $\mathcal{A}(B,t,\omega)$ is a nonempty compact subset of $X$ and $\mathcal{A}(B,t,\omega)\subset K(t,\omega). $
\begin{definition}
If, for all $t\in \mathbb{R}$ and $\omega\in \tilde{\Omega},$ the random closed set $\omega\rightarrow \mathcal{A}(t,\omega)$ satisfying the following properties:
\par
(1) $\mathcal{A}(t,\omega)$ is a nonempty compact subset of $X,$
\par
(2) $\mathcal{A}(t,\omega)$ is the minimal closed attracting set,
i.e., if $\tilde{\mathcal{A}}(t,\omega)$ is another closed attracting set, then $\mathcal{A}(t,\omega)\subset \tilde{\mathcal{A}}(t,\omega),$
\par
(3) it is invariant,  in the sense that, for all $s\leq t,$
 \[
 S (t,s;\omega)\mathcal{A}(s,\omega)=\mathcal{A}(t,\omega),
 \]
$\mathcal{A}(t,\omega)$  is called the random attractor.
\end{definition}
Let
$$\mathcal{A}(\omega)=\mathcal{A}(0,\omega). $$
Then the invariance property writes
$$S(t,s;\omega)\mathcal{A}(\vartheta_{s}\omega)=\mathcal{A }(\vartheta_{t}\omega).$$
We will prove the existence of the random attractor using Theorem 2.2 in $\cite{CDF} $.
For the convenience of reference, we cite it here.
\begin{theorem}
Let $(S(t, s; \omega))_{t\geq s, \omega\in \tilde{\Omega}}$ be a stochastic dynamical system
satisfying $\mathrm{(i)}, \mathrm{(ii)}, \mathrm{(iii)}$ and $\mathrm{(iv)}$. Assume that there exists a group $ \vartheta_{t}, t\in \mathbb{R},$ of measure preserving mappings such that condition $(c)$ holds and that, for $\tilde{\mathbb{P}}$-a.e.\ $\omega,$ there exists a compact attracting set $K(\omega)$ at time $0.$ For $\tilde{\mathbb{P}} $-a.e.\ $\omega,$ we set
$$\mathcal{A}(\omega)=\overline{\bigcup_{B\subset X}\mathcal{A}(B,\omega) } $$
where the union is taken over all the bounded subsets of $X$. Then we have for $\tilde{\mathbb{P}}  $-a.e.\ $\omega\in \tilde{\Omega}.$
\par
(1)\ $\mathcal{A}(\omega)$ is a nonempty compact subset of $X$, and if $X$ is connected,
it is a connected subset of $K(\omega)$.
\par
(2)\ The family $\mathcal{A}(\omega),\ \omega\in \Omega$, is measurable.
\par
(3)\ $\mathcal{A}(\omega)$ is invariant in the sense that
$$S(t,s;\omega)\mathcal{A}(\vartheta_{s}\omega)= \mathcal{A}(\vartheta_{t}\omega),\ \ s\leq t.$$
\par
(4)\ It attracts all bounded sets from $-\infty$: for bounded $B\subset X$ and $\omega\in \tilde{\Omega}$
\begin{eqnarray*}
d(S(t,s;\omega)B, \mathcal{A}(\vartheta_{t}\omega))\rightarrow 0,\ \ when\ s\rightarrow -\infty.
\end{eqnarray*}
Moreover, it is the minimal closed set with this property: if $\tilde{\mathcal{A}}(\vartheta_{t}\omega)$ is a closed attracting set, then $\mathcal{A}(\vartheta_{t}\omega)\subset \tilde{\mathcal{A}}(\vartheta_{t}\omega).$
\par
(5)\ For any bounded set $ B\subset X,\ d(S(t,s;\omega)B, \mathcal{A}(\vartheta_{t}\omega))\rightarrow 0$ in probability when $t\rightarrow \infty.$\\
And if the time shift $\vartheta_{t},t\in \mathbb{R}$ is ergodic
\par
(6)\ there exists a bounded set $B\subset X$ such that
\begin{eqnarray*}
\mathcal{A}(\omega)= \mathcal{A}(B, \omega).
\end{eqnarray*}
\par
(7)\ $ \mathcal{A}(\omega)$ is the largest compact measurable set which is invariant in sense of Definition $2.5.$
\end{theorem}
Before showing the existence of random attractor, we will cite Aubin-Lions Lemma which is vital to prove our main result Theorem 3.1 in this section.
\begin{lemma}
Let $B_{0}, B, B_{1}$ be Banach spaces such that $B_{0}, B_{1}$ are reflexive and $B_{0}\overset{c}{\subset}B\subset B_{1}.$ Define,
for $0<T<\infty,$
\begin{eqnarray*}
X:=\Big{\{}h \Big{|}h\in L^{2}([0,T]; B_{0}), \frac{dh}{dt}\in L^{2}([0,T]; B_{1})\Big{\}}.
\end{eqnarray*}
Then $X$ is a Banach space equipped with the norm $|h|_{L^{2}([0,T]; B_{0})}+|h'|_{L^{2}([0,T]; B_{1})}.$ Moreover, $$X \overset{c}\subset {L^{2}([0,T]; B)}.$$
\end{lemma}
The following Lemma, a special case of a general result of Lions and Magenes $\cite{LM}$,  will help us to show the continuity of the solution to stochastic PEs with respect to time in $(H^{1}(\mho))^{3}.$ For the proof of the Lemma we can see $\cite{T2}.$
\begin{lemma}
Let $V , H, V'$ be three Hilbert spaces such that $V \subset H = H¡ä\subset V' $
, where$ H'$ and $V'$ are the dual spaces of $H$ and $V$ respectively. Suppose $u \in
L^{2}(0, T; V )$ and $u'\in L^{2}(0, T; V')$. Then $u$ is almost everywhere equal to a function
continuous from $[0, T]$ into $H$.
\end{lemma}

\section{Main Result}
The aim of the paper is to prove:
\begin{theorem}
Let $Q\in L^{2}(\mho), \upsilon_{0}\in \mathcal{V}_{1}, T_{0}\in \mathcal{V}_{2} , \alpha>\frac{1}{8}$ in $(1.4).$ Assume conditions $(2.12)$ and $(2.15)$
hold. Then the solution operator $(S(t,s;\omega))_{t\geq s,\omega\in \tilde{\Omega}} $ of 3D stochastic PEs $(1.1)-(1.5): S(t,s;\omega)(\upsilon_{s}, T_{s})=(\upsilon(t), T(t) ) $ has properties $\mathrm{(i)}-\mathrm{(iv)}$ of Theorem 2.2 and possesses a compact absorbing ball $\mathcal{B}(0,\omega)$ in $\mathcal{V}$ at time $0.$  Furthermore, for $\tilde{\mathbb{P}}$-a.e. $\omega,$ the set
$$\mathcal{A}(\omega)=\overline{\bigcup_{B\subset \mathcal{V}}\mathcal{A}(B,\omega) } $$
where the union is taken over all the bounded subsets of $\mathcal{V},$ is the random attractor of stochastic PEs $(1.1)-(1.5)$ and possesses the properties $(1)-(7)$ of Theorem $2.2$ with space $X$ replaced by space $\mathcal{V}.$
\end{theorem}

To prove this theorem, we need the
following result concerning global well-posedness of strong solution to stochastic PEs. The regularity of the strong solution is key to prove the Theorem 3.1.
\begin{theorem}
Let $Q\in L^{2}(\mho), \upsilon_{0}\in \mathcal{V}_{1}, T_{0}\in \mathcal{V}_{2} , \mathcal{T}>0.$ Assume conditions $(2.12)$ and $(2.15)$
hold. Then there exists a unique strong solution $(\upsilon, T )$ of the system $(1.1)-(1.5)$ or equivalently $(u, \theta)$ of the system $(2.17)-(2.22)$ on the interval $[0, \mathcal{T}]$
which is Lipschitz continuous  with respect to the initial data and the noises in $\mathcal{V}$ and $C([0,T]; \mathcal{V})$ respectively.
\end{theorem}
\begin{Rem}\ \\
 $\mathbf{(1)}$
 Compared with regularity of strong solution in $\cite{GH},$  we improve the regularity of the strong solution by proving the continuity of strong solution with respect to initial data in $(H^{1}(\mho))^{3}$. This is key to prove the compact property of the solution operator in $\mathcal{V}$ and construct the compact absorbing ball in $\mathcal{V}$. Notice that $\cite{GH}$ only proved the strong solution is Lipschitz continuous in the space $( L^{2}(\mho))^{3}$ with respect to the initial data but this is not enough to study the asymptotical behavior in $(H^{1}(\mho))^{3}$ considered  here.\\
$\mathbf{(2)}$ With the help of Lemma 2.2, we prove the the continuity of strong solution with respect to time in $\mathcal{V}$ and obtain $a$ $priori$ estimates to prove the compact property of the solution operator in $\mathcal{V}$.\\
$\mathbf{(3)}$ We reduce the regularity of $Q$ from $H^{1}(\mho)$ to $L^{2}(\mho)$ which is more natural. There is some gap between $H^{1}(\mho)$ and $L^{2}(\mho).$ \\
$\mathbf{(4)}$ Due to the radiation properties of the clouds, we should consider the effect of the noise on the equations of  the heat conduction.  Therefore,  the stochastic PEs where both horizontal momentum equations and heat conduction equations are disturbed by the fractional noises are considered. Since the authors of $\cite{GH}$ didn't consider the effect of the noises on the heat conduct equation, our model is more complicated than the previous one. There are more challenges for the present model, especially when we consider the existence of random attractor.\\
$\mathbf{(5)}$ We simplify the proof of the the global well-posedness of stochastic PEs in $\cite{GH}$. By making a delicate  and careful  argument of the $a$ $priori$ estimates in $(L^{4}(\mho))^{3},$ we find the $a$ $priori$  estimates in $(L^{3}(\mho))^{3}$ is not necessary.\\
$\mathbf{(6)}$Lastly, we establish the regularity of the strong solution with respect to the fractional noises. We know this property is important to the ergodicity of stochastic PEs with Wiener noises.
\end{Rem}

\section{Proof Of Main Result}
The aim of this section is to prove the Theorem 3.1. Since we need the regularity of the strong solution and the $a$ $priori$ estimates to prove the compact property of solution operator, we should firstly prove Theorem 3.2. Then we study the growth property and the moment estimates of O-U process driven by fractional noises. The proof of the  Theorem 3.1 is completed in the last of this section.

\noindent ${4.1.\ \mathbf{Proof}\ \mathbf{Of}\  \mathbf{Theorem}\  \mathbf{3.2}. } $ Before giving our proof, we should notify that the global well-posedness of $(2.17)-(2.22)$ is equivalent to $(5.86)-(5.91).$
\par
In the following, we will complete our proof of the global well-posedness of stochastic PEs by three steps. Firstly, we will prove the global existence of strong solution. Then, we will show that the solution is continuous in the space $\mathcal{V}$ with respect to $t$. Finally,  we will obtain the continuity in $\mathcal{V}$ with respect to the initial data.  \\
$\mathbf{Step}\ 1$: We prove the global existence of strong solution.\\
In the appendix, we prove the local existence of the strong solution to stochastic PEs and obtain the $a$ $priori$ estimates in $\mathcal{V}.$   As we have indicated in the appendix that $[0, \tau_{*})$ is the maximal interval of existence of the solution of $(5.86)-(5.91),$ we infer that $\tau_{*}=\infty,$ a.s..
Otherwise, if there exists $A\in \mathcal{F}$ such that $\mathbb{P}(A)>0$ and for fixed $\omega\in A, \tau_{*}(\omega)<\infty,$ it is clear that
\begin{eqnarray*}
\limsup\limits_{t\rightarrow \tau_{*}^{-}(\omega)}(\| u(t)\|_{1}+\|\theta(t)\|_{1} )=\infty,\ \ \mathrm{for}\ \mathrm{any}\ \omega\in A,
\end{eqnarray*}
which contradicts the priori estimates $(5.134), (5.138)$
and $(5.140).$ Therefore $\tau_{*}=\infty,$ a.s.,  and the strong solution $(u, \theta)$ exists globally in time a.s..\\
$\mathbf{Step}\ 2$: We show the continuity of strong solutions with respect to $t$.\\
Multiplying $(5.86)$ by $\eta\in \mathcal{V}_{1}$, integrating with respect to space variable, yields
\begin{eqnarray*}
&&\langle \partial_{t}A_{1}^{\frac{1}{2}} u, \eta \rangle= \langle \partial_{t}u, A_{1}^{\frac{1}{2}}  \eta \rangle=- \langle A_{1} u, A_{1}^{\frac{1}{2}}\eta\rangle - \langle[(u+Z_{1})\cdot\nabla ](u+Z_{1}),A_{1}^{\frac{1}{2}}\eta\rangle\nonumber\\
&& -\langle \varphi(u+Z_{1}) \partial_{z}(u+Z_{1}), A_{1}^{\frac{1}{2}}  \eta \rangle-\langle  f( u+Z_{1})^{\bot}, A_{1}^{\frac{1}{2}}  \eta \rangle  \nonumber\\
&& + \langle\int_{-1}^{z}\nabla \theta dz',   A_{1}^{\frac{1}{2}}  \eta \rangle+ \langle\beta Z_{1}, A_{1}^{\frac{1}{2}}  \eta \rangle,
\end{eqnarray*}
where we have used $ \langle \nabla p_{s},  A_{1}^{\frac{1}{2}}  \eta \rangle=0$ which follows by integration by parts formula.
Taking a similar argument in  $(5.136)$, we get
\begin{eqnarray*}
\langle \varphi(u+Z_{1}) \partial_{z}(u+Z_{1}), A_{1}^{\frac{1}{2}}  \eta \rangle\leq \|u+Z_{1}\|_{1}\|u+Z_{1}\|_{2}|A_{1}^{\frac{1}{2}} \eta|_{2}.
\end{eqnarray*}
By H$\ddot{o}$lder inequality and Sobolev embedding theorem, we have
\begin{eqnarray*}
\|\partial_{t}(A_{1}^{\frac{1}{2}}u)\|_{\mathcal{V}_{1}'}\leq C(\|u\|_{2}+\|u+Z_{1}\|_{1}\|u+Z_{1}\|_{2}+|u|_{2}+|Z_{1}|_{2}+|\nabla \theta|_{2} ).
\end{eqnarray*}
In view of Proposition 2.1  and the following results about the regularity of $u$
\begin{eqnarray*}
u\in L^{\infty}([0, T]; \mathcal{V}_{1})\cap L^{2}([0, T]; (H^{2}(\mho))^{2}), \ \ \ \ Z_{1}\in C([0,T]; (H^{3}(\mho))^{2}),\ \ \forall T>0,
\end{eqnarray*}
we obtain
\begin{eqnarray*}
A_{1}^{\frac{1}{2}}u\in L^{2}([0,T]; \mathcal{V}_{1} ),\ \ \ \partial_{t}(A_{1}^{\frac{1}{2}}u )\in L^{2}([0,T]; \mathcal{V}_{1}' ),\ \ \forall T>0,
\end{eqnarray*}
which together with Lemma $2.2$ implies $$A_{1}^{\frac{1}{2}}u\in C([0,T]; H_{1})\ \mathrm{or}\  u\in C([0,T]; \mathcal{V}_{1}).  a.s.. $$ To study the regularity of $\theta,$ we choose $\xi\in \mathcal{V}_{2}$. By $(5.87)$ we have
\begin{eqnarray*}
\langle \partial_{t}A_{2}^{\frac{1}{2}} \theta ,      \xi \rangle& =&  \langle \partial_{t} \theta ,    A_{2}^{\frac{1}{2}}  \xi \rangle
=\langle  A_{2}\theta,   A_{2}^{\frac{1}{2}} \xi  \rangle- \langle (u+Z_{1})\cdot \nabla (\theta+Z_{2}), A_{2}^{\frac{1}{2}}\xi     \rangle  \nonumber\\
&+&\langle \varphi(u+Z_{1})\partial_{z}(\theta+Z_{2}),      A_{2}^{\frac{1}{2}}\xi   \rangle +\langle Q,    A_{2}^{\frac{1}{2}}\xi \rangle
+\beta \langle Z_{2},   A_{2}^{\frac{1}{2}}\xi   \rangle.
\end{eqnarray*}
Taking a similar argument as above, we get
\begin{eqnarray*}
\langle \varphi(u+Z_{1})\partial_{z}(\theta+Z_{2}),      A_{2}^{\frac{1}{2}}\xi   \rangle \leq C\|u+Z_{1}\|_{1}^{\frac{1}{2}}\|u+Z_{1}\|_{2}^{\frac{1}{2}}
\|\theta+Z_{2}\|_{1}^{\frac{1}{2}}\|\theta+Z_{2}\|_{2}^{\frac{1}{2}}| A_{2}^{\frac{1}{2}}\xi |_{2}.
\end{eqnarray*}
Then by H$\ddot{o}$lder inequality and Sobolev embedding theorem, we arrive at
\begin{eqnarray*}
|\partial_{t} A_{2}^{\frac{1}{2}} \theta|_{\mathcal{V}_{2}'}&\leq& C(|A_{2}\theta|_{2}+\|u+Z_{1}\|_{1}\|\theta+Z_{2}\|_{2}\nonumber\\
&&+ \|u+Z_{1}\|_{1}^{\frac{1}{2}}\|u+Z_{1}\|_{2}^{\frac{1}{2}}
\|\theta+Z_{2}\|_{1}^{\frac{1}{2}}\|\theta+Z_{2}\|_{2}^{\frac{1}{2}} +|\theta|_{2}+|Z_{2}|_{2}).
\end{eqnarray*}
By step one and Proposition 2.2, we have also proved that
\begin{eqnarray*}
\theta\in L^{\infty}([0, T]; \mathcal{V}_{2})\cap L^{2}([0, T]; H^{2}(\mho)), \ \ \ \ Z_{2}\in C([0,T]; H^{3}(\mho)),\ \ \forall T>0.
\end{eqnarray*}
Therefore, by the above argument we have
\begin{eqnarray*}
A_{1}^{\frac{1}{2}}\theta\in L^{2}([0,T]; \mathcal{V}_{2} ),\ \ \ \partial_{t}(A_{1}^{\frac{1}{2}}\theta )\in L^{2}([0,T]; \mathcal{V}_{1}' ),\ \ \forall T>0,
\end{eqnarray*}
which together with Lemma $2.2$ implies
\begin{eqnarray*}
A_{2}^{\frac{1}{2}}\theta\in C([0,T]; H_{2})\  \mathrm{or}\  \theta\in C([0,T]; \mathcal{V}_{2}),\ \  a.s..
\end{eqnarray*}
$\mathbf{Step}\ 3$: We obtain the continuity in $\mathcal{V}$ with respect to the initial data.\\
In order to show the uniqueness of the solutions. Let $(\upsilon_{1}, T_{1} )$ and $(\upsilon_{2}, T_{2}) $ be two solutions of the system
$(5.86)-(5.91)$ with corresponding pressure $p_{b}{'}$ and $p_{b}{''},$ and initial data $((\upsilon_{0})_{1}, (T_{0})_{1})$ and
$((\upsilon_{0})_{2}, (T_{0})_{2}),$ respectively. Denote by $v=\upsilon_{1}-\upsilon_{2}, p_{b}=p_{b}{'}-p_{b}{''}$ and $\mathbb{T}=T_{1}-T_{2}.$ Then
we have
\begin{eqnarray}
&&\partial_{t}v-\Delta v-\partial_{zz} v+[(\upsilon_{1}+Z_{1} )\cdot \nabla ]v
+(v\cdot \nabla)(\upsilon_{2}+Z_{1} )\nonumber\\
&&+\varphi( \upsilon_{1}+Z_{1})v_{z}+\varphi(v)\partial_{ z}(\upsilon_{2}+Z_{1}) +fk\times v+\nabla p_{b}
-\int_{-1}^{z}\nabla \mathbb{T}dz'=0,\\
&&\partial_{t}\mathbb{T}-\Delta \mathbb{T}-\partial_{zz}\mathbb{T}+[(\upsilon_{1}+Z_{1} )\cdot \nabla  ]\mathbb{T}
+(v\cdot \nabla)(T_{2}+Z_{2})\nonumber\\
&&+\varphi (\upsilon_{1}+Z_{1} )\mathbb{T}_{z}+\varphi (v) \partial_{z}(T_{2}+Z_{2}) =0,\\
&&\int_{-1}^{0}\nabla\cdot vdz=0,\\
&&v(x,y,z,0)=(\upsilon_{0})_{1}-(\upsilon_{0})_{2},\ \ \ \ \mathbb{T}(x,y,z,0)=(T_{0})_{1}-(T_{0})_{2},\\
&&(v,\mathbb{T}) \mathrm{satisfies}\ \mathrm{the}\  \mathrm{boundary}\  \mathrm{value}\  \mathrm{conditions}\ (5.89)-(5.90).
\end{eqnarray}
Recall the notations $L_{i}$ in the introduction with $i=1,2.$ Multiplying $L_{1}v$ in equation $(4.23)$ and integrating with respect to spatial variable yields,
\begin{eqnarray}
&&\frac{1}{2}\partial_{t} (|\nabla v|_{2}^{2}+|\partial_{z} v|_{2}^{2})+| \Delta v|_{2}^{2}+|\partial_{z} v|_{2}^{2}+|\nabla v_{z}|_{2}^{2}\nonumber\\
&=&-\int_{\mho}\{[(\upsilon_{1}+Z_{1}  )\cdot \nabla ]v  \}\cdot L_{1}v -\int_{\mho}\varphi(\upsilon_{1}+Z_{1}  ) v_{z}\cdot L_{1}v\nonumber\\
&&-\int_{\mho}[\varphi(v)\partial_{z} (\upsilon_{2}+Z_{1} ) ]\cdot L_{1}v -\int_{\mho}[(v\cdot\nabla )(\upsilon_{2}+Z_{1} )]\cdot L_{1}v \nonumber\\
&& - \int_{\mho}(fk\times v)\cdot L_{1}v
+\int_{\mho}(\int_{-1}^{z}\nabla \mathbb{T} dz' )\cdot L_{1}v\nonumber\\
&=&\Sigma_{i=1}^{6}K_{i}.
\end{eqnarray}
To estimate $K_{1},$ by the Agmon inequality and H$\ddot{o}$lder inequality, we obtain
\begin{eqnarray*}
K_{1}&\leq& |v_{1}+Z_{1}|_{\infty}|\nabla v|_{2}|L_{1}v|_{2}\nonumber\\
&\leq& C\|v_{1}+Z_{1}\|_{1}^{\frac{1}{2}}\|v_{1}+Z_{1}\|_{2}^{\frac{1}{2}}\|v\|_{1} \|v\|_{2}\nonumber\\
&\leq& C\|v_{1}+Z_{1}\|_{1}\|v_{1}+Z_{1}\|_{2}\|v\|_{1}^{2}+\varepsilon \|v\|_{2}^{2}.
\end{eqnarray*}
Then, using H$\ddot{o}$lder inequality, interpolation inequality and Sobolev embedding theorem,  we get the estimate of $K_{2},$
\begin{eqnarray*}
K_{2}&\leq&\int_{\mho}|\int_{-1}^{0}(\nabla \cdot v_{1}+ \nabla \cdot Z_{1} )dz |\cdot |v_{z}|\cdot|L_{1}v |\nonumber\\
&\leq & |\nabla \cdot \bar{v}_{1}+\nabla \cdot \bar{Z}_{1}  |_{L^{4}(M)}\int_{-1}^{0}|v_{z}|_{L^{4}(M)}|L_{1}v|_{L^{2}(M)}\nonumber\\
&\leq & C(\|v_{1}\|_{1}^{\frac{1}{2}}+\|Z_{1}\|_{1}^{\frac{1}{2}})(\|v_{1}\|_{2}^{\frac{1}{2}}+\|Z_{1}\|_{2}^{\frac{1}{2}})\|v\|_{1}^{\frac{1}{2}}
\|v\|_{2}^{\frac{1}{2}}\|v\|_{2}\nonumber\\
&\leq &\varepsilon \|v\|_{2}^{2}+C(\|v_{1}\|_{1}^{2}+\|Z_{1}\|_{1}^{2})(\|v_{1}\|_{2}^{2}+\|Z_{1}\|_{2}^{2})\|v\|_{1}^{2}.
\end{eqnarray*}
By an analogous argument as above, we infer that
\begin{eqnarray*}
K_{3}&\leq& \int_{\mho} |\int_{-1}^{0}\nabla\cdot v dz|\cdot|\partial_{z}( v_{2}+Z_{1}) |\cdot |L_{1}v|\nonumber\\
&\leq&|\nabla \cdot \bar{v} |_{L^{4}(M)}\int_{-1}^{0}|\partial_{z}(v_{2}+Z_{1})|_{L^{4}(M)}|L_{1}v|_{L^{2}(M)}\nonumber\\
&\leq&C\|v\|_{1}^{\frac{1}{2}}\|v\|_{2}^{\frac{1}{2}}\int_{-1}^{0}\|v_{2}+Z_{1}\|_{H^{1}(M)}^{\frac{1}{2}}\|v_{2}+Z_{1}\|_{H^{2}(M)}^{\frac{1}{2}}\|v\|_{2}dz\nonumber\\
&\leq&\varepsilon \|v\|_{2}^{2}+C\|v\|_{1}^{2}\|v_{2}+Z_{1}\|_{1}^{2}\|v_{2}+Z_{1}\|_{2}^{2}.
\end{eqnarray*}
By virtue of H$\ddot{o}$lder inequality,
\begin{eqnarray*}
K_{4}&\leq& |v|_{4}(|\nabla v_{2} |_{4}+|\nabla Z_{1}|_{4}  )|L_{1}v|_{2}\leq \varepsilon \|v\|_{2}^{2}+C\|v\|_{1}^{2}(\|v_{2}\|_{2}^{2}+\|Z_{1}\|_{2}^{2} ).
\end{eqnarray*}
After some basic calculus, we obtain the estimates of $K_{5}$ and $K_{6}$ as the following,
\begin{eqnarray*}
K_{5}+K_{6}&\leq& \varepsilon  \|v\|_{2}^{2}+ C|v|_{2}^{2}+C\|\mathbb{T}\|_{1}^{2}.
\end{eqnarray*}
By the boundary conditions, $|\mathbb{T}|_{2}^{2}$ is equivalent to $|\mathbb{T}_{z}|_{2}^{2}+|\mathbb{T}(z=0)|_{L^2(M)}^{2},$ then
$\|\mathbb{T}\|_{1}^{2}$ is equivalent to $|\nabla \mathbb{T}|_{2}^{2}+|\mathbb{T}_{z}|_{2}^{2}+|\mathbb{T}(z=0)|_{L^2(M)}^{2}.$ Keeping this in mind and taking an inner product of the equation $(4.24)$ with $L_{2}\mathbb{T}$, we have
\begin{eqnarray}
&&\frac{1}{2}\partial_{t}( |\nabla \mathbb{T}|_{2}^{2}+|\mathbb{T}_{z}|_{2}^{2}+\alpha|\mathbb{T}(z=0)|_{2}^{2} )
+ |\Delta \mathbb{T} |_{2}^{2}+|\mathbb{T}_{zz}|_{2}^{2}+|\nabla \mathbb{T}_{z}|_{2}^{2}+\alpha|\nabla \mathbb{T}(z=0) |_{2}^{2}\nonumber\\
&=&-\int_{\mho}[(\upsilon_{1}+Z_{1} )\cdot \nabla  \mathbb{T}]L_{2}\mathbb{T}
-\int_{\mho}[v\cdot \nabla (T_{2}+Z_{2})]L_{2}\mathbb{T}\nonumber\\
 &&-\int_{\mho}\varphi (\upsilon_{1}+Z_{1} )\mathbb{T}_{z}L_{2}\mathbb{T}+\varphi (v)\partial_{z}( T_{2} +Z_{2})L_{2}\mathbb{T}
 :=\sum_{i=1}^{4}J_{i}.
\end{eqnarray}
By Agmon inequality, we get that
\begin{eqnarray*}
J_{1}+J_{2}&\leq& |v_{1}+Z_{1}|_{\infty}|\nabla \mathbb{T}|_{2}|L_{2}\mathbb{T}|_{2}+|v|_{4}|\nabla (T_{2}+Z_{2})|_{4}|L_{2}\mathbb{T}|_{2}\nonumber\\
&\leq&C\|v_{1}+Z_{1}\|_{1}^{\frac{1}{2}}\|v_{1}+Z_{1}\|_{2}^{\frac{1}{2}}\|\mathbb{T}\|_{1}\|\mathbb{T}\|_{2}+C \|v\|_{1}\|T_{2}+Z_{2}\|_{2}\|\mathbb{T}\|_{2}\nonumber\\
&\leq& \varepsilon \|\mathbb{T}\|_{2}^{2}+C\|v\|_{1}^{2}\|T_{2}+Z_{2}\|_{2}^{2}+C\|v_{1}+Z_{1}\|_{1}\|v_{1}+Z_{1}\|_{2}\|\mathbb{T}\|_{1}^{2}.
\end{eqnarray*}
Taking an similar argument as $K_{2},$ we obtain
\begin{eqnarray*}
J_{3}&\leq&\int_{\mho} (\int_{-1}^{0}|\nabla \cdot (v_{1}+Z_{1} )|dz)\cdot |\mathbb{T}_{z}|\cdot |L_{2}\mathbb{T}| \nonumber\\
&\leq& \Big{(}\int_{-1}^{0}|L_{2}\mathbb{T}|_{L^{2}(M)}|\mathbb{T}_{z}|_{L^{4}(M)}dz \Big{)}|\int_{-1}^{0}|\nabla \cdot (v_{1}+Z_{1} )|dz|_{L^{4}(M)}\nonumber\\
&\leq&C\int_{-1}^{0}\|\mathbb{T}\|_{H^{2}(M)}^{\frac{3}{2}}\|\mathbb{T}\|_{H^{1}(M)}^{\frac{1}{2}}dz\int_{-1}^{0}|\nabla \cdot (v_{1}+Z_{1} )|_{L^{4}(M)}dz\nonumber\\
&\leq&C\|\mathbb{T}\|_{2}^{\frac{3}{2}}\|\mathbb{T}\|_{1}^{\frac{1}{2}}\int_{-1}^{0}\|v_{1}+Z_{1}\|_{H^{1}(M)}^{\frac{1}{2}}\|v_{1}+Z_{1}\|_{H^{2}(M)}^{\frac{1}{2}}dz\nonumber\\
&\leq&C\|\mathbb{T}\|_{2}^{\frac{3}{2}}\|\mathbb{T}\|_{1}^{\frac{1}{2}}\|v_{1}+Z_{1}\|_{1}^{\frac{1}{2}}\|v_{1}+Z_{1}\|_{2}^{\frac{1}{2}}\nonumber\\
&\leq&\varepsilon \|\mathbb{T}\|_{2}^{2}+C\|\mathbb{T}\|_{1}^{2}\|v_{1}+Z_{1}\|_{1}^{2}\|v_{1}+Z_{1}\|_{2}^{2}.
\end{eqnarray*}
Follow the similar steps, we can prove that
\begin{eqnarray*}
J_{4}&\leq& \int_{\mho}(\int_{-1}^{0}|\nabla \cdot v|)|\partial_{z}(T_{2}+Z_{2})  |\cdot|L_{2}\mathbb{T} |\nonumber\\
&\leq&\Big{(} \int_{-1}^{0} |L_{2}\mathbb{T}|_{L^{2}(M)}|\partial_{z}(T_{2}+Z_{2}) |_{L^{4}(M)}dz\Big{)}|\int_{-1}^{0}|\nabla \cdot v|dz |_{L^{4}(M)}\nonumber\\
&\leq&\Big{(} \int_{-1}^{0}\|\mathbb{T}\|_{H^{2}(M)}\|T_{2}+Z_{2}\|_{H^{1}(M)}^{\frac{1}{2}}\|T_{2}+Z_{2}\|_{H^{2}(M)}^{\frac{1}{2}}dz\Big{)}\int_{-1}^{0}|\nabla \cdot v|_{L^{4}(M)}dz\nonumber\\
&\leq& C\|\mathbb{T}\|_{2}\|T_{2}+Z_{2}\|_{1}^{\frac{1}{2}}\|T_{2}+Z_{2}\|_{2}^{\frac{1}{2}}\|v\|_{1}^{\frac{1}{2}}\|v\|_{2}^{\frac{1}{2}}\nonumber\\
&\leq& \varepsilon \|\mathbb{T}\|_{2}^{2}+\varepsilon \|v\|_{2}^{2}+C\|T_{2}+Z_{2}\|_{1}^{2}\|T_{2}+Z_{2}\|_{2}^{2}\|v\|_{1}^{2}.
\end{eqnarray*}
Denote
\begin{eqnarray*}
\eta(t)&:=& \|v(t)\|_{1}^{2}+\|\mathbb{T}(t)\|_{1}^{2},\\
\xi(t)&:=&(\|v_{1}\|_{1}+\|Z_{1}\|_{1})(\|v_{1}\|_{2}+\|Z_{1}\|_{2} )+(\|v_{1}\|_{1}^{2}+\|Z_{1}\|_{1}^{2} )(\|v_{1}\|_{2}^{2}+\|Z_{1}\|_{2}^{2} )\\
&&+(\|v_{2}\|_{1}^{2}+\|Z_{1}\|_{1}^{2} )(\|v_{2}\|_{2}^{2}+\|Z_{1}\|_{2}^{2} )+( \|v_{2}\|_{2}^{2}+\|Z_{1}\|_{2}^{2} )+\|T_{2}\|_{2}^{2}+\|T_{2}+Z_{2}\|_{1}^{2} \|T_{2}+Z_{2}\|_{2}^{2}  +1
\end{eqnarray*}
Since $|\nabla v|_{2}^{2}+|\partial_{z} v|_{2}^{2}$ is equivalent to $\|v\|_{1}^{2}$ and $|\nabla \mathbb{T}|_{2}^{2}+|\mathbb{T}_{z}|_{2}^{2}+|\mathbb{T}(z=0)|_{L^2(M)}^{2}$ is equivalent to  $\|\mathbb{T}\|_{1}^{2},$  letting $\varepsilon$ be small enough,  by $(4.28)-(4.29)$ and estimates of $K_{1}-K_{6}$ and $J_{1}-J_{4}$ we have
\begin{eqnarray}
\frac{d \eta(t)}{dt}+\|v\|_{2}^{2}+\|\mathbb{T}\|_{2}^{2}\leq \eta(t)\xi(t).
\end{eqnarray}
Since $(v_{i}(t),T_{i}(t)),i=1,2,$ is the solution of stochastic PEs in sense of Definition $2.3$ which ensure that
\begin{eqnarray*}
\int_{0}^{t}\xi(s)ds< \infty,\ \ a.s.,\ \mathrm{for}\ \mathrm{all}\ t\in(0, \infty),
\end{eqnarray*}
we conclude that
\begin{eqnarray*}
\eta(t)\leq \eta(0)e^{C\int_{0}^{t}\xi(s)ds}.
\end{eqnarray*}
Therefore, we proved that for any $t\in(0, \infty) , (u(t), \theta(t))$ is Lipschitz continuous in $\mathcal{V}_{1}\times \mathcal{V}_{2}$ with respect to
the initial data $(u(0), \theta(0)),$ which  is equivalent to that strong solution $ (\upsilon(t), T(t))$ of $(1.1)-(1.5)$ is Lipschitz continuous in $\mathcal{V}_{1}\times \mathcal{V}_{2}$ with respect to the initial data      $(v_{0}, T_{0}),$ for any $t\in (0, \infty).$ Following an analogous argument, we can also show that $ (\upsilon(t), T(t))_{t\in [0,\tau]}$ is Lipschitz continuous with respect to the noises in $C([0,\tau]; \mathcal{V})$ with the norm $\sup\limits_{t\in [0,\tau]}\|\cdot\|_{1}$ for arbitrary $\tau>0.$
\hspace{\fill}$\square$
\par
\noindent ${4.2.\ \mathbf{Growth }\ \mathbf{Property }\ \mathbf{And}\ \mathbf{Moment}\ \mathbf{Estimates}\  \mathbf{Of}\ \ \mathbf{O-U}\  \mathbf{Processes}. } $
\par
In order to obtain the existence of random attractor for system $(1.1)-(1.5),$ we need the following property of Ornstein-Uhlenbeck process driven by fBm.
First of all, we extend the definition of generalized Stieltjes integral $(2.6)$ to infinite interval. For arbitrary $a, b\in \mathbb{R}$ and $ a<b , $
we assume $f\in W^{\alpha,1}([a,b]; \mathbb{R})$( see section 2) and $g$ satisfies
\begin{eqnarray*}
C_{\alpha}(g)|_{a}^{b}&=&\frac{1}{\Gamma(1-\alpha)\Gamma(\alpha)}\sup\limits_{a<s_{1}<s_{2}<b}\Big{(}\frac{|g(s_{1})-g(s_{2})|}
{(s_{2}-s_{1})^{1-\alpha}}+\int_{s_{1}}^{s_{2}}\frac{|g(u)-g(s_{1})|}{(u-s_{1})^{2-\alpha}}du   \Big{)}\\
&=&\frac{1}{\Gamma(1-\alpha)\Gamma(\alpha)}\sup\limits_{a<s_{1}<s_{2}<b}|D^{1-\alpha}_{s_{2}-}g_{s_{2}-}(s_{1}) |<\infty.
\end{eqnarray*}
Then, for $t\in \mathbb{R}$, we define
\begin{eqnarray*}
\int^{t}_{-\infty}fdg=\lim\limits_{\tau\rightarrow -\infty}\int^{t}_{\tau}fdg
\end{eqnarray*}
provided the limit exists and $ \int^{t}_{\tau}fdg$ is defined as $(2.6)$ (or see $\cite{NR}$ and $\cite{Z}$).
\par
To study the long-term behavior of stochastic PEs,  we should consider the moment estimates and growth properties of O-U processes $(Z_{j}(t))_{t\in \mathbb{R}},j=1,2,$ where
\begin{eqnarray*}
Z_{j}(t)=\int_{-\infty}^{t}e^{-(t-s)(A_{j}+\beta)}dW^{H}_{j}(s).
\end{eqnarray*}
\begin{lemma}
Assume $(2.12)$ hold, then for any $\varepsilon>0$ and $m\geq2$, there exists positive constant $\beta$ depending only on $\varepsilon$ and $m$ such that
\begin{eqnarray}
E\|Z_{1}(0)\|_{3}^{m}<\varepsilon.
\end{eqnarray}
\end{lemma}
\begin{proof}
Due to $(2.12)$ and Theorem $2.5$ in $\cite{Z}$, we have
 for $t\in [-1,0]$
\begin{eqnarray}
\|Z_{1}(t)\|_{3}&\leq &\sum^{\infty}_{i=1}\sqrt{\lambda_{i,1}}\gamma_{i,1}^{\frac{3}{2}}| \int_{-\infty}^{t}e^{-(\gamma_{i,1}+\beta )(t-s)}dB^{H}_{i}(s) |\nonumber\\
&\leq& \sum^{\infty}_{i=1}\sqrt{\lambda_{i,1}}\gamma_{i,1}^{\frac{3}{2}}| \int_{t-1}^{t}e^{-(\gamma_{i,1}+\beta )(t-s)}dB^{H}_{i}(s) |\nonumber\\
&&+ \sum^{\infty}_{i=1}\sqrt{\lambda_{i,1}}\gamma_{i,1}^{\frac{3}{2}}| \int_{-\infty}^{t-1}e^{-(\gamma_{i,1}+\beta )(t-s)}dB^{H}_{i}(s) |\nonumber\\
&:=&I_{1}+I_{2}.
\end{eqnarray}
By the definition of stochastic calculus and inequality $(2.7),$ we  infer that
\begin{eqnarray}
I_{1}&\leq& \sum^{\infty}_{i=1}\sqrt{\lambda_{i,1}}\gamma_{i,1}^{\frac{3}{2}}\Big{(}C_{\alpha}(B_{i}^{H})|_{t-1}^{t}\Big{)}\Big{[}\int_{t-1}^{t}
\frac{e^{-(\gamma_{i,1}+\beta)(t-s)}}{[s-(t-1)]^{\alpha}}   ds\nonumber\\
&&+\int_{t-1}^{t}\int_{t-1}^{s}
\frac{e^{-(\gamma_{i,1}+\beta )(t-s)}-e^{-(\gamma_{i,1}+\beta )(t-u)}     }{(s-u)^{1+\alpha}}duds\Big{]}\nonumber\\
&\leq& \sum^{\infty}_{i=1}\sqrt{\lambda_{i,1}}\gamma_{i,1}^{\frac{3}{2}}\Big{(}C_{\alpha}(B_{i}^{H})|_{t-1}^{t}\Big{)} \Big{[} (\gamma_{i,1}+\beta )^{\alpha-1}
\int_{t-1}^{t}(t-s)^{\alpha-1}[s-(t-1)]^{\alpha}ds\nonumber\\
&&+
\int_{t-1}^{t}e^{-(\gamma_{i,1}+\beta )(t-s)}\int_{t-1}^{s}\frac{(\gamma_{i,1}+\beta )^{\frac{1}{2}}(s-u)^{\frac{1}{2}}}{(s-u)^{1+\alpha}}duds\Big{]}\nonumber\\
&\leq&C \sum^{\infty}_{i=1}\sqrt{\lambda_{i,1}}\gamma_{i,1}^{\frac{3}{2}}\Big{(}C_{\alpha}(B_{i}^{H})|_{t-1}^{t}\Big{)}[(\gamma_{i,1}+\beta )^{\alpha-1}
+(\gamma_{i,1}+\beta )^{-\frac{1}{2}}]\nonumber\\
&\leq&C \beta^{-\frac{1}{2}}\sum^{\infty}_{i=1}\sqrt{\lambda_{i,1}}\gamma_{i,1}^{\frac{3}{2}}C_{\alpha}(B_{i}^{H})|_{t-1}^{t}.
\end{eqnarray}
From $(4.33)$ and Minkovsky inequality, we reach
\begin{eqnarray}
(EI_{1}^{m})^{\frac{1}{m}}&\leq& C(EC^{m}_{\alpha}(B_{i}^{H})|_{t-1}^{t})^{\frac{1}{m}}\beta^{-\frac{1}{2}},
\end{eqnarray}
where he last two inequalities follow by the the facts that  $(B_{i}^{H})_{i\in \mathbb{N^{+}}}$ is a i.i.d. sequence and each fBm $ B_{i}^{H}$ has stationary increments which satisfy Lemma $7.5$ in $\cite{NR}.$
To estimate $I_{2},$ by $(4.32)$ we deduce that
\begin{eqnarray*}
I_{2}\leq\sum^{\infty}_{i=1}\sqrt{\lambda_{i,1}}\gamma_{i,1}^{\frac{3}{2}} \sum _{n=1}^{\infty}| \int_{(n+1)(t-1)}^{n(t-1)}e^{-(\gamma_{i,1}+\beta )(t-s)}dB^{H}_{i}(s) |.
\end{eqnarray*}
Since $ e^{-(\gamma_{i,1}+\beta )((t-1)-s)}$ has all derivatives of any order in any interval and $B^{H}_{i}$ is $\gamma$-H$\ddot{o}$lder continuous for all $\gamma <H,$
by the Theorem $4.2.1$ in $\cite{Z}$  stochastic calculus $ \int_{(n+1)(t-1)}^{n(t-1)}e^{-(\gamma_{i,1}+\beta )((t-1)-s)}dB^{H}_{i}(s)$ defined by $(2.6)$ is equal to Riemann-Stieltjes integral. Therefore
\begin{eqnarray}
I_{2}\leq\sum^{\infty}_{i=1}\sqrt{\lambda_{i,1}}\gamma_{i,1}^{\frac{3}{2}}e^{-(\gamma_{i,1}+\beta )} \sum _{n=1}^{\infty}| \int_{(n+1)(t-1)}^{n(t-1)}e^{-(\gamma_{i,1}+\beta )((t-1)-s)}dB^{H}_{i}(s) |.
\end{eqnarray}
For $n\geq2,$ by $(2.7)$ and elementary arguments we deduce
\begin{eqnarray}
&&|\int_{(n+1)(t-1)}^{n(t-1)}e^{-(\gamma_{i,1}+\beta )((t-1)-s)}dB^{H}_{i}(s)|\nonumber\\
&\leq& C_{\alpha}(B_{i}^{H})|_{(n+1)(t-1)}^{n(t-1)} \Big{(}\int_{(n+1)(t-1)}^{n(t-1)}\frac{e^{-(\gamma_{i,1}+\beta)(t-1-s)}}{(s-(n+1)(t-1))^{\alpha}}ds\nonumber\\
&&+\int_{(n+1)(t-1)}^{n(t-1)}\int_{(n+1)(t-1)}^{s}\frac{|e^{-(\gamma_{i,1}+\beta)(t-1-s)}-e^{-(\gamma_{i,1}+\beta)(t-1-u)} |}{(s-u)^{1+\alpha}}duds\Big{)}\nonumber\\
&\leq& C_{\alpha}(B_{i}^{H})|_{(n+1)(t-1)}^{n(t-1)} \Big{(}\int_{(n+1)(t-1)}^{n(t-1)}\frac{(\gamma_{i,1}+\beta)^{-2}(t-1-s)^{-2}}{(s-(n+1)(t-1))^{\alpha}}ds\nonumber\\
&&+\int_{(n+1)(t-1)}^{n(t-1)}e^{-(\gamma_{i,1}+\beta)(t-1-s)} \int_{(n+1)(t-1)}^{s} \frac{1- e^{-(\gamma_{i,1}+\beta)(s-u)}}{(s-u)^{1+\alpha}}duds\Big{)}\nonumber\\
&\leq& C_{\alpha}(B_{i}^{H})|_{(n+1)(t-1)}^{n(t-1)}\Big{(}(\gamma_{i,1}+\beta)^{-2}(n-1)^{-2}(t-1)^{-2}\frac{|t-1|^{1-\alpha}}{1-\alpha}\nonumber\\
&&+\int_{(n+1)(t-1)}^{n(t-1)}(\gamma_{i,1}+\beta)^{-2}(t-1-s)^{-2}\int_{(n+1)(t-1)}^{s}
\frac{(\gamma_{i,1}+\beta)^{\frac{1}{2}}(s-u)\frac{1}{2}}{(s-u)^{1+\alpha}}duds\Big{)}\nonumber\\
&\leq& C_{\alpha}(B_{i}^{H})|_{(n+1)(t-1)}^{n(t-1)}\Big{(} (\gamma_{i,1}+\beta)^{-2}(n-1)^{-2}(t-1)^{-2}|t-1|^{1-\alpha}\nonumber\\
&&+(\gamma_{i,1}+\beta)^{-\frac{3}{2}}(n-1)^{-2}(t-1)^{-2}|t-1|^{\frac{3}{2}-\alpha}\Big{)}\nonumber\\
&\leq& CC_{\alpha}(B_{i}^{H})|_{(n+1)(t-1)}^{n(t-1)}(\gamma_{i,1}+\beta)^{-\frac{3}{2}}(n-1)^{-2}.
\end{eqnarray}
When $n=1,$ using $(2.7)$ again we arrive at
\begin{eqnarray}
&&|\int_{(n+1)(t-1)}^{n(t-1)}e^{-(\gamma_{i,1}+\beta )((t-1)-s)}dB^{H}_{i}(s)|\nonumber\\
&\leq&  C_{\alpha}(B_{i}^{H})|_{2(t-1)}^{t-1}\Big{(} \int_{2(t-1)}^{t-1}\frac{e^{-(\gamma_{i,1}+\beta)(t-1-s)}}{(s-2(t-1))^{\alpha}}ds\nonumber\\
&&+\int_{2(t-1)}^{t-1}\int_{2(t-1)}^{s}\frac{|e^{-(\gamma_{i,1}+\beta)(t-1-s)}-e^{-(\gamma_{i,1}+\beta)(t-1-u)} |}{(s-u)^{1+\alpha}}duds\Big{)}\nonumber\\
&\leq&C_{\alpha}(B_{i}^{H})|_{2(t-1)}^{t-1}\Big{(} \int_{2(t-1)}^{t-1}\frac{(\gamma_{i,1}+\beta)^{\alpha-1}(t-1-s)^{\alpha-1}}{(s-2(t-1))^{\alpha}}ds\nonumber\\
&&+\int_{2(t-1)}^{t-1}e^{-(\gamma_{i,1}+\beta)(t-1-s)}\int_{2(t-1)}^{s}
\frac{(\gamma_{i,1}+\beta)^{\frac{1}{2}}(s-u)^{\frac{1}{2}}}{(s-u)^{1+\alpha}}duds\Big{)}\nonumber\\
&\leq& CC_{\alpha}(B_{i}^{H})|_{2(t-1)}^{t-1}(\gamma_{i,1}+\beta)^{-\frac{1}{2}} ,
\end{eqnarray}
where we have used $\alpha-1<-\frac{1}{2}$ and the fact $\int_{a}^{b}(b-s)^{\lambda-1}(s-a)^{-\lambda}ds=\frac{\pi}{\sin \lambda\pi}$ for arbitrary $a,b\in \mathbb{R},a<b$ and $\lambda\in (0,1)$ in the last step.
Combining $(4.35)-(4.37),$ we reach
\begin{eqnarray*}
I_{2}&\leq&C\sum^{\infty}_{i=1}\sqrt{\lambda_{i,1}}\frac{\gamma_{i,1}^{\frac{3}{2}}}{\sqrt{\gamma_{i,1}+\beta}}e^{-(\gamma_{i,1}+\beta )} C_{\alpha}(B_{i}^{H})|_{2(t-1)}^{t-1}\nonumber\\
 &&+ \sum^{\infty}_{i=1}\sqrt{\lambda_{i,1}}\gamma_{i,1}^{\frac{3}{2}}(\gamma_{i,1}+\beta )^{-\frac{3}{2}}e^{-(\gamma_{i,1}+\beta )}\sum _{n=2}^{\infty}C_{\alpha}(B_{i}^{H})|_{(n+1)(t-1)}^{n(t-1)}(n-1)^{-2}.
\end{eqnarray*}
Since $(B_{i}^{H})_{i\in \mathbb{N}^{+}}$ is i.i.d sequence and each fBm $B_{i}^{H}$ has stationary increments, by Lemmma $7.5$ in $\cite{NR}$ we obtain
\begin{eqnarray*}
EI_{2}&\leq&CE\Big{(}C_{\alpha}(B_{i}^{H})|_{2(t-1)}^{t-1}\Big{)}\sum^{\infty}_{i=1}\sqrt{\lambda_{i,1}}\gamma_{i,1}e^{-(\gamma_{i,1}+\beta )}\Big{(}1+ \sum _{n=2}^{\infty}C(n-1)^{-2}\Big{)}\nonumber\\
&\leq&CE\Big{(}C_{\alpha}(B_{i}^{H})|_{0}^{2}\Big{)}\sum^{\infty}_{i=1}\sqrt{\lambda_{i,1}}\gamma_{i,1}e^{-(\gamma_{i,1}+\beta )} \sum _{n=2}^{\infty}(n-1)^{-2}<\infty,
\end{eqnarray*}
which implies $I_{2}<\infty,a.s..$
Therefore, Minkowsky inequality we have
\begin{eqnarray}
(EI_{2}^{m})^\frac{1}{m}&\leq& C(EC^{m}_{\alpha}(B_{i}^{H})|_{0}^{2})^{\frac{1}{m}}\sum^{\infty}_{i=1} e^{-(\gamma_{i,1}+\beta )}\nonumber\\
&&+(EC^{m}_{\alpha}(B_{i}^{H})|_{2(t-1)}^{(t-1)})^{\frac{1}{m}}\sum^{\infty}_{i=1}\sum^{\infty}_{n=2}e^{-(\gamma_{i,1}+\beta )} (n-1)^{-2}\nonumber\\
&\leq& Ce^{-\beta}\sum^{\infty}_{i=1} e^{-\gamma_{i,1}}.
\end{eqnarray}
On account of $(4.34)$ and $(4.38),$ we can choose $\beta$ big enough such that $EI^{m}\leq \varepsilon $ which complete the proof.
\hspace{\fill}$\square$
\end{proof}
Taking a similar argument, we have the moment estimates for $Z_{2}(0)$ in the following.
\begin{lemma}
For any $\varepsilon>0$ and $m\geq2$, under condition $(2.15)$ there exists a positive constant $\beta$ depending only on $\varepsilon$ and $m$ such that
\begin{eqnarray}
E\|Z_{2}(0)\|_{3}^{m}<\varepsilon.
\end{eqnarray}
\end{lemma}
Concerning the growth properties of $(Z_{j}(t))_{t\in \mathbb{R}},j=1,2,$ we have Lemma 4.3 and Lemma 4.4 below.
\begin{lemma}
Under the condition of Proposition $2.1,$ there exists a random variable $r(\omega)$ taking finite values, i.e., $ r(\omega)<\infty, $ for all $\omega\in \Omega,$ such that
$$\|Z_{1}(t)\|_{3} < r(\omega),$$
for all $t\in (-\infty, 0].$
\end{lemma}
\begin{proof}
\begin{eqnarray}
\|Z_{1}(t)\|_{3}^{2}=\sum_{i=1}^{\infty}\lambda_{i,1}\gamma_{i,1}^{3}\lim\limits_{\tau\rightarrow -\infty}|\int^{t}_{\tau}e^{-(\gamma_{i,1}+\beta)(t-s)} dB_{i}^{H}(s)|^{2}.
\end{eqnarray}
By the concept of stochastic integral driven by fBm, we have
\begin{eqnarray}
&&\int^{t}_{\tau}e^{-(\gamma_{i,1}+\beta)(t-s)} dB_{i}^{H}(s)=(-1)^{\alpha}\int_{\tau}^{t}\frac{1}{\Gamma (\alpha)}\Big{(}\frac{e^{-(\gamma_{i,1}+\beta )(t-s)}}{(s-\tau)^{\alpha}}+\alpha\int_{\tau}^{s}\frac{e^{-(\gamma_{i,1}+\beta )(t-s)}- e^{-(\gamma_{i,1}+\beta )(t-u)} }{(s-u)^{\alpha+1}} du\Big{)}\nonumber\\
&&\cdot \frac{(-1)^{1-\alpha}}{\Gamma(\alpha)} \Big{(}\frac{B^{H}_{i}(t)- B^{H}_{i}(s)}{(t-s)^{1-\alpha}}+(1-\alpha)\int_{s}^{t}
\frac{B^{H}_{i}(s)- B^{H}_{i}(v) }{(v-s)^{2-\alpha}}dv     \Big{)}ds.
\end{eqnarray}
Substituting $(4.41)$ into $(4.40)$ yields
\begin{eqnarray}
&&\|Z_{1}(t)\|_{3}^{2} \leq C\sum_{i=1}^{\infty}\lambda_{i,1}\gamma_{i,1}^{3}\lim\limits_{\tau\rightarrow -\infty}
|\int^{t}_{\tau}\frac{e^{-(\gamma_{i,1}+\beta )(t-s)}}{(s-\tau)^{\alpha}} \frac{B^{H}_{i}(t)- B^{H}_{i}(s)}{(t-s)^{1-\alpha}}ds |^{2}\nonumber\\
&& +C\sum_{i=1}^{\infty}\lambda_{i,1}\gamma_{i,1}^{3}\lim\limits_{\tau\rightarrow -\infty}| \int^{t}_{\tau}\frac{B^{H}_{i}(t)- B^{H}_{i}(s)}{(t-s)^{1-\alpha}}\Big{(}\int_{\tau}^{s}\frac{e^{-(\gamma_{i,1}+\beta )(t-s)}- e^{-(\gamma_{i,1}+\beta )(t-u)} }{(s-u)^{\alpha+1}}du\Big{)} ds |^{2}\nonumber\\
&&+C\sum_{i=1}^{\infty}\lambda_{i,1}\gamma_{i,1}^{3}\lim\limits_{\tau\rightarrow -\infty}
|\int^{t}_{\tau}( \int_{\tau}^{s}\frac{e^{-(\gamma_{i,1}+\beta )(t-s)}- e^{-(\gamma_{i,1}+\beta )(t-u)} }{(s-u)^{\alpha+1}} du)
(\int_{s}^{t}
\frac{B^{H}_{i}(s)- B^{H}_{i}(v) }{(v-s)^{2-\alpha}}dv )ds   |^{2}\nonumber\\
&&+C\sum_{i=1}^{\infty}\lambda_{i,1}\gamma_{i,1}^{3}\lim\limits_{\tau\rightarrow -\infty}|\int_{\tau}^{t}   \frac{e^{-(\gamma_{i,1}+\beta )(t-s)}}{(s-\tau)^{\alpha}}( \int_{s}^{t}
\frac{B^{H}_{i}(s)- B^{H}_{i}(v) }{(v-s)^{2-\alpha}}dv)ds|^{2}\nonumber\\
&&:=I_{1}+I_{2}+I_{3}+I_{4}.
\end{eqnarray}
Applying Lemma $2.6$ in $\cite{MS}$, we deduce for fixed $t<0$ and all $s\in (-\infty, t]$
\begin{eqnarray}
&&e^{-(\gamma_{i,1}+\beta)(t-s)}|B^{H}_{i}(t)-B^{H}_{i}(s)|\nonumber\\
&=&e^{-(\gamma_{i,1}+\beta)(t-s)}|B^{H}_{i}(t)-B^{H}_{i}(s)|\sum_{n=1}^{\infty}I_{[(n+1)t\leq s\leq nt]}(s,t)\nonumber\\
&\leq& \sum_{n=1}^{\infty}e^{-(\gamma_{i,1}+\beta)(t-s)}|B^{H}_{i}(t)-B^{H}_{i}(s)|I_{[(n+1)t\leq s\leq nt]}(s,t)\nonumber\\
&\leq& \sum_{n=1}^{\infty}e^{(\gamma_{i,1}+\beta)(n-1)t}2[(n+1)^{2}t^{2}+k(\omega) ]\nonumber\\
&\leq& C\sum_{n=2}^{\infty}\frac{(n+1)^{2}t^{2}+k(\omega) }{[ (\gamma_{i,1}+\beta)(n-1)t ]^{4}}\nonumber\\
&=& \frac{C}{t^{2}(\gamma_{i,1}+\beta)^{4}}(\sum_{n=1}^{\infty}\frac{(n+1)^{2}}{ (n-1)^{4}}+\sum_{n=2}^{\infty}\frac{k(\omega)}{(n-1)^{4}t^{2}} )\nonumber\\
&\leq& \frac{C}{t^{2}(\gamma_{i,1}+\beta)^{4}}+\frac{Ck(\omega)}{t^{4}(\gamma_{i,1}+\beta)^{4}}.
\end{eqnarray}
Since $Z_{1}$ is continuous in $(H^{3}(\mho))^{2}$ with respect to time $t,$ we only need to show the result of this theorem is true when $t\leq -1.$
Therefore, in the following we assume $t \leq -1$. By  $(4.42)$ and  $(4.43)$ we have
\begin{eqnarray*}
I_{1}\leq C\sum_{i=1}^{\infty}\lambda_{i,1}\gamma_{i,1}^{3}\frac{C+Ck^{2}(\omega)}{(\gamma_{i,1}+\beta)^{8}} \frac{1}{t^{8}}\leq (C+Ck^{2}(\omega))\frac{1}{t^{8}}<\infty .
\end{eqnarray*}
To estimate $I_{2},$ we first consider
\begin{eqnarray}
&&|\int^{t}_{\tau}\frac{B^{H}_{i}(t)- B^{H}_{i}(s)}{(t-s)^{1-\alpha}}\Big{(}\int_{\tau}^{s}\frac{e^{-(\gamma_{i,1}+\beta )(t-s)}- e^{-(\gamma_{i,1}+\beta )(t-u)} }{(s-u)^{\alpha+1}}du\Big{)} ds |\nonumber\\
&\leq& |\int^{t}_{\tau}\frac{|B^{H}_{i}(t)- B^{H}_{i}(s)|e^{-(\gamma_{i,1}+\beta )(t-s)}}{(t-s)^{1-\alpha}}\Big{(}\int_{s-1}^{s}\frac{1- e^{-(\gamma_{i,1}+\beta )(s-u)} }{(s-u)^{\alpha+1}}du\Big{)} ds |\nonumber\\
&&+|\int^{t}_{\tau}\frac{|B^{H}_{i}(t)- B^{H}_{i}(s)|e^{-(\gamma_{i,1}+\beta )(t-s)}}{(t-s)^{1-\alpha}}\Big{(}\int_{-\infty}^{s-1}\frac{1- e^{-(\gamma_{i,1}+\beta )(s-u)} }{(s-u)^{\alpha+1}}du\Big{)} ds|\nonumber\\
&\leq& |\int^{t}_{\tau}\frac{|B^{H}_{i}(t)- B^{H}_{i}(s)|e^{-(\gamma_{i,1}+\beta )(t-s)}}{(t-s)^{1-\alpha}}\Big{(}\int_{s-1}^{s}\frac{ (\gamma_{i,1}+\beta  )^{\alpha+\varepsilon} (s-u)^{\alpha+\varepsilon } }{(s-u)^{\alpha+1}}du\Big{)} ds |\nonumber\\
&&+|\int^{t}_{\tau}\frac{|B^{H}_{i}(t)- B^{H}_{i}(s)|e^{-(\gamma_{i,1}+\beta )(t-s)}}{(t-s)^{1-\alpha}}\Big{(}\int_{-\infty}^{s-1}\frac{1}{(s-u)^{\alpha+1}}du\Big{)} ds|\nonumber\\
&=&|\int^{t}_{\tau}\frac{|B^{H}_{i}(t)- B^{H}_{i}(s)|e^{-(\gamma_{i,1}+\beta )(t-s)}}{(t-s)^{1-\alpha}}ds|(\frac{ (\gamma_{i,1}+\beta )^{\alpha+\varepsilon} }{\varepsilon} +\frac{1}{\alpha}),
\end{eqnarray}
where $\varepsilon $ is small positive constant such that $\alpha +\varepsilon<1 $ and we have used the fact
\begin{eqnarray*}
e^{-|x|}-e^{-|y|}\leq C_{r} |x-y|^{r},
\end{eqnarray*}
for $ x,y\in \mathbb{R}, r\in [0,1]$ in the seconde inequality.
Note that from $(4.43)$ and $(4.44)$ it follows that
\begin{eqnarray*}
&&|\int^{t}_{\tau}\frac{B^{H}_{i}(t)- B^{H}_{i}(s)}{(t-s)^{1-\alpha}}\Big{(}\int_{\tau}^{s}\frac{e^{-(\gamma_{i,1}+\beta )(t-s)}- e^{-(\gamma_{i,1}+\beta )(t-u)} }{(s-u)^{\alpha+1}}du\Big{)} ds |\nonumber\\
&\leq &|\int^{t}_{t-1}\frac{|B^{H}_{i}(t)- B^{H}_{i}(s)|e^{-(\gamma_{i,1}+\beta )(t-s)}}{(t-s)^{1-\alpha}}ds|(\frac{ (\gamma_{i,1}+\beta )^{\alpha+\varepsilon} }{\varepsilon} +\frac{1}{\alpha})\nonumber\\
&&+|\int^{t-1}_{-\infty}\frac{|B^{H}_{i}(t)- B^{H}_{i}(s)|e^{-(\gamma_{i,1}+\beta )(t-s)}}{(t-s)^{1-\alpha}}ds|(\frac{ (\gamma_{i,1}+\beta )^{\alpha+\varepsilon} }{\varepsilon} +\frac{1}{\alpha})\nonumber\\
&\leq &\frac{C+Ck(\omega)}{\alpha(\gamma_{i,1}+\beta)^{4}}(\frac{ (\gamma_{i,1}+\beta )^{\alpha+\varepsilon} }{\varepsilon} +\frac{1}{\alpha})\nonumber\\
&&+\frac{C+Ck(\omega)}{(\gamma_{i,1}+\beta)^{4}}(\frac{ (\gamma_{i,1}+\beta )^{\alpha+\varepsilon} }{\varepsilon} +\frac{1}{\alpha})
\int^{t-1}_{-\infty}\frac{[\frac{(\gamma_{i,1}+\beta)}{2}(t-s)]^{-2\alpha}}{(t-s)^{1-\alpha}}ds\nonumber\\
&\leq &\frac{C+Ck(\omega)}{(\gamma_{i,1}+\beta)^{4+\alpha-\varepsilon}}\nonumber
\end{eqnarray*}
which by $(4.42)$ yields that
\begin{eqnarray*}
I_{2}\leq C\sum_{i=1}^{\infty}\lambda_{i,1}\gamma_{i,1}^{3}\frac{C+Ck^{2}(\omega)}{(\gamma_{i,1}+\beta)^{8+2\alpha-2\varepsilon}}<\infty.
\end{eqnarray*}
In order to estimate $I_{3},$ we consider
\begin{eqnarray}
&&|\int^{t}_{\tau}( \int_{\tau}^{s}\frac{e^{-(\gamma_{i,1}+\beta )(t-s)}- e^{-(\gamma_{i,1}+\beta )(t-u)} }{(s-u)^{\alpha+1}} du)(\int_{s}^{t}
\frac{B^{H}_{i}(s)- B^{H}_{i}(v) }{(v-s)^{2-\alpha}}dv )ds   |\nonumber\\
&\leq& |\int^{t}_{\tau}( \int_{\tau}^{s}\frac{e^{-(\gamma_{i,1}+\beta )(t-s)}- e^{-(\gamma_{i,1}+\beta )(t-u)} }{(s-u)^{\alpha+1}} du)(\int_{s+1}^{t}
\frac{|B^{H}_{i}(s)- B^{H}_{i}(v) |}{(v-s)^{2-\alpha}}dv )ds |\nonumber\\
&&+ |\int^{t}_{\tau}( \int_{\tau}^{s}\frac{e^{-(\gamma_{i,1}+\beta )(t-s)}- e^{-(\gamma_{i,1}+\beta )(t-u)} }{(s-u)^{\alpha+1}} du)(\int_{s}^{s+1}
\frac{|B^{H}_{i}(s)- B^{H}_{i}(v) |}{(v-s)^{2-\alpha}}dv )ds|\nonumber\\
&:=&J_{1}+J_{2}.
\end{eqnarray}
By Lemma $2.6$ in $\cite{MS}$ we can derive the following estimate for $J_{1}$
\begin{eqnarray*}
J_{1}&\leq& \int^{t}_{\tau}( \int_{\tau}^{s}\frac{e^{-(\gamma_{i,1}+\beta )(t-s)}- e^{-(\gamma_{i,1}+\beta )(t-u)} }{(s-u)^{\alpha+1}} du)[2(|s|+1)^{2}+k(\omega)](t-s)ds\\
&\leq& \int^{t}_{\tau}e^{-(\gamma_{i,1}+\beta )(t-s)}s^{2}(t-s)\Big{(}\int_{\tau}^{s}\frac{1-e^{-(\gamma_{i,1}+\beta )(s-u)}}{(s-u)^{\alpha+1}}du\Big{)}ds\\
&\leq& \int^{t}_{\tau}e^{-(\gamma_{i,1}+\beta )(t-s)}s^{2}(t-s)\Big{(}\int_{s-1}^{s}\frac{1-e^{-(\gamma_{i,1}+\beta )(s-u)}}{(s-u)^{\alpha+1}}du\Big{)}ds\\
&&+\int^{t}_{\tau}e^{-(\gamma_{i,1}+\beta )(t-s)}s^{2}(t-s)\Big{(}\int_{-\infty}^{s-1}\frac{1-e^{-(\gamma_{i,1}+\beta )(s-u)}}{(s-u)^{\alpha+1}}du\Big{)}ds \\
&\leq& \int^{t}_{\tau}e^{-(\gamma_{i,1}+\beta )(t-s)}s^{2}(t-s)\Big{(}\int_{s-1}^{s}\frac{(\gamma_{i,1}+\beta )^{\alpha+\varepsilon}(s-u)^{\alpha+\varepsilon}}{(s-u)^{\alpha+1}}du\Big{)}ds \\
&&+\int^{t}_{\tau}e^{-(\gamma_{i,1}+\beta )(t-s)}s^{2}(t-s)\Big{(}\int_{-\infty}^{s-1}\frac{1}{(s-u)^{\alpha+1}}du\Big{)}ds .
\end{eqnarray*}
After some elementary calculations, we arrive at
\begin{eqnarray*}
J_{1}&\leq&C (\gamma_{i,1}+\beta )^{\alpha+\varepsilon} \int^{t}_{\tau}e^{-(\gamma_{i,1}+\beta )(t-s)}s^{2}(t-s)ds\\
&\leq&C (\gamma_{i,1}+\beta )^{\alpha+\varepsilon} \sum_{n=1}^{\infty}\int^{nt}_{(n+1)t}e^{-(\gamma_{i,1}+\beta )(t-s)}s^{2}(t-s)ds\\
&\leq&C (\gamma_{i,1}+\beta )^{\alpha+\varepsilon} \sum_{n=1}^{\infty}\int^{nt}_{(n+1)t}e^{(\gamma_{i,1}+\beta )(n-1)t}(nt)^{2}[t-(n+1)t]ds\\
&\leq&C (\gamma_{i,1}+\beta )^{\alpha+\varepsilon} \sum_{n=1}^{\infty}\frac{n^{3}t^{4}}{(\gamma_{i,1}+\beta )^{5}(n-1)^{5}|t|^{5}}\\
&\leq&\frac{C}{(\gamma_{i,1}+\beta )^{5-\alpha-\varepsilon} }.
\end{eqnarray*}
Choosing a positive constant $\varepsilon$ such that $\alpha+ H-\varepsilon >1$ and using Lemma $7.4$ in $\cite{NR}$ we obtain
\begin{eqnarray*}
J_{2}&\leq&C  |\int^{t}_{\tau}( \int_{\tau}^{s}\frac{e^{-(\gamma_{i,1}+\beta )(t-s)}- e^{-(\gamma_{i,1}+\beta )(t-u)} }{(s-u)^{\alpha+1}} du)(\int_{s}^{s+1}
\frac{(v-s)^{H-\varepsilon} }{(v-s)^{2-\alpha}}dv )ds|\\
&\leq&\frac{C}{\alpha+H-\varepsilon-1}\int^{t}_{\tau}e^{-(\gamma_{i,1}+\beta )(t-s)}( \int_{\tau}^{s}\frac{1- e^{-(\gamma_{i,1}+\beta )(s-u)} }{(s-u)^{\alpha+1}} du)ds\\
&\leq&C\int^{t}_{\tau}e^{-(\gamma_{i,1}+\beta )(t-s)}\Big{(} \int_{s-1}^{s}\frac{(\gamma_{i,1}+\beta )^{\alpha+\varepsilon}(s-u)^{\alpha+\varepsilon }}{(s-u)^{\alpha+1}  }du  \Big{)}ds\\
&&+C\int^{t}_{\tau}e^{-(\gamma_{i,1}+\beta )(t-s)}\Big{(} \int_{-\infty}^{s-1}\frac{1}{(s-u)^{\alpha+1}  }du  \Big{)}ds\\
&\leq&C(\gamma_{i,1}+\beta )^{\alpha+\varepsilon}\int^{t}_{\tau}e^{-(\gamma_{i,1}+\beta )(t-s)}ds\\
&\leq&C(\gamma_{i,1}+\beta )^{\alpha+\varepsilon-1}.
\end{eqnarray*}
Therefore, by $(4.45)$ and estimates of $J_{1}$ and $J_{2}$ we conclude
\begin{eqnarray*}
I_{3}\leq \sum_{i=1}^{\infty}\lambda_{i,1}\gamma_{i}^{3}(\frac{C}{(\gamma_{i,1}+\beta )^{5-\alpha-\varepsilon} }+C(\gamma_{i,1}+\beta )^{\alpha+\varepsilon-1})^{2}< C\sum_{i=1}^{\infty}\lambda_{i,1}\gamma_{i,1}^{2}< \infty.
\end{eqnarray*}
To estimate $I_{4},$ by Lemma $7.4$ in $\cite{NR}$ we note that
\begin{eqnarray}
&&|\int_{\tau}^{t}   \frac{e^{-(\gamma_{i,1}+\beta )(t-s)}}{(s-\tau)^{\alpha}}( \int_{s}^{t}\frac{B^{H}_{i}(s)- B^{H}_{i}(v) }{(v-s)^{2-\alpha}}dv)ds|\nonumber\\
&\leq &|\int_{\tau}^{t}   \frac{e^{-(\gamma_{i,1}+\beta )(t-s)}}{(s-\tau)^{\alpha}}( \int_{s+1}^{t}\frac{|B^{H}_{i}(s)- B^{H}_{i}(v)| }{(v-s)^{2-\alpha}}dv)ds|\nonumber\\
&&+|\int_{\tau}^{t}   \frac{e^{-(\gamma_{i,1}+\beta )(t-s)}}{(s-\tau)^{\alpha}}( \int_{s}^{s+1}\frac{|B^{H}_{i}(s)- B^{H}_{i}(v) | }{(v-s)^{2-\alpha}}dv)ds|\nonumber\\
&\leq &C|\int_{\tau}^{t}   \frac{e^{-(\gamma_{i,1}+\beta )(t-s)/2}}{(s-\tau)^{\alpha}}[\frac{(\gamma_{i,1}+\beta )(t-s)}{2}]^{\alpha-1}
( \int_{s+1}^{t}2s^{2}dv)ds|\nonumber\\
&&+C|\int_{\tau}^{t}   \frac{e^{-(\gamma_{i,1}+\beta )(t-s)/2}}{(s-\tau)^{\alpha}}[\frac{(\gamma_{i,1}+\beta )(t-s)}{2}]^{\alpha-1}
( \int_{s}^{s+1}   \frac{|s-v|^{H-\varepsilon}}{(v-s)^{2-\alpha}}dv)ds|\nonumber\\
&\leq &C\Big{(}\frac{\gamma_{i,1}+\beta}{2}\Big{)}^{\alpha-1} \int_{\tau}^{t}(s-\tau)^{-\alpha} (t-s)^{\alpha-1}e^{-(\gamma_{i,1}+\beta )(t-s)/2}
s^{2}(t-s)ds\nonumber\\
&&+C\Big{(}\frac{\gamma_{i,1}+\beta}{2}\Big{)}^{\alpha-1}\frac{1}{\alpha+H-\varepsilon-1} \int_{\tau}^{t}(s-\tau)^{-\alpha} (t-s)^{\alpha-1}e^{-(\gamma_{i,1}+\beta )(t-s)/2}ds.
\end{eqnarray}
Since
\begin{eqnarray*}
&&e^{-(\gamma_{i,1}+\beta )(t-s)/2}s^{2}(t-s)\\
&=& \sum_{n=1}^{\infty}e^{-(\gamma_{i,1}+\beta )(t-s)/2}s^{2}(t-s)I_{((n+1)t\leq s\leq nt ]}(s)\\
&\leq& \sum_{n=1}^{\infty}e^{-(\gamma_{i,1}+\beta )(n-1)|t|/2}(n+1)^{2}t^{2}|nt|\\
&\leq& C\sum_{n=1}^{\infty} \frac{n^{3}|t|^{3}}{(\gamma_{i,1}+\beta )^{5}(n-1)^{5}|t|^{5} }\leq \frac{C}{(\gamma_{i,1}+\beta )^{5} },
\end{eqnarray*}
which by $( 4.46)$ yields
\begin{eqnarray*}
|\int_{\tau}^{t}   \frac{e^{-(\gamma_{i,1}+\beta )(t-s)}}{(s-\tau)^{\alpha}}( \int_{s}^{t}\frac{B^{H}_{i}(s)- B^{H}_{i}(v) }{(v-s)^{2-\alpha}}dv)ds|\leq\frac{C}{(\gamma_{i,1}+\beta )^{1-\alpha}}\frac{\pi}{\sin \alpha\pi}.
\end{eqnarray*}
Subsequently, it follows that
\begin{eqnarray}
I_{4}\leq  C\sum_{i=1}^{\infty}\lambda_{i,1}\gamma_{i,1}^{1+2\alpha}<C\sum_{i=1}^{\infty}\lambda_{i,1}\gamma_{i,1}^{2}<\infty.
\end{eqnarray}
Combining the estimates of $I_{1}-I_{4}$ and Proposition 2.1, we complete the proof.
\hspace{\fill}$\square$
\end{proof}
Analogously,  considering the growth properties of $(Z_{2}(t))_{t\in \mathbb{R}}$ we have Lemma 4.4 whose proof is similar to Lemma 4.3 and is omitted.
\begin{lemma}
Under the condition of Proposition $2.1,$ there exists a random variable $r(\omega)$ taking finite values, i.e., $ r(\omega)<\infty, $ for all $\omega\in \Omega,$ such that
$$\|Z_{2}(t)\|_{3} < r(\omega),$$
for all $t\in (-\infty, 0].$
\end{lemma}
\noindent ${4.3.\ \mathbf{Proof }\ \mathbf{Of }\ \mathbf{Theorem}\ \mathbf{3.1}. } $
Under conditions of Theorem $ 3.2,$ for $j\in \{1,2\}$, $(W^{H}_{j}(.,t))_{t\in \mathbb{R}^{+}},$ is a fractional Wiener process with values in $\mathcal{V}_{j}$. As usual in this context, we extend the fractional Wiener process $(W^{H}_{j}(.,t))_{t\in \mathbb{R}^{+}}$ to all $\mathbb{R}$ by
setting
\begin{eqnarray*}
W^{H}_{j}(.,t)= V^{H}_{j}(.,-t),\ \ t\leq 0,
\end{eqnarray*}
 where $(V^{H}_{j}(.,t))_{t\in \mathbb{R}^{+}},$  is another fractional Wiener process with the same covariance operator as $(W^{H}_{j}(.,t))_{t\in \mathbb{R}^{+}}$.
  We will consider a canonical version  $(\omega_{j}(.,t))_{t\in \mathbb{R}}$ of this process given by the probability space
$(C_{0}(\mathbb{R},\mathcal{V}_{i}), \mathcal{B}(C_{0}(\mathbb{R},\mathcal{V}_{i})),\mathbb{P}_{i},\vartheta) $ where $\mathbb{P}_{i}$ is the Gaussian-measure generated by $(W^{H}_{j}(.,t))_{t\in \mathbb{\mathbb{R}}}$. So, if we set $(\omega(.,t))_{t\in \mathbb{R}}:= (\omega_{1}(.,t), \omega_{2}(.,))_{t\in \mathbb{R}}$, the process $(\omega(.,t))_{t\in \mathbb{R}} $ is given by the probability space
$(C_{0}(\mathbb{R},\mathcal{V}), \mathcal{B}(C_{0}(\mathbb{R},\mathcal{V})),\tilde{\mathbb{P}},\vartheta), $ where $\tilde{ \mathbb{P}}=\mathbb{P}_{1}\times\mathbb{P}_{2} .$ Now we may define the stochastic dynamical system $(S(t,s;\omega))_{t\geq s,\omega\in \tilde{\Omega}}$ by
\begin{eqnarray}
S(t,s;\omega)(\upsilon_{s},T_{s})=(u(t,\omega_{1})+Z_{1}(t,\omega_{1}), \theta(t, \omega_{2})+Z_{2}(t,\omega_{2}))
\end{eqnarray}
where $(\upsilon, T)$ is the strong solution to $(1.1)-(1.5)$ with $(\upsilon_{s}, T_{s})=(u_{s}+Z_{1}(s,\omega_{1}), \theta_{s}+Z_{2}(s,\omega_{2}  )$ and $(u, \theta)$ is the strong solution to $(2.17)-(2.22)$ or equivalently it is the strong solution to $(5.86)-(5.91).$
It can be checked that assumptions (i)-(iv) and (a)-(c) are satisfied with $X=\mathcal{V}$.
\par
Further more, by $\cite{MS} ,$ we have the result:
$(C_{0}(\mathbb{R}, \mathcal{V}), \mathcal{B}(C_{0}(\mathbb{R}, \mathcal{V} )), \tilde{\mathbb{P}}, \vartheta )$ is an ergodic metric dynamical system.
Properties $\mathrm{(i)}, \mathrm{(ii)}, \mathrm{(iv)}$ of the solution operator $(S(t,s;\omega))_{t\geq s,\omega\in \tilde{\Omega}} $ follows by Theorem $3.2$ and property $\mathrm{(iii)}$ of the solution operator also holds regarding the fact that the proof of the global existence of strong solution to $(1.1)-(1.5)$ rest upon Faedo-Galerkin method.
\par
Next, we will prove the existence of the random attractor.
\par
For $j=1,2$ and $t\in \mathbb{R}$, we have
\begin{eqnarray*}
Z_{j}(t)&=&\int_{-\infty}^{t}e^{-(t-s)(A_{j}+\beta)}dW_{j}^{H}(s)\\
&=&\sum_{i=1}^{\infty}\lambda_{i,j}^{\frac{1}{2}}e_{i,j}\int_{-\infty}^{t}e^{-(t-s)(\gamma_{i,j}+\beta)}d\beta_{i}^{H}(s).
\end{eqnarray*}
Since $e^{-(t-s)(\gamma_{i,j}+\beta)}$ is infinitely many times continuously differentiable
with respect to $s$ and $\beta_{i}^{H}(s)$ has $\alpha$--H$\mathrm{\ddot{o}}$lder continuous paths for all $\alpha\in (0, H)$ with $H>\frac{1}{2},$ by Theorem$ 4.2.1 $ in $\cite {Z}$, we know the stochastic integral on the right hand side of the above equality is equivalent to  Riemann--Stieltjes  integral. Therefore,
\begin{eqnarray*}
Z_{j}(t)&=&\sum_{i=1}^{\infty}\lambda_{i,j}^{\frac{1}{2}}e_{i,j}\lim\limits_{\tau\rightarrow -\infty}\int_{\tau-t}^{0}e^{s(\gamma_{i,j}+\beta)}d\theta_{t}\beta_{i}^{H}(s).
\end{eqnarray*}
By integration by parts and $ \theta_{t}\beta_{i}^{H}(s)$ has polynomial growth with respect to $s$, we have
\begin{eqnarray*}
Z_{j}(t)&=&-\sum_{i=1}^{\infty}\lambda_{i,j}^{\frac{1}{2}}e_{i,j}\lim\limits_{\tau\rightarrow -\infty}\int_{\tau-t}^{0}( \gamma_{i,j}+\beta )e^{s(\gamma_{i,j}+\beta)}\theta_{t}\beta_{i}^{H}(s)d(s)\\
&=&-\int_{-\infty}^{0}(A+\beta)e^{s(A+\beta)}\theta_{t}W_{j}^{H}(s)ds.
\end{eqnarray*}
Since the integral is convergent in $H^{1}(\mho)$, by Fubini theorem, we know $Z_{j}(t)$ is adapted with respect to $\mathcal{F}_{t}:=\sigma( W_{j}^{H}(s), s\leq t, j=1,2).$ Therefore,
for $s<0,$ applying the properties of ergodic metric dynamical system we have
\begin{eqnarray*}
\lim\limits_{s\rightarrow -\infty}\frac{1}{-s}\int_{s}^{0}(\|Z_{1}\|_{2}^{2}+\|Z_{1}\|_{2}^{4}+\|Z_{2}\|_{3}^{2} )ds
=E(\|Z_{1}(0)\|_{2}^{2}+\|Z_{1}(0)\|_{2}^{4}+\|Z_{2}(0)\|_{3}^{2} ).
\end{eqnarray*}
From Lemma $4.1$ and Lemma $4.2,$ there exists a $\beta$ which is big enough such that
\begin{eqnarray*}
E(\|Z_{1}(0)\|_{2}^{2}+\|Z_{1}(0)\|_{2}^{4}+\|Z_{2}(0)\|_{3}^{2} )<\frac{\gamma_{1}}{2}
\end{eqnarray*}
which implies that there exists a random variable $\tau(\omega)$ such that for $s<\tau(\omega)$ we have
\begin{eqnarray}
\int_{s}^{0}(-\gamma_{1}+\|Z_{1}\|_{2}^{2}+\|Z_{1}\|_{2}^{4}+\|Z_{2}\|_{3}^{2} )ds \leq \frac{\gamma_{1}}{2}s.
\end{eqnarray}
Due to $(5.102)$,
\begin{eqnarray}
|u(-4)|_{2}^{2}+|\theta(-4)|_{2}^{2}&\leq& (|u(t_{0})|^{2}_{2}+|\theta(t_{0})|^{2}_{2})
e^{C\int_{t_{0}}^{-4}(-\gamma_{1}+\|Z_{1}\|_{2}^{2}+\|Z_{1}\|_{2}^{4}+\|Z_{2}\|_{3}^{2})ds }\nonumber\\
&&+\int_{t_{0}}^{-4} e^{C\int_{s}^{-4}(-\gamma_{1}+\|Z_{1}\|_{2}^{2}+\|Z_{1}\|_{2}^{4}+\|Z_{2}\|_{3}^{2})ds}  (\|Z_{1}\|^{2}_{1}+|Q|^{2}_{2})ds.
\end{eqnarray}
In the following, we denote by $(u(t,\omega;t_{0},u_{0}), \theta(t,\omega;t_{0},\theta_{0})) $ the solution to $(5.86)-(5.91)$ with $(u(t_{0})=u_{0}, \theta(t_{0})=\theta_{0}).$
Then, by$(4.49)$ and $(4.50),$ there exists a random variable $c_{1}(\omega)>0,$ depending only on $\gamma_{1}, Z_{1}$ and $Z_{2},$ such that for arbitrary $\rho >0$ there exists $t(\omega)\leq -4$ such that the following
holds $P$-a.s.. For all $t_{0}\leq t(\omega)$ and $(u_{0},\theta_{0})\in \mathcal{H}$ with $|u_{0}|_{2}+ |\theta_{0}|_{2} \leq \rho,$ the solution
$(u(t,\omega;t_{0},u_{0}), \theta(t,\omega;t_{0},\theta_{0})) $ over $[t_{0}, \infty),$ satisfies
\begin{eqnarray}
&&|u(-4,\omega; t_{0}, u_{0} )|_{2}^{2}+ |\theta(-4,\omega; t_{0}, \theta_{0} )|_{2}^{2}\leq c_{1}(\omega).
\end{eqnarray}
Using $(5.102)$ again, for $t\in [-4,0]$ we have
\begin{eqnarray}
|u(t)|_{2}^{2}+|\theta(t)|_{2}^{2}&\leq& (|u(-4)|^{2}_{2}+|\theta(-4)|^{2}_{2})
e^{C\int_{-4}^{t}(-\gamma_{1}+\|Z_{1}\|_{2}^{2}+\|Z_{1}\|_{2}^{4}+\|Z_{2}\|_{3}^{2})ds }\nonumber\\
&&+\int_{-4}^{t} e^{C\int_{s}^{-4}(-\gamma_{1}+\|Z_{1}\|_{2}^{2}+\|Z_{1}\|_{2}^{4}+\|Z_{2}\|_{3}^{2})ds}  (\|Z_{1}\|^{2}_{1}+|Q|^{2}_{2})ds.
\end{eqnarray}
Moreover, integrating $(5.101)$ over $[-4,0]$ we obtain
\begin{eqnarray}
\int_{-4}^{0}(\|u\|_{1}^{2}+\|\theta\|_{1}^{2})ds&\leq& |u(-4)|^{2}_{2}+|\theta(-4)|^{2}_{2}+C\int_{-4}^{0}(|Q|_{2}^{2}+\|Z_{1}\|_{1}^{2})ds\nonumber\\
&&+C\int_{-4}^{0}(|u|_{2}^{2}+|\theta|_{2}^{2})(\|Z_{1}\|_{2}^{2}+\|Z_{1}\|_{2}^{4}+\|Z_{2}\|_{3}^{2})ds.\nonumber\\
\end{eqnarray}
Therefore, from $(4.51)-(4.53),$ we conclude that there exists two random variables $r_{1}(\omega)$ and $c_{1}(\omega)$, depending only on
$\gamma_{1}, Z_{1}$ and $Z_{2},$ such that for arbitrary $\rho >0$ there exists $t(\omega)\leq -4$ such that the following
holds $P$-a.s.. For all $t_{0}\leq t(\omega)$ and $(u_{0},\theta_{0})\in \mathcal{H}$ with $|u_{0}|_{2}+ |\theta_{0}|_{2} \leq \rho,$ the solution
$(u(t,\omega;t_{0},u_{0}), \theta(t,\omega;t_{0},\theta_{0})) $ over $[t_{0}, \infty),$ satisfies
\begin{eqnarray}
|u(t,\omega; t_{0}, u_{0} )|_{2}^{2}+ |\theta(t,\omega; t_{0}, \theta_{0} )|_{2}^{2}\leq r_{1}(\omega)\ \ \mathrm{for}\ \mathrm{all}\ t\in [-4,0]
\end{eqnarray}
and
\begin{eqnarray}
\int_{-4}^{0}(\|u(s,\omega; t_{0}, u_{0} )\|_{1}^{2}+\|\theta(s,\omega; t_{0}, \theta_{0} )\|_{1}^{2})ds \leq c_{2}(\omega).
\end{eqnarray}
For $t<-3,$ by $(5.109)$ we have
\begin{eqnarray*}
&&|\theta(-3, \omega; t_{0}, \theta_{0} )|^{2}_{4}\nonumber\\
&\leq& |\theta(t, \omega; t_{0}, \theta_{0} )|^{2}_{4}e^{-C(-3-t)}\nonumber\\
&&+C\int_{t}^{-3}e^{-C(-3-s)}(|Q|_{2}^{2}+ \|Z_{2}(s)\|_{3}^{2}+\|Z_{1}(s)\|_{1}^{2}+ \|Z_{2}(s)\|_{3}^{2}\|u(s,\omega; t_{0}, u_{0})\|_{1}^{2})ds.
\end{eqnarray*}
Integrating in $t$ over the interval $[-4, -3]$ yields
\begin{eqnarray}
&&|\theta(-3, \omega; t_{0}, \theta_{0})|^{2}_{4}\nonumber\\
&\leq& \int_{-4}^{-3} |\theta(t, \omega; t_{0}, \theta_{0})|^{2}_{4}e^{-C(-3-t)}dt\nonumber\\
&&+C\int_{-4}^{-3}\int_{t}^{-3}e^{-C(-3-s)}(|Q|_{2}^{2}+ \|Z_{2}(s)\|_{3}^{2}+\|Z_{1}(s)\|_{1}^{2}+ \|Z_{2}(s)\|_{3}^{2}\|u(s, \omega; t_{0}, u_{0} )\|_{1}^{2})dsdt\nonumber\\
&\leq& C\int_{-4}^{-3}( \|\theta(t, \omega; t_{0}, \theta_{0})\|^{2}_{1}+\|u(t, \omega; t_{0}, u_{0} )\|_{1}^{2}) dt\nonumber\\
&&+C\int_{-4}^{-3}e^{-C(-3-s)}(|Q|_{2}^{2}+ \|Z_{2}(s)\|_{3}^{2}+\|Z_{1}(s)\|_{1}^{2})dt,
\end{eqnarray}
where the seconde inequality follows by the boundedness of $e^{-C(-3-s)}\|Z_{2}(s)\|_{3}^{2} $ for $s\leq-3.$

By $(4.53)$ and $(4.56),$ there exists random variable $c_{3}(\omega)$, depending only on
$\gamma_{1}, Z_{1}$ and $Z_{2},$ such that for arbitrary $\rho >0$ there exists $t(\omega)\leq -3$ such that the following
holds $P$-a.s.. For all $t_{0}\leq t(\omega)$ and $\theta_{0}\in L^{4}(\mho)$ with $|\theta_{0}|_{4} \leq \rho,$
$ \theta(t,\omega;t_{0},\theta_{0}) $ satisfies
\begin{eqnarray}
|\theta(-3,\omega;t_{0},\theta_{0})|_{4}^{2}\leq c_{3}(\omega).
\end{eqnarray}
Repeating the argument as $(4.54),$ by $(4.57)$ we have that
 there exists random variable $c_{4}(\omega)$, depending only on
$\gamma_{1}, Z_{1}$ and $Z_{2},$ such that for arbitrary $\rho >0$ there exists $t(\omega)\leq -3$ such that the following
holds $P$-a.s.. For all $t_{0}\leq t(\omega)$ and $\theta_{0}\in L^{4}(\mho)$ with $|\theta_{0}|_{4} \leq \rho,$
$ \theta(t,\omega;t_{0},\theta_{0}) $ satisfies
\begin{eqnarray}
|\theta(t,\omega;t_{0},\theta_{0})|_{4}^{2}\leq c_{4}(\omega),\ \ \mathrm{for}\ \mathrm{all}\ t\in [-3,0].
\end{eqnarray}
By $(5.118),$ taking a similar argument as $(4.58)$ we deduce that there exists random variable $c_{5}(\omega)$, depending only on
$\gamma_{1}, Z_{1}$ and $Z_{2},$ such that for arbitrary $\rho >0$ there exists $t(\omega)\leq -3$ such that the following
holds $P$-a.s.. For all $t_{0}\leq t(\omega)$ and $\tilde{u}_{0}\in (L^{4}(\mho))^{2}$ with $|\tilde{u}_{0}|_{4} \leq \rho,$
$\tilde{u}(t,\omega;t_{0},\theta_{0}) $ satisfies
\begin{eqnarray}
|\tilde{u}(t,\omega;t_{0},\tilde{u}_{0})|_{4}^{2}\leq c_{5}(\omega),\ \ \mathrm{for}\ \mathrm{all}\ t\in [-3,0].
\end{eqnarray}
By $(5.117),$ making an analogous argument as $(4.55)$ we infer that
that there exists random variable $c_{6}(\omega)$, depending only on
$\gamma_{1}, Z_{1}$ and $Z_{2},$ such that for $\rho >0$ there exists $t(\omega)\leq -3$ such that the following
holds $P$-a.s.. For all $t_{0}\leq t(\omega)$ and $\tilde{u}_{0}\in (L^{4}(\mho))^{2}$ with $|\tilde{u}_{0}|_{4} \leq \rho,$
$\tilde{u}(t,\omega;t_{0},u_{0}) $ satisfies
\begin{eqnarray}
\int_{-3}^{0} ||\tilde{u}(s,\omega; t_{0}, \tilde{u}_{0} )| |\nabla \tilde{u}(s,\omega; t_{0}, \tilde{u}_{0} ) |  |_{2}^{2}ds \leq c_{6}(\omega).
\end{eqnarray}
By $(5.123)$ and $(4.59)-(4.60)$, proceeding as $(4.59)$ we find that there exists random variable $c_{7}(\omega)$, depending only on
$\gamma_{1}, Z_{1}$ and $Z_{2},$ such that for arbitrary $\rho >0$ there exists $t(\omega)\leq -2$ such that the following
holds $P$-a.s.. For all $t_{0}\leq t(\omega)$ and $\nabla \bar{u}_{0}\in (L^{2}(\mho))^{2}$ with $|\nabla \bar{u}_{0}|_{2} \leq \rho,$
$\bar{u}(t,\omega;t_{0},\bar{u}_{0}) $ satisfies
\begin{eqnarray}
 |\nabla\bar{u}(t,\omega; t_{0}, \bar{u}_{0} )|_{2}^{2}\leq c_{7}(\omega), \ \ \mathrm{for}\ \mathrm{all}\ t\in [-2,0].
\end{eqnarray}
In view of $(5.133), (4.55)$ and $(4.59)-(4.61),$ following the steps in $(4.59 )-(4.60)$ we get that there exists two random variables $r_{2}(\omega)$ and $c_{8}(\omega)$, depending only on
$\gamma_{1}, Z_{1}$ and $Z_{2},$ such that for arbitrary $\rho >0$ there exists $t(\omega)\leq -1$ such that the following
holds $P$-a.s.. For all $t_{0}\leq t(\omega)$ and $\partial_{z}u_{0}\in (L^{2}(\mho))^{2}$ with $|\partial{u}_{0}|_{2} \leq \rho,$
$\partial_{z}{u}(t,\omega;t_{0},u_{0}) $ satisfies
\begin{eqnarray}
 |\partial_{z}{u}(t,\omega; t_{0}, u_{0} )|_{2}^{2}\leq r_{2}(\omega), \ \ \mathrm{for}\ \mathrm{all}\ t\in [-1,0]
\end{eqnarray}
and
\begin{eqnarray}
\int_{-1}^{0} |\nabla{u}_{z}(s,\omega; t_{0}, u_{0} )|_{2}^{2}ds \leq c_{8}(\omega).
\end{eqnarray}
Regarding $(5.137),(4.59)$ and $(4.61)-(4.63),$ we repeat the procedures of deriving $(4.59)-(4.60)$ to get that there exists two random variables $r_{3}(\omega)$ and $c_{9}(\omega)$, depending only on
$\gamma_{1}, Z_{1}$ and $Z_{2},$ such that for arbitrary $\rho >0$ there exists $t(\omega)\leq -1$ such that the following
holds $P$-a.s.. For all $t_{0}\leq t(\omega)$ and $\nabla u_{0}\in (L^{2}(\mho))^{4}$ with $|\nabla{u}_{0}|_{2} \leq \rho,$
$u(t,\omega;t_{0},u_{0}) $ satisfies
\begin{eqnarray}
 |\nabla {u}(t,\omega; t_{0}, u_{0} )|_{2}^{2}\leq r_{3}(\omega), \ \ \mathrm{for}\ \mathrm{all}\ t\in [-1,0]
\end{eqnarray}
and
\begin{eqnarray}
\int_{-1}^{0} |\Delta{u}(s,\omega; t_{0}, u_{0} )|_{2}^{2}ds \leq c_{9}(\omega).
\end{eqnarray}
By $(5.139), (4.59), (4.62)$ and $(4.64)-(4.65),$ proceeding as above we have that  there exists random variable $r_{4}(\omega)$, depending only on
$\gamma_{1}, Z_{1}$ and $Z_{2},$ such that for arbitrary $\rho >0$ there exists $t(\omega)\leq -1$ such that the following
holds $P$-a.s.. For all $t_{0}\leq t(\omega)$ and $ \theta_{0}\in H^{1}(\mho)$ with $\|\theta_{0}\|_{1} \leq \rho,$
$\theta(t,\omega;t_{0},\theta_{0}) $ satisfies
\begin{eqnarray*}
 \| \theta(t,\omega; t_{0}, \theta_{0} )\|_{1}^{2}\leq r_{4}(\omega), \ \ \mathrm{for}\ \mathrm{all}\ t\in [-1,0].
\end{eqnarray*}
\par
Now we are ready to prove the desired compact result.
Let $r(\omega)= \Sigma_{i=1}^{4}r_{i}(\omega)+\|Z_{1}(-1)\|_{1}^{2}+\|Z_{2}(-1)\|_{1}^{2}$, then  $B(-1,r(\omega))$, the ball of center $0\in \mathcal{V}$ and radius $r(\omega),$ is a absorbing set
 at time $-1$ for $(S(t,s;\omega))_{t\geq s,\omega\in \tilde{\Omega}}$. Therefore, in order to prove the existence of the global random attractor of  the stochastic dynamical system in space $\mathcal{V}$, we need to to construct a compact absorbing set at time $0$ in $\mathcal{V}$ according to Theorem $2.2$. Denote by $\mathcal{B}$  a bounded subset $\mathcal{V}$ and set $\mathcal{C}_{T}$ as a subset of the function space:
\begin{eqnarray*}
 \mathcal{C}_{T}:=\Big{\{} \Big{(}A_{1}^{\frac{1}{2}}\upsilon,   A_{2}^{\frac{1}{2}}T\Big{)}\Big{|}(\upsilon(-1),T(-1))\in \mathcal{B},
 (\upsilon(t),T(t))=S(t,-1;\omega)(\upsilon(-1),T(-1)),t\in[-1,0]\Big{\}}.
\end{eqnarray*}

Since $\mathcal{V}_{i}\subset H_{i}$ is compact, $\mathcal{V}_{1}\times \mathcal{V}_{2}  \subset H_{1}\times H_{2} $ is also compact. Let $(\upsilon(-1), T(-1))\in \mathcal{B} ,$ by the argument of step 2 in the proof of Theorem 3.2 we know
\[
(A_{1}^{\frac{1}{2}}u, A_{2}^{\frac{1}{2}}\theta)\in L^{2}([-1,0];\mathcal{V}_{1}\times \mathcal{V}_{2}),\ \
(\partial_{t}A_{1}^{\frac{1}{2}}u,  \partial_{t}A_{2}^{\frac{1}{2}}\theta )\in L^{2}([-1,0];\mathcal{V}_{1}'\times \mathcal{V}_{2}').
\]
Therefore, by Lemma 2.1 with
\[
B_{0}=\mathcal{V}_{1}\times \mathcal{V}_{2},\ \ B=H_{1}\times H_{2},\ \ B_{1}=\mathcal{V}_{1}'\times \mathcal{V}_{2}',
\]
$\mathcal{C}_{T}$ is compact in $L^{2}([-1,0];\mathcal{H} ).$\\
In order to show that for any fixed $t\in (-1,0],\omega\in \tilde{\Omega}, S(t,-1;\omega) $ is a compact operator in $\mathcal{V} ,$ we take any bounded sequences
$\{(\nu_{0,n}, \tau_{0,n} ) \}_{n\in \mathbb{N}}\subset \mathcal{B}$ and we want to extract, for any fixed $t\in (-1,0]$ and $\omega\in \tilde{\Omega},$ a convergent subsequence from
$\{S(t,-1;\omega)( \nu_{0,n}, \tau_{0,n}) \}$.
Since $\{(A_{1}^{\frac{1}{2}}\upsilon, A_{2}^{\frac{1}{2}}T  ) \}\subset \mathcal{C}_{T} ,$ by Lemma 2.1, there is a function $(\nu_{*},\theta_{*})$:
\[
(\nu_{*},\theta_{*})\in  L^{2}([-1,0];\mathcal{V}),
\]
and a subsequence of $\{S(t,-1;\omega)(\nu_{0,n},\tau_{0,n} ) \}_{n\in \mathbb{N}},$ still denoted by  $\{S(t,-1;\omega)(\nu_{0,n},\tau_{0,n} ) \}_{n\in \mathbb{N}}$  for simplicity, such that
\begin{eqnarray}
\lim\limits_{n\rightarrow \infty}\int_{-1}^{0}\|S(t,-1;\omega)(\nu_{0,n},\tau_{0,n} )-(\nu_{*}(t),\theta_{*}(t))  \|_{1}^{2}dt=0.
\end{eqnarray}
By measure theory, convergence in mean square implies almost sure convergence of a subsequence. Therefore, it follows from $(4.66)$ that there exists a subsequence of $\{S(t,-1;\omega)(\nu_{0,n},\tau_{0,n} ) \}_{n\in \mathbb{N}}, $
still denoted by $\{S(t,-1;\omega)(\nu_{0,n},\tau_{0,n} ) \}_{n\in \mathbb{N}}$ for simplicity , such that
\begin{eqnarray}
 \lim\limits_{n\rightarrow \infty}\|S(t,-1;\omega)(\nu_{0,n},\tau_{0,n} )-(\nu_{*}(t),\theta_{*}(t))  \|_{ 1}=0,\ \ a.e.\ t\in (-1,0].
\end{eqnarray}
Fix any $t\in (-1,0].$ By $(4.67),$ we can select a $t_{0}\in (-1,t)$ such that
\[
 \lim\limits_{n\rightarrow \infty}\|S(t_{0},-1,\omega)(\nu_{0,n},\tau_{0,n} )-(\nu_{*}(t_{0}),\theta_{*}(t_{0}))  \|_{ 1}=0.
\]
Then by the continuity of the map $S(t-t_{0},t_{0};\omega) $ in $ \mathcal{V}$ with respect to initial value, we have
\begin{eqnarray*}
S(t,-1;\omega)(\nu_{0,n},\tau_{0,n})&=&S(t-t_{0},t_{0};\omega)S(t_{0},-1;\omega)(\nu_{0,n},\tau_{0,n})\\
&&\rightarrow S(t-t_{0},t_{0};\omega)(\nu_{*}(t_{0}),\theta_{*}(t_{0})),\ \ \ \mathrm{in}\ \mathcal{V}.
\end{eqnarray*}
Hence for any $t\in(-1,0], \{S(t,-1;\omega)(\nu_{0,n},\tau_{0,n} ) \}_{n\in \mathbb{N}}$   contains a subsequence which is convergent $\mathrm{in}\ \mathcal{V},$ which implies that for any fixed $t\in (-1,0],\omega\in \tilde{\Omega}, S(t,-1;\omega) $ is a compact operator in $\mathcal{V}.$ Let $\mathcal{B}(0,\omega)= \overline{S(0,-1;\omega)B(-1,r(\omega))}$ be the closed set of $ S(0,-1;\omega)B(-1,r(\omega)).$ Then, by the above arguments, we know $\mathcal{B}(0,\omega)$ is a random compact set in $\mathcal{V}.$  More precisely,  $\mathcal{B}(0,\omega)$ is a compact absorbing set in $\mathcal{V}$ at time $0.$ Indeed, for $(\nu_{0,n},\tau_{0,n} )\in \mathcal{B}, $ there exists $s(\mathcal{B})\in \mathbb{R}_{-}$ such that if $s\leq s(\mathcal{B}),$ we have
\begin{eqnarray*}
S(0,s;\omega)(\nu_{0,n},\tau_{0,n})=S(0,-1;\omega)S(-1,s;\omega)(\nu_{0,n},\tau_{0,n})\subset S(0,-1;\omega) B(-1,r(\omega))\subset \mathcal{B}(0,\omega) .
\end{eqnarray*}
Therefore, conclusions $(1)-(7)$ of Theorem $3.1$ follows by Theorem $2.2$.
\hspace{\fill}$\square$

\section{Appendix: A Priori Estimates}
\par
\noindent{$5.1.\ \ \mathbf{proof}\ \mathbf{of}\ \ \mathbf{local}-\mathbf{existence}\  \mathbf{of}\  \mathbf{strong}\  \mathbf{solutions}$. }
Using a similar argument in $\cite{ GMR},$ we obtain the existence of local strong solutions to stochastic PEs.
Assume that $\eta$ is the solution of the initial boundary value problem
\begin{eqnarray*}
&&\partial_{t} \eta+ \nabla p_{b_{1}}-\Delta \eta-\partial_{zz}\eta=0,\\
&&\partial_{z} \eta |_{\Gamma_{u},\Gamma_{b}}=0,\ \ \eta\cdot \vec{n}|_{\Gamma_{l}}=0,\ \  \partial_{\vec{n} } \eta \times \vec{n}|_{\Gamma_{l}}=0,\\
&&\int_{-1}^{0}\nabla\cdot \eta dz=0,\\
&&\eta(0,w)=\upsilon_{0}.
\end{eqnarray*}
If $\upsilon_{0}\in \mathcal{V}_{1},$ then, for any $T>0$ and $a.s.w\in \Omega,$
\begin{eqnarray*}
\eta\in L^{\infty}(0,T;\mathcal{V}_{1})\cap L^{2}(0,T;(H^{2}(\mho))^{2}),
\end{eqnarray*}
see, e.g.,$\cite{GMR}$.
Let $u(t)=\upsilon(t)-Z_{1}(t)-\eta(t):=\upsilon(t)-\tilde{Z}_{1}(t)$ and $\theta(t)= T(t)-Z_{2}(t), t\in \mathbb{R}^{+}.$ Given $\mathcal{T}>0,$ a stochastic process $U(t,w)=(\upsilon,T)$ is a strong solution to $(1.1)-(1.5)$ on $[0,\mathcal{T}]$, if and only if $(u,\theta)$ is a strong solution to the following problem on $[0,\mathcal{T}]:$
\begin{eqnarray}
&&\partial_{t} u-\Delta u-\partial_{zz} u+[(u+\tilde{Z}_{1})\cdot\nabla ](u+\tilde{Z}_{1})+\varphi(u+\tilde{Z}_{1}) \partial_{z}(u+\tilde{Z}_{1})\nonumber\\
&& +f( u+\tilde{Z}_{1})^{\bot}+\nabla p_{s}-\int_{-1}^{z}\nabla T dz'=0;\\
&&\partial_{t}\theta -\Delta \theta-\partial_{zz}\theta +[(u+\tilde{Z}_{1})\cdot\nabla ](\theta+Z_{2})+\varphi(u+\tilde{Z}_{1}) \partial_{z} (\theta+Z_{2})=Q;\\
&&\int^{0}_{-1}\nabla\cdot udz=0;\\
&&\partial_{z} u|_{\Gamma_{u}}=\partial_{z}u|_{\Gamma_{b}}=0;
u\cdot \vec{n}|_{\Gamma_{s}}=0, \partial_{\vec{n}}u\times \vec{n}|_{\Gamma_{s}}=0;\\
&&\Big{(}\partial_{z}\theta+\alpha \theta\Big{)}|_{\Gamma_{u}}=\partial_{z}\theta|_{\Gamma_{b}}=0, \ \ \partial_{\vec{n}}\theta|_{\Gamma_{s}}=0;\\
&&(u\Big{|}_{t=0}, \theta\Big{|}_{t=0})=(0, T_{0}).
\end{eqnarray}
\begin{theorem}
[Existence of local solutions to (5.68)-(5.73)] If $Q\in L^{2}(\mho), v_{0}\in \mathcal{V}_{1}$ and $T_{0}\in \mathcal{V}_{2},$ then, for $P-a.e.\omega \in\Omega,$ there exists a stopping time $T^{*}>0$
such that $(u,\theta)$ is a strong solution of the system $(5.68)-(5.73)$ on the interval $[0, T^{*}].$
\end{theorem}
\begin{proof}
We use Faedo-Galerkin method to prove the result. Let $(u_{m},T_{m})$ be an approximate solution for the problem $(5.68)-(5.73)$, where $(u_{m},T_{m})=\Sigma_{i=1}^{m}c_{i,m}(t)\xi_{i}(x)$ and $\{\xi_{i} \}_{i\in \mathbb{N}}$ is a completely orthonormal basis of $\mathcal{V}.$
Then $(u_{m},T_{m})$ satisfies
\begin{eqnarray}
&&\int_{\mho}\upsilon_{m}\cdot \partial_{t} u_{m} + \int_{\mho} \upsilon_{m}\cdot\{[(u_{m}+\tilde{Z}_{1})\cdot \nabla ]( u_{m}+\tilde{Z}_{1})+\varphi(u_{m}+\tilde{Z}_{1}) \partial_{z} (u_{m}+\tilde{Z}_{1}) \}\nonumber\\
&&+\int_{\mho} \upsilon_{m}\cdot(u_{m}+\tilde{Z}_{1})^{\bot}-\int_{\mho} \upsilon_{m}\cdot \int_{-1}^{z}\nabla T_{m}d\lambda+\int_{\mho} \upsilon_{m}\cdot L_{1}u_{m}=0,\\
&&\int_{\mho}\tau_{m}\cdot \partial_{t} T_{m}+\int_{\mho} \tau_{m}\cdot\{[(u_{m}+\tilde{Z}_{1})\cdot\nabla ](T_{m}+Z_{2})+\varphi(u_{m}+\tilde{Z}_{1})  \partial_{z}( T_{m}+Z_{2}) \}\nonumber\\
&&+\int_{\mho}\tau_{m}\cdot L_{2}T_{m}=\int_{\mho}\tau_{m}Q,\\
&&u_{m}(0)=0, T_{m}(0)=T_{0m}\rightarrow T_{0},
\end{eqnarray}
where $\upsilon_{m}\in \mathcal{V}_{1m}, \tau_{m}\in  \mathcal{V}_{2m}  $, $ T_{0m} \in  \mathcal{V}_{2m},$ and $\mathcal{V}_{1m}\times \mathcal{V}_{2m}=\mathrm{span}\{\xi_{1},...,\xi_{m} \}. $ We first estimate $u_{m}$ and $T_{m}$ in $(L^{2}(\mho))^{2}$ and $L^{2}(\mho)$ respectively. Let $\tau_{m}=T_{m}.$ By integration by parts and H$\ddot{o}$lder inequality as well as Sobolev embedding theorem, we have
\begin{eqnarray}
\partial_{t}|T_{m}|^{2}_{2}+C\|T_{m}\|_{1}^{2}&\leq& C|Q|^{2}_{2}+\int_{\mho}T_{m}\{[(u_{m}+\tilde{Z}_{1})\cdot \nabla  ]Z_{2}+\varphi(u_{m}+\tilde{Z}_{1} )\partial_{z} Z_{2} \}\nonumber\\
&\leq&C|Q|_{2}^{2}+ | \nabla Z_{2}|_{\infty}|T_{m} |_{2} |u_{m}+\tilde{Z}_{1}|_{2}+|\partial_{z}Z_{2}|_{\infty}|T_{m}|_{2}\|u_{m}+\tilde{Z}_{1}\|_{1}\nonumber\\
&\leq& \varepsilon \|u_{m}\|_{1}^{2}+C\|Z_{2}\|_{3}^{2}|T_{m}|_{2}^{2}+C\|Z_{2}\|_{3}|T_{m}|_{2}\|\tilde{Z}_{1}\|_{1}+C|Q|_{2}^{2}.
\end{eqnarray}
Using H$\ddot{o}$lder inequality, Sobolev embedding theorem and interpolation inequalities, we have
\begin{eqnarray}
&&\int_{\mho}u_{m}[(\tilde{Z}_{1}\cdot \nabla)u_{m}+(u_{m}\cdot \nabla)\tilde{Z}_{1}+(\tilde{Z}_{1}\cdot\nabla)\tilde{Z}_{1}]\nonumber\\
&\leq&|\tilde{Z}_{1}|_{\infty}|u_{m}|_{2}|\nabla u_{m}|_{2}+|u_{m}|_{3}^{2}|\nabla \tilde{Z}_{1}|_{3}+|\tilde{Z}_{1}|_{\infty}|u_{m}|_{2}|\nabla \tilde{Z}_{1}|_{2}\nonumber\\
&\leq&\varepsilon \|u_{m}\|_{1}^{2}+C\|\tilde{Z}_{1}\|_{2}^{2}(|u_{m}|_{2}^{2}+|u_{m}|_{2} ).
\end{eqnarray}
By virtue of  Minkowski inequality, interpolation inequalities as well as Sobolev imbedding theorem, we obtain
\begin{eqnarray*}
&&\int_{\mho}u_{m}\varphi(u_{m} )\partial_{z} \tilde{Z}_{1} d\mho\nonumber\\
&\leq&\int_{M}\Big{(}\int_{-1}^{0} |\mathrm{div} u_{m}|dz \int_{-1}^{0}|u_{m}||\partial_{z}\tilde{Z}_{1}|dz\Big{)} dM\nonumber\\
&\leq&\int_{M}\int_{-1}^{0}|\mathrm{div} u_{m}|dz (\int_{-1}^{0}|u_{m}|^{2}dz)^{\frac{1}{2}}(\int_{-1}^{0}|\partial_{z}\tilde{Z}_{1}|^{2}dz)^{\frac{1}{2}}dM\nonumber\\
&\leq& |\mathrm{div} u_{m}|_{2}\Big{(}\int_{-1}^{0}|u_{m}|_{L^{4}(M)}^{2}    \Big{)}^{\frac{1}{2}}\Big{(}\int_{-1}^{0}|\partial_{z}\tilde{Z}_{1}|_{L^{4}(M)}^{2}    \Big{)}^{\frac{1}{2}}\nonumber\\
&\leq&C \|u_{m}\|_{1}\Big{(}\int_{-1}^{0}|u_{m}|_{L^{2}(M)}\|u_{m}\|_{H^{1}(M)}dz\Big{)}^{\frac{1}{2}}
\Big{(}\int_{-1}^{0}\|\tilde{Z}_{1}\|_{H^{1}(M)}\|\tilde{Z}_{1}\|_{H^{2}(M)}dz\Big{)}^{\frac{1}{2}}\nonumber\\
&\leq&C|u_{m}|_{2}^{\frac{1}{2}}\|u_{m}\|_{1}^{\frac{3}{2}}\|\tilde{Z}_{1}\|_{1}^{\frac{1}{2}}\|\tilde{Z}_{1}\|_{2}^{\frac{1}{2}}\nonumber\\
&\leq&\varepsilon \|u_{m}\|_{1}^{2}+C|u_{m}|_{2}^{2}\|\tilde{Z}_{1}\|_{1}^{2}\|\tilde{Z}_{1}\|_{2}^{2}.
\end{eqnarray*}
Analogously,  we have
\begin{eqnarray*}
&&\int_{\mho}u_{m}\varphi(\tilde{Z}_{1})\partial_{z} u_{m}d\mho\\
&\leq&|\partial_{z}u_{m}|_{2}\int_{-1}^{0}|\mathrm{div} \tilde{Z}_{1}  |_{L^{4}(M)}dz|\Big{(}\int_{-1}^{0}|u_{m}|_{L^{4}(M)}^{2}dz\Big{)}^{\frac{1}{2}}  \\
&\leq&c\|u_{m}\|_{1}\int_{-1}^{0}\|\tilde{Z}_{1}\|_{H^{1}(M)}^{\frac{1}{2}}\|\tilde{Z}_{1}\|_{H^{2}(M)}^{\frac{1}{2}}dz
\Big{(}\int_{-1}^{0}|u_{m}|_{L^{2}(M)}\|u_{m}\|_{H^{1}(M)}dz\Big{)}^{\frac{1}{2}}   \\
&\leq&\varepsilon \|u_{m}\|_{1}^{2}+C |u_{m}|_{2}^{2}\|\tilde{Z}_{1}  \|_{1}^{2}\|\tilde{Z}_{1}  \|_{2}^{2}
\end{eqnarray*}
and
\begin{eqnarray*}
&&\int_{\mho}u_{m}\varphi(\tilde{Z}_{1}) \partial_{z} \tilde{Z}_{1}d\mho\\
&\leq&|\partial_{z}\tilde{Z}_{1}|_{2}\int_{0}^{1}|\mathrm{div} \tilde{Z}_{1}  |_{L^{4}(M) }dz\Big{(}\int_{-1}^{0}|u_{m}|_{L^{4}(M)}^{2}dz\Big{)}^{\frac{1}{2}}  \\
&\leq&c\|\tilde{Z}_{1}\|_{1}\int_{-1}^{0}\|\tilde{Z}_{1}\|_{H^{1}(M)}^{\frac{1}{2}}\|\tilde{Z}_{1}\|_{H^{2}(M)}^{\frac{1}{2}}dz
\Big{(}\int_{-1}^{0}|u_{m}|_{L^{2}(M)}\|u_{m}\|_{H^{1}(M)}dz\Big{)}^{\frac{1}{2}}   \\
&\leq&\varepsilon \|u_{m}\|_{1}^{2}+C|u_{m}|_{2}^{2}+C\|\tilde{Z}_{1}\|_{1}^{3}\|\tilde{Z}_{1}\|_{2}.
\end{eqnarray*}
Combining the above inequalities, we arrive at
\begin{eqnarray}
&&\int_{\mho}u_{m}[\varphi(u_{m} )\partial_{z} \tilde{Z}_{1} +\varphi(\tilde{Z}_{1})\partial_{z} u_{m}+\varphi(\tilde{Z}_{1}) \partial_{z} \tilde{Z}_{1}]d\mho\nonumber\\
&\leq&\varepsilon \|u_{m}\|_{1}^{2}+C|u_{m}|_{2}^{2}\|\tilde{Z}_{1}\|_{1}^{2}\|\tilde{Z}_{1}\|_{2}^{2} +C|u_{m}|_{2}^{2}+C\|\tilde{Z}_{1}\|_{1}^{3}\|\tilde{Z}_{1}\|_{2}.
\end{eqnarray}
Let $\upsilon_{m}=u_{m}.$ Due to integration by parts, we get
\begin{eqnarray*}
\int_{\mho}[(u_{m}\cdot \nabla)u_{m}+\varphi(u_{m} )\partial_{z} u_{m} ]\cdot u_{m}=0
\end{eqnarray*}
and
\begin{eqnarray}
\int_{\mho}u_{m}\cdot\int_{-1}^{z}\nabla T_{m}d\lambda= -\int_{\mho} \nabla \cdot u_{m} \int_{-1}^{z}T_{m}d\lambda.
\end{eqnarray}
Therefore, from $(5.74)$ and $(5.79)-(5.80),$ we conclude that
\begin{eqnarray*}
\partial_{t} |u_{m}|^{2}_{2}+\|u_{m}\|^{2}_{1}&\leq& C\|\tilde{Z}_{1}\|_{2}^{2}+C\|\tilde{Z}_{1}\|_{1}^{3}\|\tilde{Z}_{1}\|_{2}\\
&&+     C|u_{m}|_{2}^{2}(1+\|\tilde{Z}_{1}\|_{2}^{2}+\|\tilde{Z}_{1}\|_{1}^{2}\|\tilde{Z}_{1}\|_{2}^{2})+C|T_{m}|^{2}_{2},
\end{eqnarray*}
which together with $(5.77)$ implies
\begin{eqnarray*}
&&\partial_{t} (|T_{m}|_{2}^{2}+|u_{m}|^{2}_{2} )+\|T_{m}\|^{2}_{1}+\|u_{m}\|^{2}_{1}\\
&\leq& C(1+\|\tilde{Z}_{1}\|_{2}^{2}+\|\tilde{Z}_{1}\|_{1}^{2}\|\tilde{Z}_{1}\|_{2}^{2}   +\|Z_{2}\|_{3}^{2} )(|T_{m}|_{2}^{2}+|u_{m}|^{2}_{2} )\\
&&+C(|Q|_{2}^{2}+\|\tilde{Z}_{1}\|_{1}^{2}+\|\tilde{Z}_{1}\|_{1}^{3}\|\tilde{Z}_{1}\|_{2}+\|Z_{2}\|_{3}^{2}\|\tilde{Z}_{1}\|_{1}^{2}).
\end{eqnarray*}
From the estimate above we infer that $T_{m}$ and $u_{m}$ is uniformly with respect to $m$ bounded in
 $C([0,T^{*}];$ $ (L^{2}(\mho))^{2})\cap L^{2}([0,T^{*}]; (H^{1}(\mho))^{2})$ for any $T^{*}>0.$ In the following, we will estimate $u_{m} $ and $T_{m}$ in $(H^{1}(\mho))^{2}$ and $H^{1}(\mho)$ respectively.
Using H$\ddot{o}$lder inequality and interpolation inequality, we have
\begin{eqnarray}
&&\int_{\mho}L_{1}u_{m}\cdot [(u_{m}+\tilde{Z}_{1})\cdot \nabla](u_{m}+\tilde{Z}_{1} )\nonumber\\
&\leq&\|u_{m}\|_{2}|\nabla u_{m}|_{3}|u_{m}|_{6}+|\tilde{Z}_{1}|_{\infty}\|u_{m}\|_{2}|\nabla u_{m}|_{2}\nonumber\\
&&+|\tilde{Z}_{1}|_{\infty}\|u_{m}\|_{2}(\|u_{m}\|_{1}+\|\tilde{Z}_{1}\|_{1} )\nonumber\\
&\leq&C\|u_{m}\|_{2}^{\frac{3}{2}}\|u_{m}\|_{1}^{\frac{3}{2}}+\|u_{m}\|_{2}\|u_{m}\|_{1}\|\tilde{Z}_{1}\|_{1}^{\frac{1}{2}}\|\tilde{Z}_{1}\|_{2}^{\frac{1}{2}}
+C\|u_{m}\|_{2}\|\tilde{Z}_{1}\|_{2}^{\frac{1}{2}}\tilde{Z}_{1}\|_{1}^{\frac{3}{2}}\nonumber\\
&\leq&\varepsilon \|u_{m}\|_{2}^{2}+C\|u_{m}\|_{1}^{6}+C\|u_{m}\|_{1}^{2}\|\tilde{Z}_{1}\|_{1}\|\tilde{Z}_{1}\|_{2}+C\|\tilde{Z}_{1}\|_{2}\|\tilde{Z}_{1}\|_{1}^{3}.
\end{eqnarray}
Taking a similar argument in $\cite{CT1} $ to obtain the estimates of $\tilde{v}$ in $(L^{6}(\mho))^{2}$,  we get
\begin{eqnarray}
&&\int_{\mho}L_{1}u_{m}\cdot \varphi(u_{m}+\tilde{Z}_{1})\partial_{z} (u_{m}+\tilde{Z}_{1})\nonumber\\
&\leq &
\varepsilon \|u_{m}\|^{2}_{2}+C\|u_{m}\|^{2}_{2}\|u_{m}\|_{1}+C(1+\|u_{m}\|^{2}_{1})\|\tilde{Z}_{1}\|_{1}^{2}\|\tilde{Z}_{1}\|^{2}_{2}.
\end{eqnarray}
Let $\upsilon_{m}=L_{1}u_{m}$ in $(5.74).$ Since by Young inequality we get
\[
\int_{\mho} L_{1}u_{m}\cdot u_{m}^{\bot}\leq \varepsilon \|u_{m}\|_{2}^{2}+C|u_{m}|_{2}^{2},
\]
which combined with $(5.81)-(5.82)$ implies
\begin{eqnarray}
\partial_{t} \|u_{m}\|_{1}^{2}+\|u_{m} \|_{2}^{2}&\leq& C\|u_{m}\|^{2}_{2}\|u_{m}\|_{1}+C\Big{(}1+\|\tilde{Z}_{1}\|^{2}_{1} \|\tilde{Z}_{1}\|^{2}_{2}+\|T_{m}\|_{1}^{2} \Big{)}\nonumber\\
&&+ C\Big{(} \|u_{m}\|_{1}^{4}+\|\tilde{Z}_{1}\|^{2}_{1}\|\tilde{Z}_{1}\|^{2}_{2}+\|\tilde{Z}_{1}\|_{1}\|\tilde{Z}_{1}\|_{2}\Big{)}\|u_{m}\|_{1}^{2}.
\end{eqnarray}
As $u_{m}(0)=0$ and $u_{m}$ is continuous in time with values in $H^{1}(\mho)$ and progressively measurable, we can choose a stopping time $\mathcal{T}_{m}$ such that
\begin{eqnarray*}
\sup\limits_{0\leq t\leq \mathcal{T}_{m} }\|u_{m}(t)\|_{1}\leq \frac{1}{2C}.
\end{eqnarray*}
Next we will show that $\mathcal{T}_{m}$ can be independent of $m$. Integrating $(5.83)$ from $0$ to t, for $t\in [0, \mathcal{T}_{m}],$ we get
\begin{eqnarray*}
\|u_{m}(t)\|_{1}^{2}+\int_{0}^{t}\|u_{m}(t)\|_{2}^{2}ds\leq C\int_{0}^{t}(1+\|\tilde{Z}_{1}\|_{1}^{2}\|\tilde{Z}_{1}\|_{2}^{2}+\|T_{m}\|_{1}^{2} )ds.
\end{eqnarray*}
Since
$
\int_{0}^{t}\|T_{m}(t)\|_{1}^{2}dt
$ is uniformly bounded for arbitrary $t>0$, so we can choose a small $T^{*}$ independent of $m$ such that
\begin{eqnarray*}
C\int_{0}^{T^{*}}(1+\|\tilde{Z}_{1}\|_{1}^{2}\|\tilde{Z}_{1}\|_{2}^{2}+\|T_{m}\|_{1}^{2} )ds\leq \frac{1}{4C^{2}},
\end{eqnarray*}
which implies that for all $m,\mathcal{ T}_{m}$ can be chosen to be equal to $T^{*}$ such that
\begin{eqnarray}
\sup\limits_{t\in[0,T^{*}]}\|u_{m}(t)\|_{1}^{2}+\int_{0}^{T^{*}}\|u_{m}(t)\|_{2}^{2}ds\leq C(T^{*}, \tilde{Z}_{1},Z_{2},Q).
\end{eqnarray}
Therefore, $u_{m}$ uniformly bounded with respect to $m$ in $L^{\infty}([0, T^{*}]; (H^{1}(\mho))^{2})\cap L^{2}([0, T^{*}]; (H^{2}(\mho))^{2}).$ To estimate $T_{m}$ in $H^{1}(\mho),$ setting $\tau_{m}=L_{2}T_{m}$ and using the methods to prove $(5.81)$ and $(5.82),$ we arrive at
\begin{eqnarray*}
&&\int_{\mho}(L_{2}T_{m})[(u_{m}+\tilde{Z}_{1})\cdot \nabla] (T_{m}+Z_{2})\nonumber\\
&\leq& \varepsilon \|T_{m}\|^{2}_{2}+ C\|u_{m}\|^{2}_{1}\|Z_{2}\|_{2}^{2}+C\|T_{m}\|_{1}^{2}\|u_{m}\|^{4}_{1}+C\|\tilde{Z}_{1}\|_{2}^{2} \|T_{m}\|^{2}_{1} +C\|\tilde{Z}_{1}\|_{2}^{2}\|Z_{2}\|_{2}^{2}
\end{eqnarray*}
and
\begin{eqnarray*}
&&\int_{\mho}(L_{2}T_{m})\varphi(u_{m}+\tilde{Z}_{1})\partial_{z} (T_{m}+Z_{2})\nonumber\\
&\leq &\varepsilon \|T_{m}\|^{2}_{2}
+ C\|T_{m}\|^{2}_{1}(\|u_{m}\|_{1}^{2}\|u_{m}\|_{2}^{2}+\|\tilde{Z}_{1}\|_{1}^{2}\|\tilde{Z}_{1}\|_{2}^{2} )+\|Z_{2}\|_{2}^{2}\|u_{m}\|_{1}\|u_{m}\|_{2}+C\|\tilde{Z}_{1}\|_{2}^{2}\|Z_{2}\|_{2}^{2}.
\end{eqnarray*}
Combining the above bounds yields
\begin{eqnarray}
&&\partial_{t}\|T_{m}\|^{2}_{1}+\|T_{m}\|^{2}_{2}\nonumber\\
&\leq& C\|T_{m}\|_{1}^{2}(\|u_{m}\|^{4}_{1}+\|u_{m}\|^{2}_{1}\|u_{m}\|^{2}_{2}+\|\tilde{Z}_{1}\|_{2}^{2}+\|\tilde{Z}_{1}\|_{1}^{2}\|\tilde{Z}_{1}\|_{2}^{2} ) \nonumber\\
&&+C\|u_{m}\|_{1}^{2}\|Z_{2}\|_{2}^{2}+C\|u_{m}\|_{1}\|u_{m}\|_{2}\|Z_{2}\|^{2}_{2}+C\|\tilde{Z}_{1}\|_{1}^{2}\|Z_{2}\|^{2}_{2}+C|Q|^{2}_{2},
\end{eqnarray}
which together with $(5.84)$ implies that $T_{m}$ is uniformly bounded with respect to $m$ in $L^{\infty}([0, T^{*}];$ $ (H^{1}(\mho))^{2})\cap L^{2}([0, T^{*}]; (H^{2}(\mho))^{2}).$ Therefore, by $(5.84), (5.85)$ and a standard argument (see $\cite{L,T1}$), that at this point do not require any novelty,
one can prove further bounds on the time derivative of $(u_{m},T_{m})$, and reasoning on weakly
and strongly convergent subsequences one gets the existence of a solution $(u,\theta)$ with
the regularity specified by the theorem.
\hspace{\fill}$\square$
\end{proof}
\par
\noindent{$5.2. \ \mathbf{A}\ \ \mathbf{Priori}\ \ \mathbf{Estimates}\ \mathbf{For}\ \ \mathbf{The}\ \ \mathbf{Global}\ \ \mathbf{Existence}\  \mathbf{Of}\  \mathbf{Strong}\  \mathbf{Solutions}$. }
In the previous subsections we have proven the existence of strong solution
for a short interval of time, whose length depends on the initial data and other physical parameters of the system $(5.68)-(5.73).$
Let $(\upsilon_{0},T_{0})$ be a given initial condition. In this section we will consider the strong solution that corresponds to
this initial data in its maximal interval of existence $[0, \tau_{*}).$ Specially, for fixed $\omega\in \Omega,$
we will establish $a\ priori$ upper estimates for various norms of this solution in the interval $[0, \tau_{*}).$
In particular, we will show that if $\tau_{*}<\infty$ then the $(H^{1}(\mho))^{3}$ norm of the strong solution is
bounded over the interval $[0, \tau_{*}). $
This key observation plays an important role in the proof of global regularity of strong solution to the system $(5.68)-(5.73).$
\par
Similarly as in $\cite{CF},$  to study the long time behavior of stochastic PEs, we introduce a modified stochastic convolution. For $j=1,2,\ \ \beta>1$ and $t\in \mathbb{R},$ we define
\begin{eqnarray*}
Z_{j}(t):= \int_{-\infty}^{t}e^{-(t-s)(A_{j}+\beta)}dW^{H}_{j}(s)  .
\end{eqnarray*}
Then $ Z_{j}$ is the mild solution of the linear equation
\begin{eqnarray*}
&&dZ_{j}=(-A_{j}Z_{j}-\beta Z_{j})dt+dW^{H}_{j},\\
&&Z_{j}(0)=z_{0}^{j},
\end{eqnarray*}
where $z^{j}_{0}=\int_{-\infty}^{0}e^{s(A_{j}+\beta)}dW^{H}_{j}(s),\ j=1,2 .$ It is easy to see that if $(u, \theta)$ is the global solution of the following system $(5.86)-(5.91)$,
then equivalently it is also the global solution of $(5.68)-(5.73).$
\begin{eqnarray}
&&\partial_{t} u-\Delta u-\partial_{zz} u+[(u+Z_{1})\cdot\nabla ](u+Z_{1})+\varphi(u+Z_{1}) \partial_{z}(u+Z_{1})\nonumber\\
&& +f( u+Z_{1})^{\bot}+\nabla p_{s}-\int_{-1}^{z}\nabla \theta dz'=\beta Z_{1};\\
&&\partial_{t}\theta -\Delta \theta-\partial_{zz}\theta+[(u+Z_{1})\cdot\nabla ](\theta+Z_{2})+\varphi(u+Z_{1}) \partial_{z} (\theta+Z_{2})=Q+\beta Z_{2};\\
&&\int^{0}_{-1}\nabla\cdot udz=0;\\
&&\partial_{z} u|_{\Gamma_{u}}=\partial_{z}u|_{\Gamma_{b}}=0;
u\cdot \vec{n}|_{\Gamma_{s}}=0, \partial_{\vec{n}}u\times \vec{n}|_{\Gamma_{s}}=0;\\
&&\Big{(}\partial_{z}\theta+\alpha \theta\Big{)}|_{\Gamma_{u}}=\partial_{z}\theta|_{\Gamma_{b}}=0, \ \ \partial_{\vec{n}}\theta|_{\Gamma_{s}}=0;\\
&&(u\Big{|}_{t=0}, \theta\Big{|}_{t=0})=(v_{0}-z^{1}_{0},T_{0}-z^{2}_{0}).
\end{eqnarray}
We denote by
\begin{eqnarray*}
\bar{\phi}(x,y)=\int_{-1}^{0}\phi(x,y,\xi)d\xi,\ \ \forall\ (x,y)\in M.
\end{eqnarray*}
In particular,
\begin{eqnarray*}
\bar{u}(x,y)=\int_{-1}^{0}u(x,y,\xi)d\xi,\ \ \ \mathrm{in}\ M.
\end{eqnarray*}
Let
\begin{eqnarray*}
\tilde{u}=u-\bar{u}.
\end{eqnarray*}
Notice that
\begin{eqnarray*}
\tilde{\bar{u}}=0
\end{eqnarray*}
and
\begin{eqnarray*}
\nabla \cdot \bar{u}=0, \ \ \ \ \mathrm{in}\ M.
\end{eqnarray*}
In the following, we will study the properties of $\bar{u}$ and $\tilde{u}.$ By taking the average of equations $(5.86)$ in the $z$ direction,
over the interval $(-1,0),$ and using boundary conditions $(5.89),$ we have
\begin{eqnarray}
&&\partial_{t} \bar{u}+\overline{[(u+Z_{1})\cdot\nabla ](u+Z_{1})+\varphi(u+Z_{1})\partial_{z}(u+Z_{1}) }
+( \bar{u}+\bar{Z_{1}})^{\bot}\nonumber\\
&&+\nabla p_{s}-\int_{-1}^{0}\int_{-1}^{z}\nabla \theta dz'dz-\Delta \bar{u}=\beta\bar{Z_{1} }.
\end{eqnarray}
Based on the above and integration by parts we have
\begin{eqnarray}
\int_{-1}^{0}\varphi(u+Z_{1})\partial_{z}( u+Z_{1}) dz=\int_{-1}^{0}(u+Z_{1}) \nabla \cdot (u+Z_{1}) dz=
\int_{-1}^{0}\nabla\cdot(\tilde{u}+\tilde{Z_{1}})
(\tilde{u}+\tilde{Z_{1}})dz
\end{eqnarray}
and
\begin{eqnarray}
\int_{-1}^{0}(u+Z_{1})\cdot \nabla(u+Z_{1}) dz= \int_{-1}^{0}(\tilde{u}+\tilde{Z_{1}})\cdot \nabla(\tilde{u}+\tilde{Z_{1}})dz
+ (\bar{u}+\bar{Z_{1}})\cdot \nabla(\bar{u}+\bar{Z_{1}}).
\end{eqnarray}
Therefore, substituting $(5.94)$ and $(5.93)$ into $(5.92),$  we can see that $\bar{u}$ satisfies the following equations and boundary conditions:
\begin{eqnarray}
&& \partial_{t} \bar{u}-\Delta \bar{u}+\overline{(\tilde{u}+\tilde{Z_{1}})\cdot \nabla(\tilde{u}+\tilde{Z_{1}})}
+\overline{(\tilde{u}+\tilde{Z_{1}}) \nabla\cdot(\tilde{u}+\tilde{Z_{1}})}+ (\bar{u}+\bar{Z_{1}})\cdot \nabla(\bar{u}+\bar{Z_{1}})\nonumber\\
&&+(\bar{u}+\bar{Z_{1}})^{\bot}+\nabla p_{s}-\nabla \int_{-1}^{0}
\int_{-1}^{z}\theta(x,y,\lambda,t)d\lambda dz=\beta\bar{Z_{1} },
\end{eqnarray}
\begin{eqnarray}
\nabla \cdot \bar{u}=0,\ \ \ \mathrm{in}\ M,
\end{eqnarray}
\begin{eqnarray}
 \bar{u}\cdot \vec{n} =0,\ \partial_{\vec{n}} \bar{u}\times \vec{n}=0\ \ \ \mathrm{on}\ M.
\end{eqnarray}
By subtracting $(5.95)$ from $(5.86)$ and using $(5.89),(5.97)$ we conclude that $\tilde{u}$ satisfies the following equations and boundary
conditions:
\begin{eqnarray}
&& \partial_{t} \tilde{u}-\Delta \tilde{u}- \partial_{zz}\tilde{u}+[(\tilde{u}+\tilde{Z_{1}})]\cdot
 \nabla(\tilde{u}+\tilde{Z_{1}} )+\varphi(\tilde{u}+\tilde{Z_{1}}) \partial_{z} (\tilde{u}+\tilde{Z_{1}} ) \nonumber\\
 && +[ (\tilde{u}+\tilde{Z_{1}} )\cdot \nabla](\bar{u}+\bar{Z_{1}})+[(\bar{u}+\bar{Z_{1}}) \cdot \nabla](\tilde{u}+\tilde{Z_{1}})
 - \overline{(\tilde{u}+\tilde{Z_{1}})\cdot \nabla(\tilde{u}+\tilde{Z_{1}})}\nonumber\\
  &&-\overline{(\tilde{u}+\tilde{Z_{1}}) \nabla\cdot(\tilde{u}+\tilde{Z_{1}})}+(\tilde{u}+\tilde{Z_{1}} )^{\bot}-\int_{-1}^{z}\nabla \theta dz'+
  \int_{-1}^{0}  \int_{-1}^{z}\nabla \theta dz'dz=0,
\end{eqnarray}
\begin{eqnarray}
\partial_{z} \tilde{u}|_{z=0}=0,\ \  \partial_{z} \tilde{u}|_{z=-1}=0,\ \  \tilde{u}\cdot \vec{n}|_{\Gamma_{s}}=0,\ \
\partial_{ \vec{n}} \tilde{u}\times \vec{n}|_{\Gamma_{s}}=0.
\end{eqnarray}

\par
\noindent{$5.2.1.\ \ L^{2}\ \mathbf{estimates}.$}
We take the inner product of equation $(5.87)$ with $\theta,$ in $L^{2}(\mho),$ and get
\begin{eqnarray*}
&&\frac{1}{2}\partial_{t} |\theta|^{2}_{2}+|\nabla \theta|^{2}_{2}+|\theta_{z}|^{2}_{2}+\alpha|\theta(z=0)|^{2}_{2}\\
&&=\int_{\mho}Q\theta-\int_{\mho}\Big{(}(u+Z_{1})\cdot\nabla (\theta+Z_{2})+\varphi (u+Z_{1}) \partial_{z} (\theta+Z_{2})-\beta Z_{2}\Big{)}\theta.
\end{eqnarray*}
By integration by parts, we have
\begin{eqnarray*}
\int_{\mho}\Big{(}(u+Z_{1})\cdot\nabla \theta+\varphi (u+Z_{1})\partial_{z} \theta \Big{)}\theta=0.
\end{eqnarray*}
On account of H$\ddot{o}$lder inequality and Sobolev imbedding theorem, we obtain
\begin{eqnarray*}
\int_{\mho} [\varphi (u+Z_{1})\partial_{z} Z_{2}]\theta&\leq& C|\partial_{z} Z_{2}|_{\infty}|\nabla\cdot u+\nabla \cdot Z_{1}|_{2}|\theta|_{2}\nonumber\\
&\leq& \varepsilon |\nabla \cdot u |_{2}^{2}+C\| Z_{2}\|_{3}^{2}|\theta|_{2}^{2}+C\|Z_{1}\|_{2}^{2}.
\end{eqnarray*}
Since $u\cdot \vec{n}=0  $ on $\partial \mho,$ by Exercise II $\!\!.5.15$ in $\cite{Gp}$ there exists a positive constant $C=C(\mho)$ such that $|u|_{2}\leq C|\nabla u|_{2}.$
 Therefore, taking a similar argument as above we have
\begin{eqnarray*}
\int_{\mho} [(u+Z_{1})\cdot \nabla Z_{2}]\theta\leq \varepsilon |\nabla u|_{2}^{2}+C|Z_{1}|_{2}^{2}+C\|Z_{2}\|_{3}^{2}|\theta|_{2}^{2}.
\end{eqnarray*}
Combining the above bounds, we arrive at
\begin{eqnarray*}
&&\partial_{t} |\theta|^{2}_{2}+(2-\varepsilon)|\nabla \theta|^{2}_{2}+(2-\varepsilon)|\theta_{z}|^{2}_{2}+(2\alpha-\varepsilon)|\theta(z=0)|^{2}_{2}\\
&\leq & 2\varepsilon| \nabla u|_{2}^{2} +C|Q|_{2}^{2}+C\|Z_{1}\|_{1}^{2}+C|Z_{2}|_{2}^{2}+C\|Z_{2}\|_{3}^{2}|\theta|_{2}^{2},
\end{eqnarray*}
where we have used $|\theta|_{2}^{2}\leq C\|\theta\|_{1}^{2} $ and
\begin{eqnarray*}
c\|\theta\|_{1}^{2}\leq |\nabla \theta|^{2}_{2}+|\theta_{z}|^{2}_{2}+\alpha|\theta(z=0)|^{2}_{2}\leq C\|\theta\|_{1}^{2},\ \mathrm{for}\ \mathrm{some}\ c,\ C>0.
\end{eqnarray*}
To obtain the global well-posedness, we take an analogous argument in section 5.1 and reach
\begin{eqnarray*}
\partial_{t}|u|^{2}_{2}+2|\nabla u|^{2}_{2}+2|u_{z}|_{2}^{2}\leq C|u|^{2}_{2}(\|Z_{1}\|^{2}_{2}+\|Z_{1}\|^{4}_{2} )+C\|Z_{1}\|^{2}_{2}+C|\theta|^{2}_{2},
\end{eqnarray*}
which together with the bound of $|\theta|^{2}_{2}$ and Gronwall inequality imply that
\begin{eqnarray}
\sup\limits_{t \in [0, \tau_{*})}(|u(t)|^{2}_{2}+|\theta(t)|^{2}_{2} ) +\int_{0}^{\tau^{*}}\|u(s)\|^{2}_{1}+\|\theta(s)\|_{1}^{2}ds\leq C(v_{0}, T_{0}, Z_{1}, Z_{2} ).
\end{eqnarray}
In order to prove the existence of random attractor of stochastic PEs, we need to get a more delicate and careful $a\ priori$ estimate of $|u(t)|^{2}_{2}.$
From $(5.80)$ and H$\ddot{o}$lder inequality we infer that there exists $c\in (\frac{1}{2\alpha}\vee \frac{1}{2}, 4 )$ such that
\begin{eqnarray*}
\int_{\mho}u\cdot\int_{-1}^{z}\nabla\theta dz'\leq |\nabla u|_{2}|\theta|_{2}\leq \frac{c}{2} |\nabla u|_{2}^{2}+\frac{1}{2 c} |\theta|_{2}^{2},
\end{eqnarray*}
where we assume $\alpha > \frac{1}{8}$ and $\frac{1}{2\alpha}\vee \frac{1}{2}=\max\{\frac{1}{2\alpha}  , \frac{1}{2}     \} .$ Then, by the estimates of $|u_{m}|_{2}$ in section 5.1 we have
\begin{eqnarray*}
\partial_{t}|u|^{2}_{2}+2|\nabla u|^{2}_{2}+2|u_{z}|_{2}^{2}\leq C|u|^{2}_{2}(\|Z_{1}\|^{2}_{2}+\|Z_{1}\|^{4}_{2} )+C\|Z_{1}\|^{2}_{2}+\frac{1}{2 c}|\theta|^{2}_{2}+\frac{c}{2}|\nabla u|^{2}_{2}.
\end{eqnarray*}
By $(48)$ in $\cite{CT1}$,
\begin{eqnarray*}
|\theta|^{2}_{2}\leq 2|\theta_{z}|^{2}_{2}+2|\theta(z=0)|^{2}_{2},
\end{eqnarray*}
which implies
\begin{eqnarray*}
&&\frac{1}{2 c} |\theta|_{2}^{2}
\leq  \frac{1}{c}|\theta_{z}|^{2}_{2}+\frac{1}{c}|\theta(z=0)|^{2}_{2}<2|\theta_{z}|^{2}_{2}+2\alpha|\theta(z=0)|^{2}_{2} .
\end{eqnarray*}
In view of the bounds of $\theta $ and $u,$ we conclude that
\begin{eqnarray}
&&\partial_{t} (|u|^{2}_{2}+|\theta|_{2}^{2})+(2-\frac{c}{2}-3\varepsilon)(| \nabla u|^{2}_{2}+|\nabla \theta|_{2}^{2})
+(2-\frac{1}{c}-\varepsilon)|u_{z}+\theta_{z}|_{2}^{2}\nonumber\\
&&+(2\alpha-\frac{1}{c}-\varepsilon)|\theta(z=0)|_{2}^{2}\leq C(|u|_{2}^{2}+|\theta|_{2}^{2})(\|Z_{1}\|_{2}^{2}+\|Z_{1}\|_{2}^{4}+\|Z_{2}\|_{3}^{2})\nonumber\\
&&+C(|Q|_{2}^{2}+\|Z_{1}\|_{1}^{2}).
\end{eqnarray}
Since there exists $\gamma_{1}>0$ such that
\begin{eqnarray*}
(2-\frac{c}{2}-3\varepsilon)(| \nabla u|^{2}_{2}\!\!\!\!&+&\!\!\!\!|\nabla \theta|_{2}^{2})
+(2-\frac{1}{c}-\varepsilon)|u_{z}+\theta_{z}|_{2}^{2}\\
&&+(2\alpha-\frac{1}{c}-\varepsilon)|\theta(z=0)|_{2}^{2}>\gamma_{1} (|u|^{2}_{2}+|\theta|_{2}^{2} ).
\end{eqnarray*}
Then, for each $t\in [0, \tau_{*}),$ by Gronwall inequality we obtain
\begin{eqnarray}
|u(t)|^{2}_{2}+|\theta(t)|^{2}_{2}&\leq& (|u(0)|^{2}_{2}+|\theta(0)|^{2}_{2})
e^{\int_{0}^{t}-\gamma_{1}+C(\|Z_{1}\|_{2}^{2}+\|Z_{1}\|_{2}^{4}+\|Z_{2}\|_{3}^{2})ds }\nonumber\\
&&+\int_{0}^{t}e^{\int_{s}^{t}-\gamma_{1}+C(\|Z_{1}\|_{2}^{2}+\|Z_{1}\|_{2}^{4}+\|Z_{2}\|_{3}^{2})ds } (|Q|_{2}^{2}+\|Z_{1}\|_{1}^{2})ds,
\end{eqnarray}
which also implies $(5.100).$
\par
\noindent{$5.2.2.\ \ L^{4}\ \mathbf{estimates}\ \mathbf{about}\ \theta\ \mathbf{and}\ \tilde{u}$. }
Taking the inner product of the equation $(5.87)$ with $\theta^{3}$ in $L^{2}(\mho)$ and obtain
\begin{eqnarray}
&&\frac{1}{4}\partial_{t}|\theta|^{4}_{4}+\frac{3}{4}|\nabla \theta^{2}|^{2}_{2}+\frac{3}{4}|(\theta^{2})_{z}|_{2}^{2}+\alpha\int_{M}|\theta(z=0)|^{4}\nonumber\\
&=&\int_{\mho}Q\theta^{3}-\int_{\mho}[(u+Z_{1})\cdot \nabla (\theta+Z_{2})+\varphi(u+Z_{1})\partial_{z}( \theta+Z_{2}) +\beta Z_{2}]\theta^{3}.
\end{eqnarray}
By integration by parts, we have
\begin{eqnarray}
\int_{\mho}[(u+Z_{1})\cdot \nabla \theta+\varphi(u+Z_{1})\partial_{z}\theta ]\theta^{3}=0.
\end{eqnarray}
Using h$\ddot{o}$lder inequality,
\begin{eqnarray*}
\int_{\mho}[\varphi(u+Z_{1})\partial_{z} Z_{2}  ]\theta^{3}&\leq&|\partial_{z}Z_{2}|_{\infty}|\nabla \cdot u+\nabla \cdot Z_{1}|_{2}|\theta^{3}|_{2}\nonumber\\
&\leq&\|Z_{2}\|_{3}|\nabla \cdot u+\nabla \cdot Z_{1}|_{2}|\theta^{2}|_{3}^{\frac{3}{2}}.\nonumber
\end{eqnarray*}
Applying interpolation inequality to $|\theta^{2}|_{3}$, we obtain
\begin{eqnarray*}
|\theta^{2}|_{3}\leq C |\theta^{2}|_{2}^{\frac{1}{2}}(|\nabla \theta^{2}|_{2}^{\frac{1}{2}}+ |\partial_{z} \theta^{2}|_{2}^{\frac{1}{2}}+\alpha|\theta^{2}(z=0) |_{2}^{\frac{1}{2}}   ),
\end{eqnarray*}
which together with H$\ddot{o}$lder inequality implies that
\begin{eqnarray}
\int_{\mho}[\varphi(u+Z_{1})\partial_{z} Z_{2}  ]\theta^{3}\!\!\!&\leq&\!\!\! \varepsilon (|\nabla \theta^{2}|_{2}^{2}+   |\partial_{z} \theta^{2}|_{2}^{2}+\alpha|\theta^{2}(z=0) |_{2}^{2} )+C\|Z_{2}\|_{3}^{\frac{8}{5}}\|u+Z_{1}\|_{1}^{\frac{8}{5}}|\theta|_{4}^{\frac{12}{5}}.
\end{eqnarray}
Taking a similar argument, we have
\begin{eqnarray}
\int_{\mho}[(u+Z_{1} )\cdot \nabla Z_{2} ]\theta^{3}\!\!\!&\leq&\!\!\! \varepsilon (|\nabla \theta^{2}|_{2}^{2}+   |\partial_{z} \theta^{2}|_{2}^{2}+\alpha|\theta^{2}(z=0) |_{2}^{2} )+C\|Z_{2}\|_{3}^{\frac{8}{5}}|u+Z_{1}|_{2}^{\frac{8}{5}}|\theta|_{4}^{\frac{12}{5}}.
\end{eqnarray}
Analogously,  we deduce
\begin{eqnarray}
\int_{\mho}(Q+Z_{2})\theta^{3}\leq \varepsilon (|\nabla \theta^{2}|_{2}^{2}+   |\partial_{z} \theta^{2}|_{2}^{2}+\alpha|\theta^{2}(z=0) |_{2}^{2} )+C(|Q|_{2}^{\frac{8}{5}}+ |Z_{2}|_{2}^{\frac{8}{5}}  )|\theta|_{4}^{\frac{12}{5}}.
\end{eqnarray}
Therefore, combining $(5.103)-(5.107),$ we arrive at
\begin{eqnarray}
&&\partial_{t}|\theta|^{4}_{4}+|\nabla \theta^{2}|^{2}_{2}+|(\theta^{2})_{z}|_{2}^{2}+\alpha\int_{M}|\theta(z=0)|^{4}\nonumber\\
&\leq & C(|Q|_{2}^{\frac{8}{5}}+ \|Z_{2}\|_{3}^{\frac{8}{5}}+\|Z_{1}\|_{1}^{\frac{8}{5}}+ \|Z_{2}\|_{3}^{\frac{8}{5}}\|u\|_{1}^{\frac{8}{5}})|\theta|_{4}^{\frac{12}{5}}.
\end{eqnarray}
Since by Young's inequality
\begin{eqnarray*}
|\theta|_{4}^{4}=\int_{\mho}\theta^{4}\!\!&=&\!\!-\int_{M}\int_{-1}^{0}\int_{z}^{0}\partial_{z}\theta^{4}+\int_{M}\int_{-1}^{0}\theta^{4}(z=0)\\
&\leq&8 |(\theta^{2})_{z}|_{2}^{2}+\frac{1}{2}|\theta|_{4}^{4}+\int_{M}\theta^{4}(z=0),
\end{eqnarray*}
we have
\begin{eqnarray*}
|\theta|_{4}^{4}\leq 16|\partial_{z} \theta |_{2}^{2}+2|\theta(z=0)|_{4}^{4},
\end{eqnarray*}
which combined $(5.108)$ implies
\begin{eqnarray*}
&&\partial_{t}|\theta|^{4}_{4}+|\theta|^{4}_{4}\leq C(|Q|_{2}^{\frac{8}{5}}+ \|Z_{2}\|_{3}^{\frac{8}{5}}+\|Z_{1}\|_{1}^{\frac{8}{5}}+ \|Z_{2}\|_{3}^{\frac{8}{5}}\|u\|_{1}^{\frac{8}{5}})|\theta|_{4}^{\frac{12}{5}}
\end{eqnarray*}
or
\begin{eqnarray*}
\partial_{t} |\theta|^{2}_{4}+|\theta|^{2}_{4}
\leq C(|Q|_{2}^{\frac{8}{5}}+ \|Z_{2}\|_{3}^{\frac{8}{5}}+\|Z_{1}\|_{1}^{\frac{8}{5}}+ \|Z_{2}\|_{3}^{\frac{8}{5}}\|u\|_{1}^{\frac{8}{5}})|\theta|_{4}^{\frac{2}{5}}.
\end{eqnarray*}
Then, using Gronwall inequality yields,
\begin{eqnarray}
|\theta(t)|^{2}_{4}\leq |\theta(t=0)|^{2}_{4}e^{-Ct}+C\int_{0}^{t}e^{-C(t-s)}(|Q|_{2}^{2}+ \|Z_{2}\|_{3}^{2}+\|Z_{1}\|_{1}^{2}+ \|Z_{2}\|_{3}^{2}\|u\|_{1}^{2})ds
\end{eqnarray}
for $t\in[0, \tau_{*}).$
Since, by integration by parts and boundary conditions $(5.99)$ we have
\begin{eqnarray*}
&&\int_{\mho}[(\tilde{u}\cdot \nabla)\tilde{u}-(\int_{-1}^{z}\nabla \cdot \tilde{u}d\lambda)
\partial_{z} \tilde{u} ]|\tilde{u}|^{2}\tilde{u}=0
\end{eqnarray*}
and
\begin{eqnarray*}
&&\int_{\mho}[(\bar{u}+\bar{Z_{1}} )\cdot\nabla \tilde{u}]|\tilde{u}|^{2}\tilde{u}=-\frac{1}{4}\int_{\mho}|\tilde{u}|^{4}\nabla
\cdot(\bar{u}+\bar{Z_{1}} )=0,
\end{eqnarray*}
as well as
\begin{eqnarray*}
&&\int_{\mho}{[((\tilde{u}+\tilde{Z_{1}} )\cdot \nabla )(\bar{u}+\bar{Z_{1}}) }]\cdot |\tilde{u}|^{2}\tilde{u}\\
&& =-\int_{\mho}{[((\tilde{u}+\tilde{Z_{1}} )\cdot \nabla )
|\tilde{u}|^{2}\tilde{u}}]\cdot(\bar{u}+\bar{Z_{1}})-\int_{\mho} \Big{(}\nabla \cdot (\tilde{u}+\tilde{Z_{1}} ) \Big{)} |\tilde{u}|^{2}\tilde{u} \cdot (\bar{u}+\bar{Z_{1}} )
\end{eqnarray*}
and
\begin{eqnarray*}
&&\int_{\mho}\overline{(\tilde{u}+\tilde{Z_{1}} )\nabla \cdot (\tilde{u}+\tilde{Z_{1}} )
+(\tilde{u}+\tilde{Z_{1}} )\cdot \nabla (\tilde{u}+\tilde{Z_{1}} )}\cdot| \tilde{u}|^{2}\tilde{u}\\
&&=-\int_{\mho}\overline{(\tilde{u}_{k}+\tilde{Z}_{1,k} ) (\tilde{u}_{j}+\tilde{Z}_{1,j} ) }\partial_{x_{k}}(| \tilde{u}|^{2}\tilde{u}_{j} ),
\end{eqnarray*}
where $\tilde{u}_{k}$ and $\tilde{Z}_{1,k} $ are the $k$'th coordinate of  $\tilde{u}$ and $\tilde{Z_{1}} $ respectively with $k=1,2.$
Taking the inner product of the equation $(5.98)$ with $|\tilde{u}|^{2}\tilde{u}$ in $(L^{2}(\mho))^{2},$ by the above equalities about $\tilde{u}$  we get
\begin{eqnarray}
&&\frac{1}{4}\partial_{t} |\tilde{u}|_{4}^{4}+
\frac{1}{2}\int_{\mho}\Big{(}|\nabla(|\tilde{u} |^{2}) |^{2}+ |\partial_{z}(|\tilde{u} |^{2}) |^{2}\Big{)}
+\int_{\mho}|\tilde{u}|^{2}(|\nabla \tilde{u} |^{2}+|\partial_{z}\tilde{u} |^{2})\nonumber\\
&&=-\int_{\mho}[\tilde{Z_{1}}\cdot \nabla \tilde{u} + \tilde{u} \cdot \nabla \tilde{Z_{1}}+ \tilde{Z_{1}}\cdot \nabla \tilde{Z_{1}} ]
\cdot|\tilde{u}|^{2}\tilde{u}\nonumber\\
&&-\int_{\mho}[\varphi(\tilde{Z_{1}} )\partial_{z}\tilde{u}+\varphi(\tilde{u})\partial_{z}\tilde{Z_{1}}+\varphi(\tilde{Z_{1}})\partial_{z}\tilde{Z_{1}} ]\cdot
|\tilde{u}|^{2}\tilde{u}\nonumber\\
&&+\int_{\mho}(\bar{u}+\bar{Z_{1}} )\cdot[(\tilde{u}+\tilde{Z_{1}} )\cdot\nabla ]|\tilde{u}|^{2}\tilde{u}
+\int_{\mho}[\nabla\cdot(\tilde{u}+\tilde{Z_{1}} )](\bar{u}+\bar{Z_{1}} )\cdot
|\tilde{u}|^{2}\tilde{u}\nonumber\\
&&-\int_{\mho}\{[(\bar{u}+\bar{Z_{1}} )\cdot \nabla ]\tilde{Z_{1}} \}|\tilde{u}|^{2}\tilde{u}
-\int_{\mho}\overline{(\tilde{u}_{k}+\tilde{Z}_{1,k} ) (\tilde{u}_{j}+\tilde{Z}_{1,j} ) }\partial_{x_{k}}(| \tilde{u}|^{2}\tilde{u}_{j} )\nonumber\\
&&-\int_{\mho}(fk\times \tilde{Z_{1}})\cdot|\tilde{u}|^{2}\tilde{u}\nonumber\\
&&-\int_{\mho}\Big{(}\int_{-1}^{z}\theta d\lambda-  \int_{-1}^{0}\int_{-1}^{z}\theta d\lambda dz  \Big{)}\nabla \cdot |\tilde{u}|^{2}\tilde{u}
:=\Sigma_{j=1}^{8}I_{j}.
\end{eqnarray}
Next, we estimate $I_{j}$ respectively,  for $j=1,\cdots, 8$. Since by integration by parts, we have
\begin{eqnarray}
&&\int_{\mho}(\tilde{Z_{1}} \cdot \nabla \tilde{u} )\cdot |\tilde{u}|^{2}\tilde{u}+\int_{\mho}(\varphi(\tilde {Z_{1}})\partial_{z} \tilde{u})|\tilde{u}|^{2}\tilde{u}\nonumber\\
&=&\int_{\mho}(\tilde{Z_{1}}\cdot \nabla \tilde{u_{i}})| \tilde{u}|^{2}\tilde{u_{i}}+\int_{\mho}(\varphi(\tilde {Z_{1}})\partial_{z}\tilde{u_{i}})
| \tilde{u}|^{2}\tilde{u_{i}}\nonumber\\
&=&\frac{1}{4}\int_{\mho}\tilde{Z_{1}} \cdot \nabla |\tilde{u} |^{4}+\frac{1}{4}\int_{\mho}\varphi(\tilde {Z_{1}})\partial_{z}|\tilde{u_{i}}|^{4}=0.
\end{eqnarray}
Therefore, to estimate $I_{1}$(or $I_{2}$), we need only to estimate the other two terms. By interpolation inequalities and H$\ddot{o}$lder inequality, we get
\begin{eqnarray*}
\int_{\mho} (\tilde{u}\cdot \nabla \tilde{Z_{1}})|\tilde{u}|^{2}\tilde{u}&\leq&
|\nabla \tilde{Z_{1}}|_{3}\Big{(}\int_{\mho} (|\tilde{u}|^{2} )^{3}\Big{)}^{\frac{2}{3}} \nonumber\\
&\leq&C |\nabla \tilde{Z_{1}}|_{3} |(|\tilde{u}|^{2})|_{2} \|(|\tilde{u}|^{2}) \|_{1}\nonumber\\
&\leq&C |\nabla \tilde{Z_{1}}|_{3} |(|\tilde{u}|^{2})|_{2}\Big{(}|\nabla( |\tilde{u}|^{2}) |_{2}
+ |\partial_{z}( |\tilde{u}|^{2}) |_{2} +|(|\tilde{u}|^{2})|_{2} \Big{)}\nonumber\\
&\leq& \varepsilon \Big{(} |\nabla( |\tilde{u}|^{2}) |_{2}^{2}
+ |\partial_{z}( |\tilde{u}|^{2}) |_{2}^{2}   \Big{)}+C\| \tilde{Z_{1}}\|^{2}_{2}|\tilde{u}|^{4}_{4}.
\end{eqnarray*}
Analogously, we have
\begin{eqnarray*}
\int_{\mho}(\tilde{Z_{1}}\cdot\nabla \tilde{Z_{1}} )\cdot |\tilde{u}|^{2} \tilde{u}&\leq& |(|\tilde{u}^{2}|) |_{3}^{\frac{3}{2}}
|\nabla \tilde{Z_{1}} |_{3}|\tilde{Z_{1}} |_{6}\nonumber\\
&\leq &C |(|\tilde{u}|^{2})|^{\frac{3}{4}}_{2}\|(|\tilde{u}|^{2}) \|_{1}^{\frac{3}{4}}|\nabla \tilde{Z_{1}} |_{3}|\tilde{Z_{1}} |_{6}\nonumber\\
&\leq& \varepsilon \Big{(} |\nabla( |\tilde{u}|^{2}) |_{2}^{2}
+ |\partial_{z}( |\tilde{u}|^{2}) |_{2}^{2}   \Big{)}+C\|\tilde{Z_{1}} \|_{2}^{\frac{16}{5}}|\tilde{u} |^{\frac{12}{5}}_{4}.
\end{eqnarray*}
Therefore, based on the above bounds we infer that
\begin{eqnarray*}
I_{1}&\leq& 2\varepsilon \Big{(} |\nabla( |\tilde{u}|^{2}) |_{2}^{2}
+ |\partial_{z}( |\tilde{u}|^{2}) |_{2}^{2}   \Big{)}+C\|Z_{1}\|_{2}^{2}|\tilde{u}|_{4}^{4}+C\|Z_{1} \|_{2}^{\frac{16}{5}}|\tilde{u} |^{\frac{12}{5}}_{4}.
\end{eqnarray*}
Applying integration by parts  formula on $M$, we reach
\begin{eqnarray*}
\int_{\mho}\varphi(\tilde{u})\partial_{z}\tilde{Z}_{1} \cdot |\tilde{u}|^{2}\tilde{u}&=&\int_{-1}^{0}[\int_{M}\varphi(\tilde{u})\partial_{z}\tilde{Z}_{1,i}   |\tilde{u}|^{2}\tilde{u}_{i}]dz\\
&=&\int_{-1}^{0}\Big{(}\int_{M}(\int_{-1}^{z}\tilde{u}dz' )[(\nabla \partial_{z}\tilde{Z}_{1,i})(|\tilde{u}|^{2}\tilde{u}_{i})+(\partial_{z}\tilde{Z}_{1,i}) \nabla (|\tilde{u}|^{2}\tilde{u}_{i}) ]\Big{)}dz
\end{eqnarray*}
which together with H$\ddot{o}$lder inequality implies
\begin{eqnarray*}
\int_{\mho}\varphi(\tilde{u})\partial_{z}\tilde{Z}_{1}
\cdot |\tilde{u}|^{2}\tilde{u}
&\leq& |\nabla \partial_{z} \tilde{Z}_{1,i}|_{2}
|( \int_{-1}^{z}\tilde{u}dz')|\tilde{u}_{i}||_{4}||\tilde{u}|^{2} |_{4}\nonumber\\
&&+|\nabla |\tilde{u}|^{2}|_{2} | (\int_{-1}^{z}\tilde{u}dz')|\tilde{u}|  |_{3}|\partial_{z}\tilde{Z}_{1,i} |_{6}\nonumber\\
&\leq&C\|Z_{1}\|_{2}|\tilde{u}|_{8}^{4}+C\|Z_{1}\|_{2}|\tilde{u}|_{6}^{2}|\nabla |\tilde{u}|^{2} |_{2}.
\end{eqnarray*}
Since, by interpolation inequality,
\begin{eqnarray*}
|\tilde{u}|_{8}^{4}=||\tilde{u}|^{2}|_{4}^{2}\leq C ||\tilde{u}|^{2}|_{2}^{\frac{1}{2}}( |\nabla|\tilde{u}|^{2}|_{2}^{\frac{3}{2}}  +
|\partial_{z}|\tilde{u}|^{2}|_{2}^{\frac{3}{2}}   )
\end{eqnarray*}
and
\begin{eqnarray*}
|\tilde{u}|_{6}^{2}=||\tilde{u}|^{2}|_{3}\leq C ||\tilde{u}|^{2}|_{2}^{\frac{1}{2}}( |\nabla|\tilde{u}|^{2}|_{2}^{\frac{1}{2}}  +
|\partial_{z}|\tilde{u}|^{2}|_{2}^{\frac{1}{2}}   ),
\end{eqnarray*}
taking into account of H$\ddot{o}$lder inequality we conclude
\begin{eqnarray}
\int_{\mho}\varphi(\tilde{u})\partial_{z}\tilde{Z_{1}}
\cdot |\tilde{u}|^{2}\tilde{u}
&\leq& \varepsilon ( |\nabla|\tilde{u}|^{2}|_{2}^{2}  +
|\partial_{z}|\tilde{u}|^{2}|_{2}^{2}   )+C\|Z_{1}\|_{2}^{4}|\tilde{u}|_{4}^{4}.
\end{eqnarray}
Using again integration by parts formula and H$\ddot{o}$lder inequality,
\begin{eqnarray}
\int_{\mho}\varphi(\tilde{Z}_{1})\cdot \partial_{z}\tilde{Z}_{1}\cdot |\tilde{u}|^{2}\tilde{u}
&\leq&|(|\tilde{u}|^{3})|_{2} |\partial_{z} \tilde{Z}_{1}|_{3} |\nabla\cdot\tilde{Z}_{1}|_{6}\leq C|(|\tilde{u}|^{2})|_{3}^{\frac{3}{2}}\|\tilde{Z}_{1}\|_{2}^{2}\nonumber\\
&\leq&C|(|\tilde{u}|^{2})|_{2}^{\frac{3}{4}}\Big{(}|\nabla( |\tilde{u}|^{2}) |_{2}^{\frac{3}{4}}
+ |\partial_{z}( |\tilde{u}|^{2}) |_{2}^{\frac{3}{4}} +|(|\tilde{u}|^{2})|_{2}^{\frac{3}{4}} \Big{)}\|\tilde{Z}_{1}\|_{2}^{2}\nonumber\\
& \leq&\varepsilon (|\nabla (|\tilde{u}|^{2} ) |^{2}_{2}+|\partial_{z}(|\tilde{u}|^{2})  |^{2}_{2} )+C\|\tilde{Z}_{1} \|_{2}^{\frac{16}{5}}|\tilde{u}|^{\frac{12}{5}}_{4}+C\|\tilde{Z}_{1}\|^{2}_{2}|\tilde{u} |^{3}_{4}.
\end{eqnarray}
Therefore, from $(5.111)-(5.113)$ we conclude
\begin{eqnarray*}
I_{2}&\leq& \varepsilon \Big{(} |\nabla (|\tilde{u}|^{2}) |^{2}_{2}+|\partial_{z}(|\tilde{u}|^{2}) |^{2}_{2}    \Big{)}
+C \|Z_{1}\|_{2}^{\frac{16}{5}}|\tilde{u}|_{4}^{\frac{12}{5}}+
C\|Z_{1}\|_{2}^{2} |\tilde{u} |^{3}_{4}+C\|Z_{1}\|_{2}^{4}|\tilde{u}|_{4}^{4}.
\end{eqnarray*}
To estimate $I_{3}$, we first consider
\begin{eqnarray}
&&\int_{\mho}(\bar{u}+\bar{Z}_{1} )\cdot \{[(\tilde{u}+\tilde{Z}_{1} )\cdot \nabla ]|\tilde{u}|^{2}\tilde{u}\}\nonumber\\
&=&\int_{\mho}\bar{u}\cdot\{[(\tilde{u}+\tilde{Z}_{1} )\cdot \nabla ]|\tilde{u}|^{2}\tilde{u}\}+\int_{\mho}\bar{Z}_{1}\cdot
\{[(\tilde{u}+\tilde{Z}_{1} )\cdot \nabla ]|\tilde{u}|^{2}\tilde{u}\}.
\end{eqnarray}
Applying H$\ddot{o}$lder inequality repeatedly on the first term on the right hand side of $(5.114)$ ,
\begin{eqnarray*}
&&\int_{\mho}\bar{u}\cdot\{[(\tilde{u}+\tilde{Z}_{1} )\cdot \nabla ]|\tilde{u}|^{2}\tilde{u}\}\\
&\leq&C\int_{M}|\bar{u}|\int_{-1}^{0}|\tilde{u}||\nabla \tilde{u}| |\tilde{u}|^{2}dz+|\tilde{Z}|_{L^{\infty}_{x}}
\int_{M}|\bar{u}|\int_{-1}^{0}|\nabla \tilde{u}||\tilde{u}|^{2}dz\\
&\leq&\int_{M}\Big{[}|\bar{u}|\Big{(}\int_{-1}^{0}|\tilde{u}|^{2}|\nabla \tilde{u}|^{2}dz \Big{)}^{\frac{1}{2}}
\Big{(}\int_{-1}^{0}|\tilde{u}|^{4}dz\Big{)}^{\frac{1}{2}}\Big{]}\\
&&+|\tilde{Z}_{1}|_{\infty}\int_{M}\Big{[}|\bar{u}|\Big{(}\int_{-1}^{0}|\tilde{u}|^{2}|\nabla \tilde{u}|^{2}dz \Big{)}^{\frac{1}{2}}
\Big{(}\int_{-1}^{0}|\tilde{u}|^{2}dz\Big{)}^{\frac{1}{2}}\Big{]}\\
&\leq&|\bar{u}|_{L^{4}(M)}|\nabla(|\tilde{u}|^{2})|_{2}\Big{(}\int_{M}(\int_{-1}^{0}|\tilde{u}|^{4} dz)^{2}  \Big{)}^{\frac{1}{4}}\\
&&+|\tilde{Z}_{1}|_{\infty}|\bar{u}|_{L^{4}(M)}
|\nabla(|\tilde{u}|^{2})|_{2}\Big{(}\int_{M}(\int_{-1}^{0}|\tilde{u}|^{2} dz)^{2}  \Big{)}^{\frac{1}{4}}.
\end{eqnarray*}
Then by Minkowski inequality and interpolation inequalities, proceeding from above we have
\begin{eqnarray*}
\int_{\mho}\bar{u}\cdot\{[(\tilde{u}+\tilde{Z} )\cdot \nabla ]|\tilde{u}|^{2}\tilde{u}\}&\leq&|\bar{u}|_{L^{4}(M)}|\nabla(|\tilde{u}|^{2})|_{2}
\Big{(}\int_{-1}^{0}(\int_{M}|\tilde{u}|^{8} )^{\frac{1}{2}}\Big{)}^{\frac{1}{2}}\\
&&+|\tilde{Z}_{1}|_{\infty}|\bar{u}|_{L^{4}(M)}|\nabla(|\tilde{u}|^{2})|_{2}| \tilde{u}|_{4}\\
&\leq&|\bar{u}|_{L^{4}(M)}|\nabla(|\tilde{u}|^{2})|_{2}\Big{(}\int_{-1}^{0}|(|\tilde{u}|^{2}) |_{L^{2}(M)} \|(|\tilde{u}|^{2}) \|_{H^{1}(M)}
 dz\Big{)}^{\frac{1}{2}}\\
&& +|\tilde{Z}_{1}|_{\infty}|\bar{u}|_{L^{4}(M)}|\nabla(|\tilde{u}|^{2})|_{2}| \tilde{u}|_{4}\\
&\leq&|\bar{u}|_{L^{4}(M)}|\nabla(|\tilde{u}|^{2})|_{2}|(|\tilde{u}|^{2})|_{2}^{\frac{1}{2}}
[|(|\tilde{u}|^{2})|_{2}^{\frac{1}{2}}+|\nabla (|\tilde{u}|^{2}) |^{\frac{1}{2}}_{2}+|\partial_{z} (|\tilde{u}|^{2})|_{2}^{\frac{1}{2}}]\nonumber\\
&& +|\tilde{Z}_{1}|_{\infty}|\bar{u}|_{L^{4}(M)}|\nabla(|\tilde{u}|^{2})|_{2}| \tilde{u}|_{4},
\end{eqnarray*}
which together with  H$\ddot{o}$lder inequality implies
\begin{eqnarray}
\int_{\mho}\bar{u}\cdot\{[(\tilde{u}+\tilde{Z}_{1} )\cdot \nabla ]|\tilde{u}|^{2}\tilde{u}\}&\leq& \varepsilon (|\nabla (|\tilde{u}|^{2} )|^{2}_{2}+ |\partial_{z} (|\tilde{u}|^{2} )|^{2}_{2})\nonumber\\
&&+C(|\bar{u}|_{L^{4}(M)}^{2}+|\bar{u}|_{L^{4}(M)}^{4} )|\tilde{u}|^{4}_{4}+C\|\tilde{Z}_{1}\|_{2}^{2}|\bar{u}|_{L^{4}(M)}^{2}|\tilde{u}|^{2}_{4} .
\end{eqnarray}
By H$\ddot{o}$lder inequality,
\begin{eqnarray}
\int_{\mho}\bar{Z}_{1}\cdot
\{[(\tilde{u}+\tilde{Z}_{1} )\cdot \nabla ]|\tilde{u}|^{2}\tilde{u}\}&\leq& \varepsilon \int_{\mho}|\nabla \tilde{u} |^{2}
|\tilde{u}|^{2}+C|Z_{1}|^{4}_{L^{\infty}_{x}}|\tilde{u}|^{2}_{2}+C|Z|^{2}_{\infty}|\tilde{u}|^{4}_{4}\nonumber\\
&\leq& \varepsilon \int_{\mho}|\nabla \tilde{u} |^{2}
|\tilde{u}|^{2}+C\|Z_{1}\|^{4}_{2}|\tilde{u}|^{2}_{2}+C|Z_{1}|^{2}_{2}|\tilde{u}|^{4}_{4}.
\end{eqnarray}
Therefore, from $(5.114)-(5.116)$ we have
\begin{eqnarray*}
I_{3}&\leq& \varepsilon (|\nabla (|\tilde{u}|^{2} )|^{2}_{2}+ |\partial_{z} (|\tilde{u}|^{2} )|^{2}_{2}+
\int_{\mho}|\nabla \tilde{u} |^{2}
|\tilde{u}|^{2}  )\\
&&+C(|\bar{u}|_{L^{4}(M)}^{2}+ |\bar{u}|_{L^{4}(M)}^{4} + |Z_{1}|_{2}^{2}
 )|\tilde{u}|^{4}_{4} +C\|Z_{1}\|_{2}^{2}|\bar{u}|_{L^{4}(M)}^{2}|\tilde{u}|_{4}^{2}+C\|Z_{1}\|_{2}^{4}|\tilde{u}|_{2}^{2}.
\end{eqnarray*}
Similar to the estimate for $I_{2},$ we obtain
\begin{eqnarray*}
I_{4}&=&\int_{\mho}[\nabla \cdot (\tilde{u}+\tilde{Z}_{1} )](\bar{u}+\bar{Z}_{1} )\cdot|\tilde{u}|^{2}\tilde{u}\nonumber\\
&\leq&\varepsilon \Big{(}\int_{\mho}[|\nabla \tilde{u} |^{2}|\tilde{u}|^{2}+|\nabla( |\tilde{u}|^{2}) |^{2}_{2}+ |\partial_{z}( |\tilde{u}|^{2}) |^{2}_{2}]\Big{)} \nonumber\\
&&+C(|\bar{u}|_{L^{4}(M)}^{2}+|\bar{u}|_{L^{4}(M)}^{4}+\|Z_{1}\|^{2}_{2}   )|\tilde{u}|^{4}_{4} \nonumber\\
&&+C(\|Z_{1}\|_{2}^{\frac{16}{5}}+\|Z_{1}\|_{2}^{\frac{8}{5}}|\bar{u}|_{L^{4}(M)}^{\frac{8}{5}}  )|\tilde{u}|_{4}^{\frac{12}{5}}
+C(\|Z_{1}\|_{2}^{2}+\|Z_{1}\|_{2}|\bar{u}|_{L^{4}(M)}  )|\tilde{u}|_{4}^{3}.
\end{eqnarray*}
By H$\ddot{o}$lder inequality and interpolation inequality,
\begin{eqnarray*}
I_{5}&=&\int_{\mho}\{[(\bar{u}+\bar{Z}_{1} )\cdot \nabla ]\tilde{Z}_{1} \}|\tilde{u}|^{2}\tilde{u}\nonumber\\
&\leq & |\tilde{u} |^{3}_{6} |\nabla \tilde{Z}_{1}|_{4}(|\bar{u} |_{L^{4}(M)}+ |\bar{Z}_{1} |_{L^{4}(M)} )\nonumber\\
&\leq & |(|\tilde{u}|^{2} )  |_{2}^{\frac{3}{4}}\| (|\tilde{u}|^{2} ) \|_{1}^{\frac{3}{4}}\| \tilde{Z}_{1}\|_{2}
(|\bar{u} |_{L^{4}(M)}+ |\bar{Z}_{1} |_{L^{4}(M)} )\nonumber\\
&\leq & \varepsilon ( |\nabla (|\tilde{u}|^{2} )|^{2}_{2}+ |\partial_{z} (|\tilde{u}|^{2}) |^{2}_{2} )\nonumber\\
&&+C\|{Z}_{1}\|_{2}(|\bar{u}|_{4}+|{Z}_{1}|_{4} )  |\tilde{u}|^{3}_{4}+C\|{Z}_{1} \|^{\frac{8}{5}}_{2}
(|\bar{u} |^{\frac{8}{5}}_{L^{4}(M)}+ |{Z}_{1}  |^{\frac{8}{5}}_{L^{4}(M)}  )|\tilde{u}|_{4}^{\frac{12}{5}}
.
\end{eqnarray*}
Using H$\ddot{o}$lder inequality,  we have
\begin{eqnarray*}
I_{6}&=&-\int_{\mho}\overline{(\tilde{u}_{k}+\tilde{Z}_{1,k} )
(\tilde{u}_{j}+\tilde{Z}_{1,j} ) }\partial_{x_{k}}(| \tilde{u}|^{2}\tilde{u}_{j} )\nonumber\\
&\leq&\int_{M}\Big{(}\int_{-1}^{0} |\tilde{u}|^{2}dz  \int_{-1}^{0} | \nabla\tilde{u}| |\tilde{u}|^{2}dz  \Big{)}\nonumber\\
&&+|Z_{1}|_{\infty}\int_{M}(\int_{-1}^{0}|\tilde{u}| )\int_{-1}^{0}|\nabla \tilde{u}||\tilde{u}|^{2}dz+ |Z_{1}|_{\infty}^{2}||\nabla \tilde{u}| |\tilde{u}|  |_{2}  |\tilde{u}|_{2}
\nonumber\\
&\leq& \Big{(}\int_{\mho}| \nabla\tilde{u}|^{2} |\tilde{u}|^{2} \Big{)}^{\frac{1}{2}}
\Big{(} \int_{M} \Big{(}\int_{-1}^{0}|\tilde{u}|^{2}\Big{)}^{3}\Big{)}^{\frac{1}{2}} \nonumber\\
&&+|Z_{1}|_{\infty}\Big{(}\int_{\mho}|\nabla \tilde{u} |^{2}|\tilde{u}|^{2}\Big{)}^{\frac{1}{2}}
 \Big{(} \int_{\mho}(\int_{-1}^{0}|\tilde{u}|^{2} )^{2}  \Big{)}^{\frac{1}{2}}
 + C\|Z_{1}\|_{2}^{2}||\nabla \tilde{u}| |\tilde{u}|  |_{2}  |\tilde{u}|_{2},
 \end{eqnarray*}
which together with Minkowski inequality and interpolation inequality imply
\begin{eqnarray*}
I_{6}&\leq& \Big{(}\int_{\mho}| \nabla\tilde{u}|^{2} |\tilde{u}|^{2} \Big{)}^{\frac{1}{2}}
\Big{(}\int_{-1}^{0}\Big{(}\int_{M}|\tilde{u}|^{6}\Big{)}^{\frac{1}{3}} \Big{)}^{\frac{3}{2}} \nonumber\\
&&+|Z_{1}|_{\infty}\Big{(}\int_{\mho}| \nabla\tilde{u}|^{2} |\tilde{u}|^{2} \Big{)}^{\frac{1}{2}} |\tilde{u}|_{4}^{2}+ C\|Z_{1}\|_{2}^{2}||\nabla \tilde{u}| |\tilde{u}|  |_{2}  |\tilde{u}|_{2}\nonumber\\
&\leq& C\Big{(}\int_{\mho}| \nabla\tilde{u}|^{2} |\tilde{u}|^{2} \Big{)}^{\frac{1}{2}}
\Big{(}\int_{-1}^{0}|\tilde{u}|^{\frac{4}{3}}_{L^{4}(M)}\cdot\|\tilde{u}\|^{\frac{2}{3}}_{H^{1}(M)} dz \Big{)}^{\frac{3}{2}}\nonumber\\
&&+\|Z_{1}\|_{2}\Big{(}\int_{\mho}| \nabla\tilde{u}|^{2} |\tilde{u}|^{2} \Big{)}^{\frac{1}{2}} |\tilde{u}|_{4}^{2}
+ C\|Z_{1}\|_{2}^{2}||\nabla \tilde{u}| |\tilde{u}|  |_{2}  |\tilde{u}|_{2}  \nonumber\\
&\leq& C\Big{(}\int_{\mho}| \nabla\tilde{u}|^{2} |\tilde{u}|^{2} \Big{)}^{\frac{1}{2}}
\Big{(}\int_{0}^{1}|\tilde{u}|^{4}_{L^{4}(M)}dz\Big{)}^{\frac{1}{2}}\int_{0}^{1}\|\tilde{u}\|_{H^{1}(M)}dz\nonumber\\
&&+\|Z_{1}\|_{2}\Big{(}\int_{\mho}| \nabla\tilde{u}|^{2} |\tilde{u}|^{2} \Big{)}^{\frac{1}{2}} |\tilde{u}|_{4}^{2}
+ C\|Z_{1}\|_{2}^{2}||\nabla \tilde{u}| |\tilde{u}|  |_{2}  |\tilde{u}|_{2}.
\end{eqnarray*}
Therefore, using H$\ddot{o}$lder inequality again we arrive at
\begin{eqnarray*}
I_{6}&\leq&\varepsilon \int_{\mho}| \nabla\tilde{u}|^{2} |\tilde{u}|^{2}+C\|Z_{1}\|_{2}^{4}|\tilde{u}|_{2}^{2}
+C(\|Z_{1}\|_{2}^{2}+\| \tilde{u}\|^{2}_{1})|\tilde{u}|_{4}^{4}.
\end{eqnarray*}
By H$\ddot{o}$lder inequality and Sobolev imbedding theorem,
\begin{eqnarray*}
I_{7}&=&-\int_{\mho}(fk\times \tilde{Z}_{1})\cdot | \tilde{u}|^{2}\tilde{u}\leq C\|Z_{1}\|_{1}| \tilde{u}|_{4}^{3}.
\end{eqnarray*}
Analogously, we have
\begin{eqnarray*}
I_{8}&=&-\int_{\mho}\Big{(}\int_{-1}^{z}\theta d\lambda-\int_{-1}^{0}\int_{-1}^{z}\theta d\lambda dz \Big{)}\nabla \cdot |\tilde{u}|^{2}\tilde{u}\nonumber\\
&\leq& \Big{(} \int_{\mho}|\nabla\tilde{u}|^{2}|\tilde{u}|^{2}\Big{)}^{\frac{1}{2}}
\Big{(}\int_{\mho}|\tilde{u}|^{4}\Big{)}^{\frac{1}{4}}\Big{(}\int_{\mho}|\theta|^{4}\Big{)}^{\frac{1}{4}}\nonumber\\
&\leq& \varepsilon \int_{\mho}|\nabla\tilde{u}|^{2}|\tilde{u}|^{2}+C| \tilde{u}|_{4}^{2}|\theta|_{4}^{2}.
\end{eqnarray*}
Therefore, by $(5.110)$ and the estimates of $I_{1}-I_{8}$, we have
\begin{eqnarray}
&&\partial_{t} |\tilde{u}|^{4}_{4}+
\int_{\mho}\Big{(}|\nabla(|\tilde{u} |^{2}) |^{2}+ |\partial_{z}(|\tilde{u} |^{2}) |^{2}\Big{)}
+\int_{\mho}|\tilde{u}|^{2}(|\nabla \tilde{u} |^{2}+|\partial_{z}\tilde{u} |^{2})\nonumber\\
&\leq&C(\|Z_{1}\|_{2}^{2}+\|Z_{1}\|_{2}^{4}+\|u\|_{1}^{2}+|u|_{2}^{2}\|u\|_{1}^{2})|\tilde{u}|_{4}^{4}\nonumber\\
&&+C(\|Z_{1}\|_{1}+\|Z_{1}\|_{2}\|u\|_{1}+\|Z_{1}\|_{2}^{2} )|\tilde{u}|_{4}^{3}\nonumber\\
&&+ C(\|Z_{1}\|_{2}^{\frac{8}{5}}\|u\|_{1}^{\frac{8}{5}}+\|Z_{1}\|_{2}^{\frac{16}{5}}+\|Z_{1}\|_{2}^{\frac{8}{5}}+\|u\|_{1}^{\frac{8}{5}}   )|\tilde{u}|_{4}^{\frac{12}{5}}\nonumber\\
&&+C(\|Z_{1}\|_{2}^{4}+\|Z_{1}\|_{2}^{2}\|u\|_{1}^{2}+|\theta|_{4}^{2})|\tilde{u}|_{4}^{2}
\end{eqnarray}
and
\begin{eqnarray}
\partial_{t}|\tilde{u}|^{2}_{4}&\leq&C(\|Z_{1}\|_{2}^{2}+\|Z_{1}\|_{2}^{4}+\|u\|_{1}^{2}+|u|_{2}^{2}\|u\|_{1}^{2}+\|Z_{1}\|_{2}^{2} \|u\|_{1}^{2} )|\tilde{u}|_{4}^{2}\nonumber\\
&&+C(\|Z_{1}\|_{2}^{4}+\|Z_{1}\|_{2}^{2}\|u\|_{1}^{2}+|\theta|_{4}^{2}).
\end{eqnarray}
Subsequently, by Gronwall inequlity, $(5.110)$ and $(5.117)-(5.118)$, we conclude that
\begin{eqnarray}
&&\sup\limits_{t\in [0, \tau*)}|\tilde{u}(t)|_{4}^{4}+
\int_{0}^{\tau*}\int_{\mho}\Big{(}|\nabla(|\tilde{u} |^{2}) |^{2}+ |\partial_{z}(|\tilde{u} |^{2}) |^{2}\Big{)}ds
+\int_{0}^{\tau*}\int_{\mho}|\tilde{u}|^{2}(|\nabla \tilde{u} |^{2}+|\partial_{z}\tilde{u} |^{2})ds\nonumber\\
&\leq& C(\tau_{*}, Q, Z_{1}, Z_{2}, v_{0}, T_{0} ).
\end{eqnarray}
\par
\noindent{$5.2.3.\ \ H^{1}\ \mathbf{estimates}\ \mathbf{about}\    \theta\ \mathbf{and}\ u$. }
By integration by parts and $(5.96)-(5.97)$(for more detail, see $\cite{CT1} $ ), we have
\begin{eqnarray*}
\int_{M}fk\times \bar{u}\cdot \Delta \bar{u}=0,\ \ \ \
\int_{M}\nabla p_{s}\cdot\Delta \bar{u}=0
\end{eqnarray*}
and
\begin{eqnarray*}
\int_{M}\nabla \int_{-1}^{0}\int_{-1}^{z}\theta(x,y,\lambda, t)d\lambda dz\cdot\Delta \bar{u}=0.
\end{eqnarray*}
Then, taking the inner product of equation $(5.92)$ with $-\Delta \bar{u}$ in $L^{2}(M),$ we arrive at
\begin{eqnarray}
&&\frac{1}{2}\partial_{t} |\nabla \bar{u} |_{2}^{2}+|\Delta \bar{u}|_{2}^{2}=
\int_{M}(\bar{u}+\bar{Z}_{1})\cdot \nabla(\bar{u}+\bar{Z}_{1})\cdot \Delta\bar{u}
+\int_{M}\overline{(\tilde{u}+\tilde{Z}_{1})\cdot \nabla(\tilde{u}+\tilde{Z}_{1})}\cdot \Delta\bar{u}\nonumber\\
&&+\int_{M}\overline{(\tilde{u}+\tilde{Z}_{1}) \nabla\cdot(\tilde{u}+\tilde{Z}_{1})}\cdot \Delta\bar{u}+\int_{M}fk\times\bar{Z}_{1} \cdot \Delta\bar{u}+ \int_{\mho}Z_{1}\cdot \Delta\bar{u}.
\end{eqnarray}
By H$\ddot{o}$lder inequality and interpolation inequalities, we have
\begin{eqnarray}
&&\int_{M}(\bar{u}+\bar{Z}_{1})\cdot \nabla(\bar{u}+\bar{Z}_{1})\cdot \Delta\bar{u}\nonumber\\
&=&
\int_{M}\bar{u}\cdot \nabla(\bar{u}+\bar{Z}_{1})\cdot \Delta\bar{u}
+\int_{M}\bar{Z}_{1}\cdot \nabla(\bar{u}+\bar{Z}_{1})\cdot \Delta\bar{u}\nonumber\\
&\leq&C|\bar{u}|_{2}^{\frac{1}{2}}(|\nabla \bar{u}|_{2}+ |\nabla \bar{Z}_{1}|_{2} )|\Delta \bar{u}|_{2}^{\frac{3}{2}}
+C|Z_{1}|_{\infty}(|\nabla \bar{u}|_{2}+ |\nabla \bar{Z}_{1}|_{2} )|\Delta \bar{u}|_{2}\nonumber\\
&\leq & \varepsilon |\Delta\bar{u}|_{2}^{2}+C|\bar{u}|_{2}^{2}(|\nabla \bar{u}|_{2}^{4}+|\nabla \bar{Z}_{1}|_{2}^{4}   )
+C|Z_{1}|_{\infty}^{2}(|\nabla \bar{u}|_{2}^{2}+|\nabla \bar{Z}_{1}|_{2}^{2}   ).
\end{eqnarray}
Using  H$\ddot{o}$lder inequality , Minkowski inequality and Sobolev imbedding theorem, we obtain
\begin{eqnarray}
&&\int_{M}\overline{(\tilde{u}+\tilde{Z}_{1}) \nabla\cdot(\tilde{u}+\tilde{Z}_{1} )}\cdot \Delta\bar{u}
+\int_{M}\overline{(\tilde{u}+\tilde{Z}_{1}) \cdot \nabla(\tilde{u}+\tilde{Z}_{1} )}\cdot \Delta\bar{u}
\nonumber\\
&\leq&2\int_{M}\Big{(}\int_{-1}^{0}(|\tilde{u}|+|\tilde{Z}_{1}|  )( |\nabla\tilde{u}|+|\nabla\tilde{Z}_{1}| )    \Big{)}|\Delta \bar{u}|\nonumber\\
&\leq& 2|\Delta \bar{u}|_{2}\Big{[}\Big{(}\int_{\mho}|\tilde{u} |^{2}  |\nabla\tilde{u} |^{2}   \Big{)}^{\frac{1}{2}}
+\Big{(}\int_{\mho}|\tilde{u} |^{2}  |\nabla\tilde{Z}_{1} |^{2}   \Big{)}^{\frac{1}{2}}\nonumber\\
&&+\Big{(}\int_{\mho}|\tilde{Z}_{1} |^{2}  |\nabla\tilde{u} |^{2}   \Big{)}^{\frac{1}{2}}+
\Big{(}\int_{\mho}|\tilde{Z}_{1} |^{2}  |\nabla\tilde{Z}_{1} |^{2}   \Big{)}^{\frac{1}{2}}
\Big{]}\nonumber\\
&\leq& \varepsilon |\Delta \bar{u}|_{2}^{2}+C\Big{(}\int_{\mho}|\tilde{u} |^{2}  |\nabla\tilde{u} |^{2}
+|\tilde{u}|_{4}^{2}\|Z_{1}\|_{2}^{2}
+|Z_{1}|_{\infty }\|u \|_{1}^{2}+|Z_{1}|_{\infty} \|Z_{1} \|_{1}^{2}\Big{)}.
\end{eqnarray}
From $(5.120)-(5.122),$ we conclude that
\begin{eqnarray}
\partial_{t} |\nabla \bar{u} |_{2}^{2} +|\Delta \bar{u}|_{2}^{2}&\leq& C(|u|_{2}^{2}\|u\|_{1}^{2}+\|Z_{1}\|_{2}^{2})
|\nabla \bar{u}|_{2}^{2}+C(|Z_{1}|_{2}^{2}+\|Z_{1}\|_{2}^{3}+\|Z_{1}\|_{2}^{4}\nonumber\\
&&+\|Z_{1}\|_{1}^{4}|u|_{2}^{2} +\|Z_{1}\|_{2}^{2} |\tilde{u}|_{4}^{2}+\|Z_{1}\|_{2}\|u\|_{1}^{2} +\int_{\mho}|\tilde{u} |^{2}  |\nabla\tilde{u} |^{2}).
\end{eqnarray}
Therefore, by Proposition $2.1$, Proposition $2.2$, $(5.100), (5.119), (5.123)$ and Gronwall inequality, we arrive at
\begin{eqnarray}
\sup\limits_{t\in[0, \tau*)}|\nabla \bar{u} (t)|_{2}^{2}\leq
C(\tau_{*}, Q, Z_{1},Z_{2}, v_{0}, T_{0} ).
\end{eqnarray}
Taking the derivative, with respect to $z$, of equation $(5.86),$ we get
\begin{eqnarray}
&&\partial_{t} u_{z}-\Delta u_{z}-\partial_{zz} u_{z}+
[(u+Z_{1})\cdot\nabla ](u_{z}+\partial_{z}Z_{1})+[(u_{z}+\partial_{z}Z_{1})\cdot\nabla ](u+Z_{1})\nonumber\\
&&-(u_{z}+\partial_{z} Z_{1})\nabla \cdot(u+Z_{1})+\varphi(u+Z_{1})(u_{zz}+\partial_{zz}Z_{1})\nonumber\\
 &&+fk\times( u_{z}+\partial_{z}Z_{1})-\nabla \theta-\beta \partial_{z}Z_{1}=0.
\end{eqnarray}
By integration by parts, we obtain
\begin{eqnarray}
\int_{\mho}\Big{(}(u\cdot \nabla )u_{z}+\varphi(u)  u_{zz} \Big{)}\cdot u_{z}=0 .
\end{eqnarray}
Similarly,  using H$\ddot{o}$lder inequality, interpolation inequality and Sobolev embedding theorem, we have
\begin{eqnarray}
\int_{\mho}[(u\cdot\nabla)\partial_{z} Z_{1}]\cdot u_{z}&\leq& \|Z_{1}\|_{2}|u|_{4}|u_{z}|_{4}\nonumber\\
&\leq& \|Z_{1}\|_{2}|u|_{4}|u_{z}|_{4}^{\frac{1}{4}}
[|\nabla u_{z}|_{2}^{\frac{3}{4}}+|u_{zz}|_{2}^{\frac{3}{4}}+|u_{z}|_{2}^{\frac{3}{4}} ]\nonumber\\
&\leq& \varepsilon (|\nabla u_{z}|_{2}^{2}+|u_{zz}|_{2}^{2} )+C\|Z_{1}\|_{2}^{2}(|\tilde{u}|_{4}^{2}+|\nabla \bar{u}|_{2}^{2})+C|u_{z}|_{2}^{2}.
\end{eqnarray}
By virtue of H$\ddot{o}$lder inequality, we have
\begin{eqnarray}
\int_{\mho}(Z_{1}\cdot \nabla )(u_{z}+\partial_{z} Z_{1})\cdot u_{z}
&\leq& |Z_{1}|_{\infty}(|\nabla u_{z}|_{2}+ |\nabla \partial_{z} Z_{1}|_{2})| u_{z}|_{2}  \nonumber\\
&\leq& \varepsilon |\nabla u_{z} |_{2}^{2}+ C\|Z_{1}\|_{2}^{2}(|u_{z}|_{2}^{2}+1) .
\end{eqnarray}
Thanks to  H$\ddot{o}$lder inequality, interpolation inequality and Sobolev imbedding theorem, we reach
\begin{eqnarray*}
&& \int_{\mho}[(u_{z} \cdot \nabla)u]\cdot u_{z}
+\int_{\mho}[(u_{z} \cdot \nabla)Z_{1}]\cdot u_{z}\nonumber\\
&\leq & C\int_{\mho} |u||u_{z}||\nabla u_{z}|
+C\int_{\mho} |Z_{1}||u_{z}||\nabla u_{z}|
\nonumber\\
&\leq& C|\nabla u_{z} |_{2}|u|_{4}|u_{z}|_{4}
+C|Z_{1}|_{\infty}|u_{z}|_{2}|\nabla u_{z} |_{2}
\nonumber\\
&\leq& C|\nabla u_{z} |_{2}|u|_{4}|u_{z}|_{2}^{\frac{1}{4}}(|\nabla u_{z} |_{2}^{\frac{3}{4}}+ |\partial_{z} u_{z} |_{2}^{\frac{3}{4}}+
|u_{z}|_{2}^{\frac{3}{4}})
+C|Z_{1}|_{\infty}|u_{z}|_{2}|\nabla u_{z} |_{2}\nonumber\\
&\leq&\varepsilon (|\nabla u_{z} |_{2}^{2}+ |\partial_{z} u_{z} |_{2}^{2})
+C(|u|_{4}^{8}+\|Z_{1} \|_{2}^{8}+1) |u_{z}|_{2}^{2}.
\end{eqnarray*}
By integration by parts, H$\ddot{o}$lder inequality and interpolation inequality, we have
\begin{eqnarray*}
&&\int_{\mho}[(\partial_{z} Z_{1} \cdot \nabla)(u+Z_{1})]\cdot u_{z}\nonumber\\
&=&\int_{\mho}[\partial_{z}Z_{1}\cdot\nabla (u_{j}+Z_{1,j})  ]\partial_{z}u_{j}
=-\int_{\mho}\nabla\cdot[(\partial_{z}u_{j})( \partial_{z}Z_{1})](u_{j}+Z_{1,j})\nonumber\\
&\leq& |\nabla \partial_{z}Z_{1}|_{2}|u_{z}|_{4}|u+Z_{1}|_{4}+|\nabla \partial_{z}u|_{2}|\partial_{z}Z_{1}|_{4}|u+Z_{1}|_{4}\nonumber\\
&\leq&C\|Z_{1}\|_{2}|u_{z}|_{2}^{\frac{1}{4}}[ |\nabla \partial_{z}u|_{2}^{\frac{3}{4}}+|u_{zz}|_{2}^{\frac{3}{4}}+|u_{z}|_{2}^{\frac{3}{4}}  ]|u+Z_{1}|_{4}\nonumber\\
&&+ C|\nabla \partial_{z}u|_{2}\|Z_{1}\|_{2}|u+Z_{1}|_{4}\nonumber\\
&\leq&\varepsilon(|\nabla \partial_{z}u|_{2}^{2}+|u_{zz}|_{2}^{2})+C\|Z_{1}\|_{2}^{2}|u+Z_{1}|_{4}^{2}+|u_{z}|_{2}^{2},
\end{eqnarray*}
where $j=1,2, Z_{1}=(Z_{1,1},Z_{1,2})$ and $u=(u_{1},u_{2}) .$
Since
\begin{eqnarray*}
&&\int_{\mho}[(u_{z}+\partial_{z} Z_{1})\cdot \nabla (u+Z_{1})]\cdot u_{z}\\
&=&  \int_{\mho}[(u_{z} \cdot \nabla)u]\cdot u_{z}
+\int_{\mho}[(u_{z} \cdot \nabla)Z_{1}]\cdot u_{z}+\int_{\mho}[(\partial_{z} Z_{1} \cdot \nabla)(u+Z_{1})]\cdot u_{z},
\end{eqnarray*}
we conclude from the above two bounds that
\begin{eqnarray}
&&\int_{\mho}[(u_{z}+\partial_{z} Z_{1})\cdot \nabla (u+Z_{1})]\cdot u_{z}\nonumber\\
&\leq&\varepsilon(|\nabla \partial_{z}u|_{2}^{2}+|u_{zz}|_{2}^{2})+ C(|u|_{4}^{8}+\|Z_{1} \|_{2}^{8}+1) |u_{z}|_{2}^{2}+C\|Z_{1}\|_{2}^{2}|u+Z_{1}|_{4}^{2}\nonumber\\
&\leq& \varepsilon(|\nabla \partial_{z}u|_{2}^{2}+|u_{zz}|_{2}^{2})
+C(|\tilde{u}|_{4}^{8}+|\nabla \bar{u} |_{2}^{8}+\|Z_{1} \|_{2}^{8}+1) |u_{z}|_{2}^{2}\nonumber\\
 &&+C\|Z_{1}\|_{2}^{2}(|\tilde{u}+Z_{1}|_{4}^{2}+|\nabla\bar{u}|_{2}^{2}).
\end{eqnarray}
Proceeding as in $(5.129)$, we get
\begin{eqnarray}
&&\int_{\mho}\nabla \cdot(u+Z_{1})(u_{z}+\partial_{z} Z_{1})\cdot u_{z}\nonumber\\
&\leq &\varepsilon (|\nabla u_{z} |_{2}^{2}+ |u_{zz} |_{2}^{2})
+C(|\bar{u}|_{4}^{8}+|\nabla\bar{u}|_{2}^{2}+\|Z_{1} \|_{2}^{8}+1) |u_{z}|_{2}^{2}.
\end{eqnarray}
Using integration by parts,
\begin{eqnarray*}
\int_{\mho}\varphi (u) (\partial_{zz}Z_{1} )u_{z}=\int_{\mho}(\nabla\cdot u) (\partial_{z}Z_{1})u_{z}-\int_{\mho}\varphi (u) ( \partial_{z}Z_{1})u_{zz}.
\end{eqnarray*}
Similarly,  we have
\begin{eqnarray*}
\int_{\mho}(\nabla\cdot u) \partial_{z}Z_{1}\cdot u_{z}=-\int_{\mho}u\cdot [(\nabla \partial_{z}Z_{1} )\cdot u_{z}  ]- \int_{\mho}u\cdot [ \partial_{z}Z_{1} \cdot (\nabla u_{z})  ] .
\end{eqnarray*}
Then, due to H$\ddot{o}$lder inequality, Sobolev imbedding theorem and interpolation inequality,
\begin{eqnarray*}
\int_{\mho}(\nabla\cdot u) \partial_{z}Z_{1}\cdot u_{z}&\leq&|\int_{\mho}u\cdot [(\nabla \partial_{z}Z_{1} )\cdot u_{z}  ]|+| \int_{\mho}u\cdot [ \partial_{z}Z_{1} \cdot (\nabla u_{z})  ]|\nonumber\\
& \leq& |\nabla \partial_{z}Z_{1}|_{2}|u|_{4}|u_{z}|_{4}+|\nabla u_{z}|_{2}|\partial_{z}Z_{1}|_{4}|u|_{4}\nonumber\\
& \leq&\|Z_{1}\|_{2}(|\tilde{u} |_{4}+|\nabla \bar{u}|_{2} )|u_{z}|_{2}^{\frac{1}{4}}[|\nabla u_{z} |_{2}^{\frac{3}{4}}+|u_{zz}|_{2}^{\frac{3}{4}}
+|u_{z}|_{2}^{\frac{3}{4}} ]\nonumber\\
&&+|\nabla u_{z}|_{2}\|Z_{1}\|_{2}(|\tilde{u} |_{4}+|\nabla \bar{u}|_{2} )\nonumber\\
& \leq& \varepsilon ( |\nabla u_{z} |_{2}^{2}+|u_{zz}|_{2}^{2}  )+C|u_{z}|_{2}^{2}+C\|Z_{1}\|_{2}^{2}(|\tilde{u}|_{4}^{2}+|\nabla \bar{u} |_{2}^{2} ).
\end{eqnarray*}
By H$\ddot{o}$lder inequality, Minkowski inequality and interpolation inequality, we obtain
\begin{eqnarray*}
\int_{\mho} \varphi (u) (\partial_{z}Z_{1}) u_{zz}&\leq& |\partial_{z}Z_{1}|_{\infty} |\nabla u|_{2}|u_{zz}|_{2}\leq \varepsilon |u_{zz}|_{2}^{2}+C\|Z_{1}\|_{3}^{2}|\nabla u|_{2}^{2}.
\end{eqnarray*}
Therefore, by the above argument we reach
\begin{eqnarray}
\int_{\mho}\varphi (u) (\partial_{zz}Z_{1}) u_{z}\leq \varepsilon (|u_{zz}|_{2}^{2}+| \nabla u_{z}|_{2}^{2})+C|u_{z}|_{2}^{2}
+C\|Z_{1}\|_{2}^{2}(|\tilde{u}|_{4}^{2}+|\nabla \bar{u}|_{2}^{2} )+C\|Z_{1}\|_{3}^{2}|\nabla u|_{2}^{2}.
\end{eqnarray}
According to H$\ddot{o}$lder inequality, Sobolev imbedding theorem and interpolation inequality, we obtain
\begin{eqnarray}
\int_{\mho}\varphi (Z_{1})(u_{zz}+\partial_{zz}Z_{1})u_{z}&\leq& |u_{zz}|_{2}|u_{z}|_{4}|\varphi (Z_{1})|_{4}+|\partial_{zz}Z_{1}|_{2}|\nabla Z_{1}|_{4}|u_{z}|_{4}\nonumber\\
&\leq& C(|u_{zz}|_{2}\|Z_{1}\|_{2}+\|Z_{1}\|_{2}^{2})|u_{z}|_{2}^{\frac{1}{4}}(|u_{zz}|_{2}^{\frac{3}{4}}+|\nabla u_{z} |_{2}^{\frac{3}{4}} )\nonumber\\
&\leq& \varepsilon (|u_{zz}|_{2}^{2}+ |\nabla u_{z} |_{2}^{2})+C\|Z_{1}\|_{2}^{8}|u_{z}|_{2}^{2}+C\|Z_{1}\|_{2}^{\frac{16}{5}}|u_{z}|_{2}^{\frac{2}{5}}.
\end{eqnarray}
Taking the inner product of the equation $(5.125)$ with $u_{z}$ in $(L^{2}(\mho))^{2}$ and by $(5.126)-(5.132),$ we obtain
\begin{eqnarray}
&&\partial_{t} |u_{z}|_{2}^{2}+|\nabla u_{z} |_{2}^{2}+| u_{zz} |_{2}^{2}\nonumber\\
&\leq&
C(|\tilde{u}|_{4}^{8}+|\nabla \bar{u}|_{2}^{8}+\|Z_{1}\|_{2}^{8} +1 )|u_{z}|_{2}^{2}+C(|\tilde{u}|_{4}^{2}+|\nabla {u}|_{2}^{2} +1)\|Z_{1}\|_{3}^{2}.
\end{eqnarray}
Therefore, from Gronwall inequality,  $(5.100), (5.119)$ and $(5.124),$ we reach
\begin{eqnarray}
\sup\limits_{t\in[0,\tau*)}|u_{z}(t)|_{2}^{2}+\int_{0}^{\tau*}(|\nabla u_{z}(s) |_{2}^{2}+|u_{zz} (s)|_{2}^{2})ds
\leq C(\tau_{*}, Q, Z_{1}, Z_{2}, v_{0}, T_{0} ).
\end{eqnarray}
By   H$\ddot{o}$lder inequality, interpolation inequality and Sobolev inequality,
\begin{eqnarray}
&&\int_{\mho}\{[(u+Z_{1})\cdot \nabla ](u+Z_{1} ) \}\cdot\Delta u\nonumber\\
&\leq& |\Delta u|_{2}|\nabla u|_{4}|u+Z_{1}|_{4}+|\Delta u|_{2}|\nabla Z_{1}|_{4}|u+Z_{1}|_{4}\nonumber\\
&\leq& |\Delta u|_{2}|\nabla u|_{2}^{\frac{1}{4}}(|\Delta u|_{2}^{\frac{3}{4}}+ |\nabla u_{z}|_{4}^{\frac{3}{4}}+|\nabla u|_{2}^{\frac{3}{4}}  )
|u+Z_{1}|_{4}\nonumber\\
&&+|\Delta u|_{2}\|Z_{1}\|_{2}|u+Z_{1}|_{4}\nonumber\\
&\leq& \varepsilon (|\Delta u|_{2}^{2}+|\nabla u_{z}|_{2}^{2})+ C\| Z_{1}\|_{2}^{2}(|u+Z_{1}|_{4}^{2} ) \nonumber\\
&&+C(|u+Z_{1}|_{4}^{2} +|u+Z_{1}|_{4}^{8} )
|\nabla u|_{2}^{2}\nonumber\\
&\leq& \varepsilon (|\Delta u|_{2}^{2}+|\nabla u_{z}|_{2}^{2})+C \| Z_{1}\|_{2}^{2}(|\tilde{u}|_{4}^{2}+|\nabla \bar{u}|_{2}^{2}+\|Z_{1}\|_{1}^{2} )\nonumber\\
&&+C(|\tilde{u}|_{4}^{2}+|\nabla \bar{u}|_{2}^{2}+\|Z_{1}\|_{1}^{2} +|\tilde{u}|_{4}^{8}+ |\nabla \bar{u}|_{2}^{8}+\|Z_{1}\|_{1}^{8} )
|\nabla u|_{2}^{2}.
\end{eqnarray}
Due to H$\ddot{o}$lder inequality,
\begin{eqnarray*}
&&\int_{\mho}[\varphi(u+Z_{1})( u_{z}+\partial_{z}Z_{1}) ]\cdot \Delta u\nonumber\\
&\leq&\int_{M}\int_{-1}^{0}|\nabla \cdot u +\nabla \cdot Z_{1}|dz\int_{-1}^{0}|u_{z}+\partial_{z}Z_{1}|\cdot|\Delta u|dz
\nonumber\\
&\leq&|\Delta u|_{2}\Big{(}\int_{M}(\int_{-1}^{0}|\nabla \cdot u +\nabla \cdot Z_{1}|   dz )^{4} \Big{)}^{\frac{1}{4}}
\Big{(}\int_{M} (\int_{-1}^{0}| u_{z}+\partial_{z}Z_{1}|^{2} dz  )^{2} \Big{)}^{\frac{1}{4}}.
\end{eqnarray*}
Then by Minkowsky inequality and interpolation inequality, as well as Sobolev embedding theorem and H$\ddot{o}$lder inequality,
\begin{eqnarray}
&&\int_{\mho}[\varphi(u+Z_{1})( u_{z}+\partial_{z}Z_{1}) ]\cdot \Delta u\nonumber\\
&\leq& C|\Delta u|_{2}\Big{(}\int_{-1}^{0}[| \nabla u|_{L^{2}(M)}^{\frac{1}{2}}(| \Delta u|_{L^{2}(M)}^{\frac{1}{2}}
+ | \nabla u|_{L^{2}(M)}^{\frac{1}{2}}      )+\|Z_{1}\|_{2}]  dz    \Big{)}\nonumber\\
&&\cdot \Big{(}\int_{-1}^{0}[|u_{z}|_{L^{2}(M)}(|\nabla u_{z}|_{L^{2}(M)}+ |u_{z}|_{L^{2}(M)}  )+\|Z_{1}\|_{2}^{2}] dz  \Big{)}^{\frac{1}{2}}\nonumber\\
&\leq & \varepsilon (|\Delta u|_{2}^{2}+|\nabla u_{z}|_{2}^{2} dz   )
+C(\|Z_{1}\|_{2}^{2}+ \|Z_{1}\|_{2}^{4})  |u_{z}|_{2}^{2}+C\|Z_{1}\|_{2}^{4}\nonumber\\
&&+  C|\nabla u|_{2}^{2}(|u_{z}|_{2}^{2}+\|Z_{1}\|_{2}^{2}+|u_{z}|_{2}|\nabla u_{z}|_{2}+|u_{z}|_{2}^{4}+|u_{z}|_{2}^{2}|\nabla u_{z}|_{2}^{2}+\|Z_{1}\|_{2}^{4}).
\end{eqnarray}
Since
\begin{eqnarray*}
\int_{\mho}(fk\times u)\cdot\Delta u=0\ \ \mathrm{and}\ \ \int_{\mho}\nabla p_{s}\cdot\Delta u=0,
\end{eqnarray*}
taking the inner product of equation $(5.86)$ with $-\Delta u$ in $(L^{2}( \mho))^{2},$ by $(5.135)-(5.136)$ we reach
\begin{eqnarray}
&&\partial_{t} |\nabla u|_{2}^{2}+|\Delta u|_{2}^{2}\nonumber\\
&\leq&C(\|Z_{1}\|_{2}^{4} +\|\theta\|_{1}^{2}+\|Z_{1}\|_{2}^{2}|\tilde {u}|_{4}^{2}+\|Z_{1}\|_{2}^{2}|\nabla \bar{u}|_{2}^{2}+\|Z_{1}\|_{2}^{2} |u_{z}|_{2}^{2}+\|Z_{1}\|_{2}^{4}|u_{z}|_{2}^{2})  \nonumber\\
&&+C(\|Z_{1}\|_{2}^{2}+\|Z_{1}\|_{1}^{8}+\|Z_{1}\|_{2}^{4}+  |\tilde{u}|_{4}^{2}+|\tilde{u}|_{4}^{8}+ |\nabla \bar{u}|_{2}^{2}+ |\nabla \bar{u}|_{2}^{8} \nonumber\\
&&+ |u_{z}|_{2}^{2}+  |u_{z}|_{2}^{4}+ |\nabla u_{z}|_{2}^{2}+ |u_{z}|_{2}^{2}|\nabla u_{z}|_{2}^{2})
|\nabla u|_{2}^{2}.
\end{eqnarray}
By $(5.100),(5.119), (5.124), (5.134)$ and thanks to Gronwall inequality, we obtain
\begin{eqnarray}
\sup\limits_{t\in[0, \tau*)}|\nabla u(t)|_{2}^{2}+\int_{0}^{\tau*}|\Delta u(t)|_{2}^{2}dt
\leq C(\tau_{*}, Q, Z_{1}, Z_{2}, v_{0}, T_{0}).
\end{eqnarray}
Taking the inner product of the equation $(5.87)$ with $-\Delta \theta-\theta_{zz} $ in $L^{2}( \mho),$
and making an analogous argument in $(5.138)$ we get
\begin{eqnarray}
&&\frac{1}{2}\partial_{t}(|\nabla \theta|_{2}^{2}+|\theta_{z}|_{2}^{2}+\alpha|\nabla \theta(z=0) |_{2}^{2}  )\nonumber\\
&&+|\Delta \theta|_{2}^{2}+2( |\nabla \theta_{z} |_{2}^{2}+\alpha|\nabla \theta(z=0) |_{2}^{2} )+|\theta_{zz}|_{2}^{2}\nonumber\\
&=&\int_{\mho}[(u+Z_{1} )\cdot \nabla (\theta+Z_{2})+\varphi(u+Z_{1} )(\theta_{z}+\partial_{z}Z_{2})-Q+\beta Z_{2}]       (\Delta \theta+\theta_{zz} )\nonumber\\
&\leq& \varepsilon (| \Delta \theta|_{2}^{2}+|\theta_{zz}|_{2}^{2}+|\nabla \theta_{z}|_{2}^{2} )+C|Q|_{2}^{2}+C\|Z_{2}\|_{3}^{2}(1+\|u\|_{1}^{2}+\|Z_{1}\|_{1}^{2})\nonumber\\
&&+C(|\nabla u|_{2}^{2}+ |\nabla Z_{1}|_{2}^{2} )(|\Delta u|_{2}^{2}+ |\Delta Z_{1}|_{2}^{2} )|\theta_{z}|_{2}^{2}\nonumber\\
&&+C (|\tilde{u}|_{4}^{2}+|\nabla \bar{u}|_{2}^{2}+ \|Z_{1}\|_{1}^{2}+|\tilde{u}|_{4}^{8}+|\nabla \bar{u}|_{2}^{8}+ \|Z_{1}\|_{1}^{8})|\nabla \theta|_{2}^{2}.
\end{eqnarray}
Therefore, by  Gronwall inequality,  we conclude that
\begin{eqnarray}
\sup\limits_{t\in [0,\tau*)}\|\theta(t)\|_{1}^{2}+ \int_{0}^{\tau*}\| \theta(t)\|_{2}^{2} dt
\leq C(\tau_{*}, Q, Z_{1}, Z_{2}, T_{0}).
\end{eqnarray}
$\mathbf{Acknowledgments}.$ This work was started during the author's visit to Professor Zhao Dong in July 2014. When the
first version of this manuscript was finished, it was reported in Chinese Academy of Science in August 2015. The author is grateful for the kind invitations and the hospitality of Professor Zhao Dong. The author also thanks Professor Zhenqing Chen for interesting discussions and kind suggestions during the visit to University of Washington in 2016.  He is also
deeply grateful for Professor Boling Guo's long-time and constant help, encouragement and
kind introduction of hydrodynamics into his interest.

\def\refname{ Bibliography}

\end{document}